\documentclass[a4paper,11pt]{article}

\usepackage[margin=22mm]{geometry}
\usepackage{amsmath,amssymb,amsthm,mathtools,bm}
\usepackage{booktabs,threeparttable}
\usepackage{array,tabularx}
\usepackage{graphicx,xcolor}
\usepackage{enumitem}
\usepackage{algorithm}
\usepackage{placeins}
\usepackage{float}
\usepackage{algpseudocode}
\usepackage[numbers,sort&compress]{natbib}
\usepackage[bookmarks=true,bookmarksnumbered=true,colorlinks=true,
  linkcolor=black,citecolor=black,urlcolor=black]{hyperref}
\providecommand{\doi}[1]{\url{https://doi.org/#1}}

\newcommand{\Ssym}{\mathbb S}
\newcommand{\R}{\mathbb R}
\newcommand{\A}{\mathcal A}
\newcommand{\Ast}{\mathcal A^*}
\newcommand{\Rp}{R_{\mathrm p}}
\newcommand{\Rd}{R_{\mathrm d}}
\newcommand{\Rg}{R_{\mathrm g}}
\newcommand{\Nlong}{\mathcal N_{\mathrm{long}}}

\newcommand{\norm}[1]{\lVert #1\rVert}
\newcommand{\OFAS}{\textnormal{\textsc{OFAS}}}
\newcommand{\CAOFAS}{\textnormal{\textsc{CA-OFAS}}}
\newcommand{\Hres}{\Psi}

\theoremstyle{plain}
\newtheorem{theorem}{Theorem}[section]
\newtheorem{proposition}[theorem]{Proposition}
\newtheorem{lemma}[theorem]{Lemma}
\newtheorem{corollary}[theorem]{Corollary}
\theoremstyle{definition}
\newtheorem{assumption}[theorem]{Assumption}

\theoremstyle{remark}
\newtheorem{remark}[theorem]{Remark}

\algtext*{EndIf}
\algtext*{EndFor}
\algtext*{EndWhile}
\title{A One-Factorization Predictor--Corrector Long-Step Arc-Search Method and a Curvature-Amplified Variant for Semidefinite Programming with a Homogeneous Self-Dual Embedding}
\author{Makoto Yamashita\thanks{Corresponding author. Department of Mathematical and Computing Science, Institute of Science Tokyo, 2-12-1-W8-29 Ookayama, Meguro-ku, Tokyo 152-8550, Japan. E-mail: \texttt{Makoto.Yamashita@comp.isct.ac.jp}. ORCID: \href{https://orcid.org/0000-0002-8409-036X}{0000-0002-8409-036X}} \and Yaguang Yang\thanks{Independent researcher, Rockville, MD 20850, USA. E-mail: \texttt{yaguang.yang@verizon.net}. ORCID: \href{https://orcid.org/0000-0002-2943-9389}{0000-0002-2943-9389}}}
\date{}

\begin{document}
\maketitle

\begin{abstract}
In a predictor--corrector arc-search method, a point on the predictor arc is first selected and a corrector is then computed at that point. Since the Karush--Kuhn--Tucker (KKT) matrix changes with the selected point, this correction may require a second factorization in each iteration.
We propose a one-factorization arc-search method (\OFAS) for semidefinite programming (SDP) that avoids this second factorization, retains the corrector after the arc step, and uses only one KKT factorization per iteration.
The key is a three-term decomposition of the complementarity mismatch at the selected point on the arc, which allows the correction to be assembled from linear-system solutions using the factorization at the current iterate.
We also propose a curvature-amplified variant (\CAOFAS) that amplifies the second-order arc term.
The proposed methods are formulated with a homogeneous self-dual embedding. For a wide long-step neighborhood, we prove that both methods reduce the homogeneous complementarity and residuals below $\varepsilon$ in $O(n\log(1/\varepsilon))$ iterations for an $n$-by-$n$ SDP matrix variable.
Numerical experiments show that the one-factorization structure and the correction after the arc step are associated with reductions in computation time.
Compared with the Mehrotra-type method, \CAOFAS{} reduces the geometric-mean computation time by 17\% on SDPLIB and 13\% on neural network verification SDP problems.
\end{abstract}


\section{Introduction}
\label{sec:introduction}

Semidefinite programs (SDPs) arise in control, combinatorial optimization, statistical modeling, and convex relaxations for nonconvex problems
\citep{VandenbergheBoyd1996,Alizadeh1995}. Primal--dual interior-point methods remain useful when high-accuracy solutions are required, as illustrated by widely used implementations such as SDPA, SDPT3, and SeDuMi
\citep{YamashitaFujisawaKojima2003,TutuncuTohTodd2003,Sturm1999}. Recent developments also include regularized SDP interior-point algorithms and general conic solvers built around homogeneous embeddings
\citep{CipollaGondzio2025,GoulartChen2026}. Their cost is strongly influenced by the linear algebra for the Newton/Karush--Kuhn--Tucker (KKT) equations. For SDP in particular, a Schur-complement or KKT factorization can dominate an iteration. This motivates algorithms that control the number of matrix factorizations as well as the iteration count \citep{CipollaGondzio2023}.

Homogeneous self-dual (HSD) embeddings let primal--dual interior-point iterations start without first constructing a feasible primal--dual pair. Self-dual LP formulations go back to Goldman and Tucker \citep{GoldmanTucker1956}, and Ye--Todd--Mizuno later gave an HSD interior-point algorithm for LP \citep{YeToddMizuno1994}. The embedding combines the primal and dual systems in an augmented formulation; after normalization, a limiting solution yields either an optimizer of the original problem or an infeasibility certificate. Homogeneous embeddings also underlie modern conic solvers such as Clarabel \citep{GoulartChen2026}, where symmetric cones are handled with Nesterov--Todd (NT) scaling in a primal--dual interior-point framework \citep{NesterovTodd1997,NesterovTodd1998}.

Arc-search interior-point methods incorporate local curvature by replacing a straight predictor with an arc when tracking the central path or an analogous path for infeasible iterates. At the current iterate $W_k$, denote the first two derivatives of a local path by $\dot W_k$ and $\ddot W_k$. With $\alpha$ as the arc angle, the predictor has the familiar form
\[
 W_k(\alpha)=W_k-\dot W_k\sin\alpha+\ddot W_k(1-\cos\alpha).
\]
Arc-search methods were first developed as polynomial-time algorithms for convex quadratic programming and LP
\citep{Yang2011,Yang2013}, and were later extended to infeasible interior-point methods, symmetric-cone optimization, and SDP
\citep{YangYamashita2018,YangLiuZhang2017,KheirfamMoslemi2018,ZhangYuanZhouLuoHuang2019,Kheirfam2021,KheirfamOsmanpourKeyanpour2021}. Other studies include a computationally efficient arc-search method for infeasible LP
\citep{Yang2018Efficient}, an extension to general convex optimization
\citep{Yang2023Convex}, a method combined with Nesterov's restarting strategy
\citep{IidaYamashita2024}, an extension to nonlinear constrained problems
\citep{Yang2025Nonlinear}, and a method that solves Newton systems inexactly
\citep{IidaYamashita2026}. See also \citet{Yang2020ArcSearch} for further developments of arc-search methods. We are not aware of an earlier arc-search analysis carried out directly in an HSD formulation.

Some predictor--corrector arc-search methods first select a point on the arc and then compute a separate corrector at that point, whereas other methods do not perform this additional correction
\citep{YangYamashita2018,IidaYamashita2024,Yang2018Efficient}. In the former case, the KKT coefficient matrix changes with the scaling at the selected point and may require another factorization. In the latter case, the coefficient matrix at the current iterate can be reused. A detailed comparison of these computational structures is given in Section~\ref{subsec:existing-corrector-structure-comparison}.

Motivated by this factorization cost, we construct a one-factorization arc-search method (\OFAS). It keeps the corrector applied after the arc point is selected, but factors the KKT matrix only at the current iterate. The central device is an exact decomposition of the selected-point complementarity mismatch into three terms in the current NT coordinates. We solve for the associated correction bases with the current KKT factorization and combine them only after the candidate values of the arc angle $\alpha$ and the centering-recovery parameter $\rho$ are fixed. Candidate changes therefore do not trigger a new KKT factorization. In \OFAS{}, the single factorization serves five algorithm-specific right-hand sides. HSD and predictor--corrector methods commonly reuse a factorization across several right-hand sides, as does Clarabel \citep{YeToddMizuno1994,GoulartChen2026}; here the three-term decomposition extends that reuse to a post-arc corrector whose matrix would otherwise vary with the selected point.

We further introduce a curvature-amplified one-factorization arc-search method (\CAOFAS), which selectively amplifies the arc term corresponding to the second-order direction. This modification uses curvature information more aggressively while keeping the same one-factorization structure as \OFAS{}. The amplified candidates are constructed using the directions and correction bases already computed in \OFAS{} and require no additional matrix factorization or linear-system solve. If an amplified candidate is not accepted, the method uses an \OFAS{} candidate. Thus, \CAOFAS{} tests curvature-amplified candidates while retaining the worst-case iteration bound established for \OFAS{}.

The contributions are threefold.
\begin{enumerate}[label=(\roman*),leftmargin=2.5em,itemsep=0.15em,topsep=0.25em]
\item We develop \OFAS{} for SDP, which retains a separate corrector after selecting the point on the arc while requiring only one KKT matrix factorization per iteration.
      We also construct \CAOFAS{}, a curvature-amplified method with the same one-factorization structure.
\item Using a wide long-step neighborhood, we establish the well-definedness of \OFAS{}, finite termination of the candidate selection, and the existence of a uniform admissible window at each iteration.
      The resulting iteration bound is $O(\nu\log(1/\varepsilon))$ for reducing the homogeneous complementarity and residuals below $\varepsilon$, where $\nu=n+1$ is the rank of the HSD cone.
      Under Assumption~\ref{ass:primal-dual-solvability}, we also study the projectively normalized HSD sequence and prove that its distance to the primal--dual optimal set tends to zero.
      Since \CAOFAS{} has an \OFAS{} fallback candidate, it inherits the same iteration bound.
\item We implement \OFAS{} and \CAOFAS{} in the HSD framework of Clarabel.jl and evaluate them on SDPLIB and neural network verification SDP (NNV-SDP) problems.
      First, we compare \OFAS{} with Two-factorization, which refactorizes the KKT coefficient matrix after selecting the point on the arc, and Corrector-free, which omits the separate corrector at that point. These comparisons examine the contributions of factorization reuse and the corrector. Among matched problems reaching the high-accuracy termination criterion, \OFAS{} reduced the computation time relative to Two-factorization by about 44\% on SDPLIB and about 38\% on NNV-SDP, and relative to Corrector-free by about 37\% and 42\%, respectively. We then compare \CAOFAS{} with \OFAS{} and observe that curvature amplification reduced the computation time by about 4\% and 19\%, respectively.
      Finally, we compare the methods with a Mehrotra-type method and the Liu-type method based on Liu--Liu (2012) \citep{Mehrotra1992,LiuLiu2012}. Among matched problems reaching the high-accuracy termination criterion, \CAOFAS{} reduced the computation time relative to the Mehrotra-type method by about 17\% on SDPLIB and about 13\% on NNV-SDP, and relative to the Liu-type method by about 20\% and 12\%, respectively.
\end{enumerate}

Section~\ref{sec:framework} develops the HSD-SDP setting, NT scaling, the wide long-step neighborhood, and the arc equations used later. Section~\ref{sec:methods} then introduces \OFAS{} and \CAOFAS{}, while Section~\ref{sec:complexity} establishes candidate-search termination, the iteration bound, and projective normalization for the solvable case. Numerical results are reported in Section~\ref{sec:numerical}, followed by the conclusions in Section~\ref{sec:conclusion}.

\section{SDP with an HSD Embedding and the Foundations of Long-Step Arc-Search}
\label{sec:framework}

\subsection{SDP and the HSD Formulation}

Let $n,m\in\mathbb N$. Write $\Ssym^n$ for the space of real symmetric $n\times n$ matrices; its positive semidefinite and positive definite cones are denoted by $\Ssym_+^n$ and $\Ssym_{++}^n$. For $U,V\in\Ssym^n$, the relations $U\succeq V$ and $U\succ V$ mean $U-V\in\Ssym_+^n$ and $U-V\in\Ssym_{++}^n$, respectively.
The problem data are $C\in\Ssym^n$, $b\in\R^m$, and a linear map $\A:\Ssym^n\to\R^m$ with adjoint $\Ast$. We write $U\bullet V:=\operatorname{tr}(UV)$ for the matrix inner product and use $I$ and $O$ for the identity and zero matrices of order $n$. Unless a different norm is indicated, vector norms are Euclidean and symmetric-matrix norms are Frobenius norms.

We consider the following primal SDP in equality standard form and its dual:
\begin{align*}
 p^*&:=\min_{X\in\Ssym^n}\{C\bullet X\mid \A(X)=b,\ X\succeq O\},&
 d^*&:=\max_{y\in\R^m,\,S\in\Ssym^n}\{b^Ty\mid \Ast(y)+S=C,\ S\succeq O\}.
\end{align*}
\begin{assumption}[Surjectivity]
\label{ass:surjectivity}
The map $\A$ is surjective.
\end{assumption}

Assumption~\ref{ass:surjectivity} is equivalent to linear independence of the equality constraints. In the setting considered here, linearly dependent equality constraints can be removed in preprocessing, and we assume that this preprocessing has been applied.

We embed the original primal--dual SDP into a homogeneous self-dual (HSD) system. The basic idea of the HSD formulation goes back to \citet{YeToddMizuno1994} for linear programming. Self-dual skew-symmetric embeddings and homogeneous interior-point methods for semidefinite programming were studied by \citet{deKlerkRoosTerlaky1997,PotraSheng1998}.
When the quadratic objective term is absent, the homogeneous embedding used in Clarabel reduces to the standard homogeneous self-dual embedding \citep{GoulartChen2026}.
Let the HSD variable be $W=(X,y,S,\tau,\kappa)$. We use the following system:
\begin{equation}
\begin{gathered}
  \A(X)-b\tau=0,\qquad
  \Ast(y)+S-C\tau=0,\qquad
  C\bullet X-b^Ty+\kappa=0,\\
  X\succeq O,\qquad S\succeq O,\qquad \tau\ge0,\qquad \kappa\ge0.
\end{gathered}
\label{eq:hsd-system}
\end{equation}
Here, $\tau\in\R$ homogenizes the primal--dual variables, and $\kappa\in\R$ is an auxiliary variable associated with the primal--dual gap.
We refer to this embedded system as the HSD-SDP formulation.
The pairs $(X,\tau)$ and $(S,\kappa)$ are primal and dual cone variables over $\Ssym_+^n\times\R_+$, and their complementarity is measured by $X\bullet S+\tau\kappa$.
If $\tau>0$, the original SDP variables are recovered by projective normalization, $(\bar X,\bar y,\bar S):=(X/\tau,y/\tau,S/\tau)$. For an exact HSD solution, the normalized point satisfies the primal and dual equalities of the original SDP. At an exact HSD solution with $\tau=0$ and $\kappa>0$, we have $C\bullet X-b^Ty=-\kappa<0$, so at least one of $C\bullet X<0$ and $b^Ty>0$ holds. The former certifies dual infeasibility, while the latter certifies primal infeasibility. The present complexity analysis focuses on the solvable case with $\tau>0$; iteration bounds for certificate recovery are outside its scope.
An infeasible interior-point method need not satisfy the equality part of \eqref{eq:hsd-system} exactly at every iteration. We therefore define the residuals
$\Rp(W):=\A(X)-b\tau$, $\Rd(W):=\Ast(y)+S-C\tau$, and
$\Rg(W):=C\bullet X-b^Ty+\kappa$.
Let the residual space be $\mathcal F:=\R^m\times\Ssym^n\times\R$. We also let $0_m\in\R^m$ be the zero vector and define the zero element of $\mathcal F$ by $O_{\mathcal F}:=(0_m,O,0)$.
We collect the three residuals in the HSD residual map
\begin{equation*}
 {
 \Hres(W):=(\Rp(W),\Rd(W),\Rg(W))\in\mathcal F.}
\end{equation*}
and equip $\mathcal F$ with the norm
$\|(u,V,\chi)\|_{\mathcal F}:=(\|u\|^2+\|V\|_F^2+|\chi|^2)^{1/2}$.
The map $\Hres$ is linear, and $\Hres(W)=O_{\mathcal F}$ is equivalent to the three linear equalities in the HSD system.
{In the interior-point iterations below, we call an HSD point $W=(X,y,S,\tau,\kappa)$ \emph{strictly interior} if $X,S\succ O$ and $\tau,\kappa>0$.}

For the iteration-complexity analysis and Algorithms~\ref{alg:ofas}--\ref{alg:caofas}, we start from the standard HSD initial point
$W_0:=(I,0,I,1,1)$.
This point lies in the interior of the cone, but in general does not satisfy $\Hres(W_0)=O_{\mathcal F}$.
We call $W_0$ the \emph{standard initial point}.

The next assumption is needed only for the analysis of normalized approximate solutions of the original SDP. Define the primal, dual, and joint optimal sets by
\begin{align*}
 \mathcal P^*
 &:=\{X\succeq O:\ \A(X)=b,\ C\bullet X=p^*\},\\
 \mathcal D^*
 &:=\{(y,S):\ \Ast(y)+S=C,\ S\succeq O,\ b^Ty=d^*\},\\
 {\mathcal Z^*}
 &:=\{(X,y,S):X\in\mathcal P^*,\ (y,S)\in\mathcal D^*\}.
\end{align*}

\begin{assumption}[Solvability and bounded optimal solution sets of the primal--dual problem]
\label{ass:primal-dual-solvability}
Both problems attain finite optimal values and strong duality holds, so $d^*=p^*\in\R$.
The sets $\mathcal P^*$ and $\mathcal D^*$ are bounded.
\end{assumption}

This assumption is not used in the homogeneous long-step analysis itself. It is needed only when an HSD point is quantitatively converted into a normalized approximate solution of the original SDP.
More specifically, Assumption~\ref{ass:surjectivity} is needed for nonsingularity of the KKT operator at the current iterate, whereas Assumption~\ref{ass:primal-dual-solvability} is used in Section~\ref{subsec:normalization-solvable-case} to exclude nonzero directions that would make the primal or dual optimal solution set unbounded and to control the projective scale.
Therefore, the homogeneous iteration bounds in Theorem~\ref{thm:ofas-longstep} and Corollary~\ref{cor:caofas-complexity-inheritance} do not require Assumption~\ref{ass:primal-dual-solvability}.
For the projective normalization introduced immediately after \eqref{eq:hsd-system}, the residual identities are
$\A(\bar X)-b=\Rp(W)/\tau$ and $\Ast(\bar y)+\bar S-C=\Rd(W)/\tau$ whenever $\tau>0$.
In Section~\ref{subsec:normalization-solvable-case}, under Assumption~\ref{ass:primal-dual-solvability}, we evaluate the relative scales of $\tau_k$ and the homogeneous progress measure along the sequence generated by Algorithms~\ref{alg:ofas}--\ref{alg:caofas}.
We then show that the normalized primal residual, dual residual, complementarity, and duality gap converge to zero.

\subsection{NT Scaling and Complementarity Geometry}

Throughout this subsection, let $\R_{++}:=\{t\in\R:t>0\}$, and let
$\lambda(A):=(\lambda_1(A),\ldots,\lambda_n(A))$ denote the eigenvalues of $A\in\Ssym^n$, ordered as
$\lambda_1(A)\le\cdots\le\lambda_n(A)$.
We use the direct-product Euclidean Jordan algebra
$\mathcal E:=\Ssym^n\times\R$ for one positive semidefinite (PSD) block and the scalar HSD block.
For $U=(U_X,u)$ and $V=(V_X,v)$ in $\mathcal E$, define the Jordan product by
\begin{equation*}
 {
 U\circ V:=\left(\frac{U_XV_X+V_XU_X}{2},uv\right),
 \qquad U\circ V=V\circ U.}
\end{equation*}
The identity and zero elements of $\mathcal E$ are $\bar I:=(I,1)$ and $\bar O:=(O,0)$, respectively, and the inner product is
$\langle U,V\rangle_{\mathcal E}:=U_X\bullet V_X+uv$.
The induced norm is written as
$\|U\|_{\mathcal E}:=\sqrt{\langle U,U\rangle_{\mathcal E}}$.
When the arguments clearly belong to $\mathcal E$, we omit the subscript $\mathcal E$ and write
$\langle\cdot,\cdot\rangle$ and $\|\cdot\|$.
We also define $\operatorname{Tr}_{\mathcal E}(U):=\langle\bar I,U\rangle_{\mathcal E}$ and
$\lambda_{\min}(U):=\min\{\lambda_1(U_X),u\}$.
For simplicity, we describe the method for a single PSD block.
The same construction applies to multiple PSD blocks over a direct-product cone, as briefly discussed in Remark~\ref{rem:multiple-psd-block-extension}.

Since $\Ssym_+^n\times\R_+$ is a symmetric cone, NT scaling applies to this product cone
\citep{NesterovTodd1997,NesterovTodd1998}.
We set $\nu:=n+1$ and define the average complementarity by
$\mu(W):=(X\bullet S+\tau\kappa)/\nu$.
For an HSD interior point $W$ with $X,S\succ O$, $\tau>0$, and $\kappa>0$, define
$\lambda(X,S):=\lambda(X^{1/2}SX^{1/2})$ and
$h(W):=(\lambda(X,S),\tau\kappa)\in\R_{++}^{\nu}$, and denote the $i$th component of $h(W)$ by $h_i(W)$.
The componentwise average of $h(W)$ is $\mu(W)$.

For the homogeneous analysis, we use the progress measure
\begin{equation*}
 \mathcal Q(W)
 :=\max\{\mu(W),\norm{\Rp(W)},\norm{\Rd(W)},|\Rg(W)|\}
\end{equation*}
which combines the average complementarity and the HSD residuals.
When $\tau>0$, we instead measure the KKT accuracy for the original SDP by
\begin{equation*}
 \mathcal Q_{\rm SDP}(W)
 :=\max\left\{
 \frac{\norm{\Rp(W)}}{\tau},
 \frac{\norm{\Rd(W)}}{\tau},
 \frac{X\bullet S}{\tau^2},
 \frac{|C\bullet X-b^Ty|}{\tau}
 \right\}.
\end{equation*}
This quantity measures the primal residual, dual residual, complementarity, and duality gap of the normalized point $(X/\tau,y/\tau,S/\tau)$.
To recover a solution of the original SDP, we use the stopping condition
$\mathcal Q_{\rm SDP}(W)\le\delta$ for a tolerance $\delta>0$.
The quantity $\mathcal Q(W)$ is used for the homogeneous progress analysis in Section~\ref{sec:complexity}.
Section~\ref{subsec:normalization-solvable-case} shows that the normalized stopping condition is satisfied after finitely many iterations.

For the PSD cone, we use the same NT scaling as in Clarabel.jl \citep{GoulartChen2026}.
For an HSD interior point
$\widehat W=(\widehat X,\widehat y,\widehat S,\widehat\tau,\widehat\kappa)$,
let $M(\widehat W)$ denote the inverse scaling matrix maintained at that point.
{
For the PSD block, an explicit symmetric NT choice is
\[
 M(\widehat W)
 =\left[
 \widehat X^{-1/2}
 \bigl(\widehat X^{1/2}\widehat S\widehat X^{1/2}\bigr)^{1/2}
 \widehat X^{-1/2}
 \right]^{1/2},
\]
where all matrix square roots are the principal positive definite square roots.
Indeed, with $G:=M(\widehat W)^TM(\widehat W)$, this choice gives
$G\widehat XG=\widehat S$. Multiplying this identity by $M(\widehat W)^{-T}$ from the left and by $M(\widehat W)^{-1}$ from the right gives the equality of the primal and dual variables in the scaled coordinates in~\eqref{eq:nt-scaling-balance}.}
Thus,
\begin{equation}
 M(\widehat W)\widehat X M(\widehat W)^T
 =M(\widehat W)^{-T}\widehat S M(\widehat W)^{-1}
 =:Q_X^{\rm NT}(\widehat W)\succ O
 \label{eq:nt-scaling-balance}
\end{equation}
holds.
We also define
$\xi(\widehat W):=\sqrt{\widehat\kappa/\widehat\tau}$.
For a generic HSD vector $V=(V_X,v_y,V_S,v_\tau,v_\kappa)$, define its NT coordinates with respect to the reference interior point $\widehat W$ by
\begin{equation}
 \begin{aligned}
 P(V;\widehat W)&:=\bigl(M(\widehat W)V_XM(\widehat W)^T,\ \xi(\widehat W)v_\tau\bigr),\\
 R(V;\widehat W)&:=\bigl(M(\widehat W)^{-T}V_SM(\widehat W)^{-1},\ \xi(\widehat W)^{-1}v_\kappa\bigr).
 \end{aligned}
 \label{eq:fixed-nt-coordinate-map}
\end{equation}
In particular,
\begin{equation*}
 {
 Q^{\rm NT}(\widehat W)
 :=P(\widehat W;\widehat W)=R(\widehat W;\widehat W)
 =\bigl(Q_X^{\rm NT}(\widehat W),\sqrt{\widehat\tau\widehat\kappa}\bigr).}
\end{equation*}
We represent complementarity in the NT coordinates of the same reference point by the quadratic map
\begin{equation*}
 {
 \Phi(V;\widehat W):=P(V;\widehat W)\circ R(V;\widehat W).}
\end{equation*}
At the current iterate $W_k$, we use the abbreviations
\[
 M_k:=M(W_k),\qquad \xi_k:=\xi(W_k),\qquad
 Q_k^{\rm NT}:=Q^{\rm NT}(W_k).
\]
Candidate points are evaluated in the NT coordinates determined by $W_k$.

The standard initial point $W_0$ satisfies $\mu(W_0)=1$ and $\Phi(W_0;W_0)=\bar I$.
Let $\mathcal K:=\Ssym_+^n\times\R_+$, and denote the central-path parameter by $\mu_{\rm c}$.
By the standard central-path theory for HSD-SDP \citep{PotraSheng1998,NesterovTodd1997,NesterovTodd1998}, under Assumption~\ref{ass:surjectivity}, for every $0<\mu_{\rm c}\le1$ there exists a unique interior point $W^{\rm c}(\mu_{\rm c})$ satisfying
\begin{equation*}
 {
 \Hres(W^{\rm c}(\mu_{\rm c}))=\mu_{\rm c}\,\Hres(W_0),\qquad
 \Phi(W^{\rm c}(\mu_{\rm c});W^{\rm c}(\mu_{\rm c}))=\mu_{\rm c}\bar I.}
\end{equation*}
The second equation is the standard centrality condition based on the Jordan product on $\mathcal K$.
For the PSD block, it is equivalent to $S=\mu_{\rm c}X^{-1}$, and for the scalar HSD block, to $\kappa=\mu_{\rm c}/\tau$.
We call $\{W^{\rm c}(\mu_{\rm c}):0<\mu_{\rm c}\le1\}$ the HSD central path.
Along this path, $\mu(W^{\rm c}(\mu_{\rm c}))=\mu_{\rm c}$.
We introduce the central path only to describe the geometry behind the local tracking curve.
Its use here does not impose a solvability assumption for the original SDP on the homogeneous iteration bound in Section~\ref{sec:complexity}.
Projective normalization for the original SDP is treated in Section~\ref{subsec:normalization-solvable-case}.

Since $P(\cdot;\widehat W)$ and $R(\cdot;\widehat W)$ are linear,
$\Phi(\cdot;\widehat W)$ is quadratic.
For an HSD interior point $\widehat W$ and HSD variables or directions $V^{(1)},V^{(2)}$, define the associated cross term by
\begin{equation*}
 \mathcal B(V^{(1)},V^{(2)};\widehat W)
 :=P(V^{(1)};\widehat W)\circ R(V^{(2)};\widehat W)
  +P(V^{(2)};\widehat W)\circ R(V^{(1)};\widehat W).
\end{equation*}
The map $\mathcal B(\cdot,\cdot;\widehat W)$ is symmetric and bilinear, and
\[
\Phi(V^{(1)}+V^{(2)};\widehat W)
=\Phi(V^{(1)};\widehat W)
+\mathcal B(V^{(1)},V^{(2)};\widehat W)
+\Phi(V^{(2)};\widehat W).
\]
For an HSD direction $D$, define the linearized HSD KKT operator at the reference interior point $\widehat W$ by
\begin{equation}
 \mathcal L(\widehat W)[D]
 :=\bigl(\Hres(D),\mathcal B(\widehat W,D;\widehat W)\bigr).
 \label{eq:general-kkt-operator-definition}
\end{equation}

\subsection{Current-Iterate KKT System, Predictor Arc, and Wide Long-Step Neighborhood}

A general infeasible interior point $W_k$ need not lie on the HSD central path.
Following the infeasible arc-search literature, we introduce a local path on which the HSD residual map and the NT-coordinate complementarity are multiplied by the same scalar parameter \citep{YangYamashita2018,IidaYamashita2024}.
Its first two derivatives determine the predictor arc.

Let an HSD direction be $D=(D_X,d_y,D_S,d_\tau,d_\kappa)$.
At the current iterate $W_k$, \eqref{eq:fixed-nt-coordinate-map} becomes
\begin{equation}
 \begin{aligned}
 P(D;W_k)&=\bigl(M_kD_XM_k^T,\ \xi_k d_\tau\bigr),\\
 R(D;W_k)&=\bigl(M_k^{-T}D_SM_k^{-1},\ \xi_k^{-1}d_\kappa\bigr).
 \end{aligned}
 \label{eq:direction-nt-components}
\end{equation}
Since $P(W_k;W_k)=R(W_k;W_k)=Q_k^{\rm NT}$ at the current iterate,
$\mathcal B(W_k,D;W_k)=Q_k^{\rm NT}\circ\{P(D;W_k)+R(D;W_k)\}$.
Therefore, the specialization of \eqref{eq:general-kkt-operator-definition} to $\widehat W=W_k$ is
\begin{equation}
 \mathcal L(W_k)[D]
 =
 \left(
 \begin{array}{c}
  \A(D_X)-bd_\tau\\
  \Ast(d_y)+D_S-Cd_\tau\\
  C\bullet D_X-b^Td_y+d_\kappa\\
  Q_k^{\rm NT}\circ\{P(D;W_k)+R(D;W_k)\}
 \end{array}
 \right)
 \in\mathcal F\times\mathcal E.
 \label{eq:current-kkt-operator-definition}
\end{equation}
Thus, \eqref{eq:current-kkt-operator-definition} gives the matrix representation of \eqref{eq:general-kkt-operator-definition} at the current iterate.
By Lemma~\ref{lem:current-kkt-nonsingular-longstep}, $\mathcal L(W_k)$ is nonsingular under Assumption~\ref{ass:surjectivity} and strict interiority of $W_k$.
In this paper, ``one factorization'' means that this matrix is factorized only once in each interior-point iteration and that the resulting factorization is reused for multiple right-hand sides.

Using the NT scaling determined by $W_k$, consider a local curve $\widehat W_k(t)$ near $t=1$ satisfying
\begin{equation}
 \widehat W_k(1)=W_k,\qquad
 \Hres(\widehat W_k(t))=t\,\Hres(W_k),\qquad
 \Phi(\widehat W_k(t);W_k)=t\,\Phi(W_k;W_k).
 \label{eq:local-tracking-curve}
\end{equation}
The Jacobian with respect to $W$ in \eqref{eq:local-tracking-curve} at $(W,t)=(W_k,1)$ is $\mathcal L(W_k)$.
Since this operator is nonsingular, the implicit function theorem gives a unique smooth curve locally around $W_k$.
This curve is distinct from the HSD central path. Along it, the parameter $t$ scales both the residual map and the NT-coordinate complementarity from their values at the current iterate.

Let $\dot W_k$ and $\ddot W_k$ be the first and second derivatives at $t=1$, respectively.
We call them the first-order and second-order directions.
Differentiating \eqref{eq:local-tracking-curve} once and using the linearity of $\Hres$ and the quadratic structure of $\Phi$ gives
\begin{equation}
 \mathcal L(W_k)[\dot W_k]
 =\bigl(\Hres(W_k),\ \Phi(W_k;W_k)\bigr).
 \label{eq:first-derivative-full-kkt-paper}
\end{equation}
Differentiating twice, the right-hand side has zero second derivative, and only the quadratic term of $\Phi$ remains. Thus,
\begin{equation}
 \mathcal L(W_k)[\ddot W_k]
 =\bigl(O_{\mathcal F},\ -2\Phi(\dot W_k;W_k)\bigr).
 \label{eq:second-order-full-kkt-paper}
\end{equation}
Hence, $\dot W_k$ and $\ddot W_k$ are obtained from two right-hand sides with the same KKT coefficient matrix.

Using these directions, we define the predictor arc by
\begin{equation}
 W_k(\alpha)
 =W_k-\dot W_k\sin\alpha+\ddot W_k(1-\cos\alpha).
 \label{eq:base-arc-paper}
\end{equation}
We have $W_k(0)=W_k$, the tangent direction at $\alpha=0$ is $-\dot W_k$, and the second derivative is $\ddot W_k$.
The minus sign reflects that the arc follows the local tracking curve as its parameter $t$ decreases from $1$.
We write $W_k(\alpha_k)$ for the point corresponding to an accepted angle $\alpha_k$.

For the long-step analysis, we use the wide long-step neighborhood
\begin{equation*}
 \Nlong(\gamma)
 :=\left\{
 W=(X,y,S,\tau,\kappa)\ \middle|\
 \begin{array}{l}
 X\succ O,\ S\succ O,\ \tau>0,\ \kappa>0,\\
 \min_i h_i(W)\ge\gamma\mu(W)
 \end{array}
 \right\},\qquad 0<\gamma<1,
\end{equation*}
which is independent of the particular scaling representation.
This complementarity neighborhood does not require HSD feasibility and therefore contains infeasible interior points.
We test candidate points in the NT coordinates of the current iterate $W_k$.
At the current iterate,
$\lambda_{\min}(\Phi(W_k;W_k))=\min_i h_i(W_k)$.
For a general candidate $W$, Lemma~\ref{lem:current-frame-certificate} shows that
$\lambda_{\min}(\Phi(W;W_k))\ge\gamma\mu(W)$ implies $W\in\Nlong(\gamma)$.
Thus, the candidate search can be performed in the current iterate's NT coordinates, while the analysis uses a scaling-independent neighborhood.

\subsection{Comparison of Corrector Structures in Existing Arc-Search Methods}
\label{subsec:existing-corrector-structure-comparison}

To compare the computational structure of OFAS with existing methods, we first consider methods that rebuild the NT scaling at the accepted point on the predictor arc and compute the corrector there.
In such methods, restoring centrality at the accepted point changes the coefficient matrix.
In Algorithm~3.1 of \citet{YangYamashita2018} and Algorithm~1 of \citet{IidaYamashita2024}, the first-order and second-order directions are computed at the current iterate, whereas the corrector system uses the selected point on the arc in its coefficient matrix.
Below, we give an HSD-SDP reference construction in NT coordinates that captures only this computational structure.

Assume that the accepted point $W_k(\alpha_k)$ is strictly interior.
Following the centering target used in these LP methods, set
$\mu_k^{\rm ref}:=(1-\sin\alpha_k)\mu(W_k)$ and $T_k^{\rm ref}:=\mu_k^{\rm ref}\bar I$.
Define the centering deviation in the NT coordinates of the accepted point by
$\Delta_k^{\rm ref}:=\Phi(W_k(\alpha_k);W_k(\alpha_k))-T_k^{\rm ref}$.
A corrector $D_k^{\rm cor,ref}$ that linearly removes this mismatch without changing the HSD residuals is obtained from
\begin{equation*}
 \Hres(D_k^{\rm cor,ref})=O_{\mathcal F},\qquad
 \mathcal B(W_k(\alpha_k),D_k^{\rm cor,ref};W_k(\alpha_k))=-\Delta_k^{\rm ref}.
\end{equation*}
This corrector uses $\mathcal L(W_k(\alpha_k))$, namely the KKT operator formed at the accepted point; in general
$\mathcal L(W_k(\alpha_k))\neq\mathcal L(W_k)$.
Thus, with a direct linear solver, constructing the predictor arc requires one factorization of $\mathcal L(W_k)$, whereas the accepted-point corrector requires another factorization of $\mathcal L(W_k(\alpha_k))$.

Other methods incorporate centering information directly into the arc construction and do not solve a separate corrector system at the accepted point.
For example, \citet{Yang2018Efficient} uses the same current-iterate coefficient matrix for the first-order and second-order directions, adds a centering term to the right-hand side for the second-order direction, and directly updates to a point on the resulting arc.
Such a method can reuse the factorization at the current iterate for multiple right-hand sides, but it does not apply a separate post-arc corrector based on the complementarity mismatch at the accepted point.
Reusing one coefficient-matrix factorization for several right-hand sides is standard.
The key distinction is that, when a method uses an accepted-point corrector, updating the scaling at the accepted point changes the coefficient matrix.

\section{Predictor--Corrector Arc-Search with Current-Iterate NT Scaling and Curvature Amplification}
\label{sec:methods}

In this section, we construct \OFAS{} and \CAOFAS{} in the current-iterate NT coordinates introduced in Section~\ref{sec:framework}.
We use $\Phi$, $\mathcal B$, $\mathcal L(W_k)$, the first-order direction $\dot W_k$, the second-order direction $\ddot W_k$, and the predictor arc $W_k(\alpha)$ defined there.
For a general candidate angle $\alpha$, we write $\sigma:=\sin\alpha$ and $\psi:=1-\cos\alpha$ throughout this section.

The main idea of \OFAS{} is to retain an explicit post-arc correction while expressing its candidate-dependent complementarity mismatch as an exact sum of three components that depend only on the current iterate $W_k$.
We precompute the correction response to each component using the same KKT factorization at $W_k$.
The dependence on the candidate parameters then appears only through scalar coefficients, so the correction can be reconstructed without refactorizing the KKT matrix at the accepted point.
We first derive this three-term decomposition and the corresponding one-factorization correction basis.
We then describe the corrected candidates, candidate selection, \OFAS{}, and \CAOFAS{}, which uses the same basis.

\subsection{Three-Term Decomposition of the Complementarity Mismatch on the Predictor Arc}
\label{subsec:corrector-basis-proof}

Fix iteration $k$ and define
\(U_k(\alpha):=W_k(\alpha)-W_k=\sigma(-\dot W_k)+\psi\ddot W_k\).
The complementarity rows of \eqref{eq:first-derivative-full-kkt-paper} and \eqref{eq:second-order-full-kkt-paper}, together with the predictor arc \eqref{eq:base-arc-paper}, give
\begin{equation}
 \mathcal B(W_k,-\dot W_k;W_k)=-\Phi(W_k;W_k),
 \qquad
 \mathcal B(W_k,\ddot W_k;W_k)=-2\Phi(-\dot W_k;W_k).
 \label{eq:pred-second-current-identities}
\end{equation}
Together with the quadratic structure of $\Phi(\cdot;W_k)$, these identities give the following exact decomposition.

\begin{lemma}[Exact quadratic decomposition of complementarity on the predictor arc]
\label{lem:current-arc-quadratic-decomposition}
For any $\alpha$,
\begin{equation}
\begin{aligned}
 \Phi(W_k(\alpha);W_k)
 &=(1-\sigma)\Phi(W_k;W_k)
 +\sigma\psi\,\mathcal B(-\dot W_k,\ddot W_k;W_k)\\
 &\quad+\psi^2\{\Phi(\ddot W_k;W_k)-\Phi(-\dot W_k;W_k)\}.
\end{aligned}
 \label{eq:current-arc-product-exact}
\end{equation}
\end{lemma}

\begin{proof}
A quadratic expansion of $\Phi(W_k+U_k(\alpha);W_k)$, together with \eqref{eq:pred-second-current-identities}, gives
\[
 \Phi(W_k(\alpha);W_k)
 =(1-\sigma)\Phi(W_k;W_k)
 +(\sigma^2-2\psi)\Phi(-\dot W_k;W_k)
 +\sigma\psi\mathcal B(-\dot W_k,\ddot W_k;W_k)
 +\psi^2\Phi(\ddot W_k;W_k).
\]
Using $\sigma^2=2\psi-\psi^2$ gives \eqref{eq:current-arc-product-exact}.
\end{proof}

For a centering-recovery parameter $\rho\in[0,1]$, define the corrector complementarity target by
\begin{equation}
 T_k(\alpha,\rho)
 :=(1-\sigma)\{(1-\rho)\Phi(W_k;W_k)+\rho\mu(W_k)\bar I\}
 \label{eq:corrector-target-complementarity}
\end{equation}
and define the complementarity mismatch by
\begin{equation}
 \Delta_k^{\rm comp}(\alpha,\rho)
 :=\Phi(W_k(\alpha);W_k)-T_k(\alpha,\rho).
 \label{eq:corrector-residual-definition}
\end{equation}
Here, $\rho$ specifies the fraction of the current centrality deviation to be corrected, while $1-\sigma$ gives the complementarity reduction produced by the predictor arc.

\begin{proposition}[Three-term decomposition of the complementarity mismatch]
\label{prop:corrector-rhs-exact-basis}
For any $\alpha$ and $\rho\in[0,1]$,
\begin{equation}
 \Delta_k^{\rm comp}(\alpha,\rho)
 =\rho(1-\sigma)\Delta_k^{\rm cen}
  +\sigma\psi\Delta_k^{\rm mix}
  +\psi^2\Delta_k^{\rm rem},
 \label{eq:corrector-rhs-basis-paper}
\end{equation}
where
\begin{equation}
 \begin{aligned}
 \Delta_k^{\rm cen}&:=\Phi(W_k;W_k)-\mu(W_k)\bar I,\\
 \Delta_k^{\rm mix}&:=\mathcal B(-\dot W_k,\ddot W_k;W_k),\\
 \Delta_k^{\rm rem}&:=\Phi(\ddot W_k;W_k)-\Phi(-\dot W_k;W_k).
 \end{aligned}
 \label{eq:corrector-basis-rhs-definitions}
\end{equation}
\end{proposition}

\begin{proof}
Rearranging \eqref{eq:corrector-target-complementarity} gives
\(T_k(\alpha,\rho)=(1-\sigma)\Phi(W_k;W_k)
-\rho(1-\sigma)\{\Phi(W_k;W_k)-\mu(W_k)\bar I\}\).
Substituting this expression into \eqref{eq:corrector-residual-definition} and applying Lemma~\ref{lem:current-arc-quadratic-decomposition} gives \eqref{eq:corrector-rhs-basis-paper}.
\end{proof}

The three components in \eqref{eq:corrector-rhs-basis-paper} correspond to the centrality deviation at the current iterate, the cross term between the first- and second-order directions, and the quadratic remainder, respectively.
Thus, once $W_k$, $\dot W_k$, and $\ddot W_k$ are fixed, all three components are fixed.
The candidate $(\alpha,\rho)$ enters only through the scalar coefficients $\rho(1-\sigma)$, $\sigma\psi$, and $\psi^2$.

\subsection{One-Factorization Correction Basis and Correction-Response Map}
\label{subsec:one-factorization-basis}

\begin{lemma}[Unique solvability of the NT KKT operator]
\label{lem:current-kkt-nonsingular-longstep}
Under Assumption~\ref{ass:surjectivity}, the NT KKT operator $\mathcal L(\widehat W)$ is nonsingular at any HSD interior point $\widehat W$.
\end{lemma}

A proof is given in Appendix~\ref{app:kkt-nonsingularity}.

For an HSD interior point $\widehat W$ and $\Delta\in\mathcal E$, define the correction-response map $\mathcal G(\Delta;\widehat W)$ by
\begin{equation}
 \mathcal L(\widehat W)[\mathcal G(\Delta;\widehat W)]
 =\bigl(O_{\mathcal F},-\Delta\bigr).
 \label{eq:current-corrector-solution-map}
\end{equation}
By Lemma~\ref{lem:current-kkt-nonsingular-longstep}, this map is well defined and $\mathcal G(\cdot;\widehat W)$ is linear.
{If $D=\mathcal G(\Delta;\widehat W)$, then
$\Hres(D)=O_{\mathcal F}$ and
$\mathcal B(\widehat W,D;\widehat W)=-\Delta$.}
Below, we use it with $\widehat W=W_k$.

Define the unit-complementarity direction by
\begin{equation}
 \mathcal L(W_k)[D_k^{\rm unit}]=\bigl(O_{\mathcal F},\bar I\bigr),
 \label{eq:unit-complementarity-direction}
\end{equation}
and define
$D_k^{\rm mix}:=\mathcal G(\Delta_k^{\rm mix};W_k)$,
$D_k^{\rm rem}:=\mathcal G(\Delta_k^{\rm rem};W_k)$, and
$D_k^{\rm cen}:=\mathcal G(\Delta_k^{\rm cen};W_k)$.
Each of $D_k^{\rm unit}$, $D_k^{\rm cen}$, $D_k^{\rm mix}$, and $D_k^{\rm rem}$ is an HSD residual-invariant direction; equivalently, $\Hres(D_k^{\rm term})=O_{\mathcal F}$ for ${\rm term}\in\{{\rm unit},{\rm cen},{\rm mix},{\rm rem}\}$.
For $D_k^{\rm mix}$ and $D_k^{\rm rem}$, the right-hand sides in \eqref{eq:current-corrector-solution-map} can be formed directly from $\dot W_k$ and $\ddot W_k$ using \eqref{eq:corrector-basis-rhs-definitions}.
The centering direction $D_k^{\rm cen}$ can instead be recovered algebraically by the following lemma, so it requires no additional linear-system solve.

\begin{lemma}[Algebraic reconstruction of the centering correction]
\label{lem:center-direction-solve-free}
\begin{equation}
 D_k^{\rm cen}
 =-W_k+\dot W_k+\mu(W_k)D_k^{\rm unit}.
 \label{eq:center-identity-paper}
\end{equation}
\end{lemma}

\begin{proof}
Using \eqref{eq:first-derivative-full-kkt-paper} and
$\mathcal L(W_k)[W_k]=(\Hres(W_k),2\Phi(W_k;W_k))$, we obtain
$\mathcal G(\Phi(W_k;W_k);W_k)=-W_k+\dot W_k$.
Equation~\eqref{eq:unit-complementarity-direction} also gives
$D_k^{\rm unit}=\mathcal G(-\bar I;W_k)$.
The identity then follows from the linearity of $\mathcal G$ and
$\Delta_k^{\rm cen}=\Phi(W_k;W_k)-\mu(W_k)\bar I$.
\end{proof}

For any $(\alpha,\rho)$, define
\begin{equation*}
 D_k^{\rm cor}(\alpha,\rho)
 :=\rho(1-\sigma)D_k^{\rm cen}
  +\sigma\psi D_k^{\rm mix}
  +\psi^2D_k^{\rm rem}.
\end{equation*}
By Proposition~\ref{prop:corrector-rhs-exact-basis} and the linearity of $\mathcal G$,
\begin{equation}
 \mathcal L(W_k)[D_k^{\rm cor}(\alpha,\rho)]
 =\bigl(O_{\mathcal F},-\Delta_k^{\rm comp}(\alpha,\rho)\bigr).
 \label{eq:corrector-direct-kkt-system}
\end{equation}
Thus, $D_k^{\rm cor}(\alpha,\rho)$ is also an HSD residual-invariant direction.

Once the KKT matrix has been factorized, solving systems with additional right-hand sides by reusing the existing factors is much less expensive than performing another factorization.
At each iteration, we factorize the matrix representation of $\mathcal L(W_k)$ only once and reuse its factors for the five algorithm-specific right-hand sides associated with
$\dot W_k$, $\ddot W_k$, $D_k^{\rm unit}$, $D_k^{\rm mix}$, and $D_k^{\rm rem}$.
The direction $D_k^{\rm cen}$ is reconstructed from \eqref{eq:center-identity-paper}, and the amplification-specific direction used by CA-OFAS below can also be reconstructed algebraically from $\ddot W_k$.
Therefore, changing the candidate parameters or testing an amplification candidate requires neither an additional KKT factorization nor an additional algorithm-specific linear-system solve.
Evaluating a candidate then requires only linear combinations of precomputed directions, computation of auxiliary scalars, and evaluation of the acceptance conditions.

\begin{remark}[Relation to the reduced KKT system in Clarabel.jl]
\label{rem:clarabel-reduced-kkt-correspondence}
The direct solver in Clarabel.jl does not directly factorize the full matrix representation of \eqref{eq:current-kkt-operator-definition}.
Instead, it eliminates the cone-slack directions and the HSD scalar coupling, forms a reduced symmetric KKT system, and applies a direct LDL factorization to the reduced matrix
\citep[Section~3]{GoulartChen2026}.
Once the NT scaling at the current iterate is fixed, the same reduced KKT matrix is used for
$\dot W_k$, $\ddot W_k$, $D_k^{\rm unit}$, $D_k^{\rm mix}$, and $D_k^{\rm rem}$; only the right-hand side changes.
Thus, the abstract one-factorization construction in \eqref{eq:current-kkt-operator-definition} can be implemented in Clarabel.jl by reusing the factors of one reduced KKT matrix for multiple right-hand sides.
\end{remark}

\subsection{Corrected Candidates and Candidate Selection with Theoretical Guarantees}
\label{subsec:ofas-mathematical-search}

We add the correction
$D=D_k^{\rm cor}(\alpha,\rho)$ in \eqref{eq:corrector-direct-kkt-system} to the point $W_k(\alpha)$ on the arc.
Since the complementarity row gives
$\mathcal B(W_k,D;W_k)=-\Delta_k^{\rm comp}(\alpha,\rho)$,
a quadratic expansion of $\Phi$ with the above $U_k(\alpha)$ gives
\begin{equation}
 \Phi(W_k(\alpha)+D;W_k)-T_k(\alpha,\rho)
 =\underbrace{\mathcal B(U_k(\alpha),D;W_k)}_{\text{arc--correction cross term}}
 +\underbrace{\Phi(D;W_k)}_{\text{quadratic correction term}}.
 \label{eq:corrector-exact-remaining-mismatch}
\end{equation}
Thus, the current-iterate KKT correction exactly removes the complementarity mismatch linearized at the current iterate.
The only remaining terms in \eqref{eq:corrector-exact-remaining-mismatch} are the cross term between the arc increment and the correction and the quadratic term of the correction.

For a candidate angle $\alpha$, define the target average complementarity by
\(\mu_{\rm tar}:=(1-\sigma)\mu(W_k)\).
For fixed $(\alpha,\rho)$, if
\begin{equation}
 \mu\!\left(
 W_k(\alpha)+D_k^{\rm cor}(\alpha,\rho)+\zeta D_k^{\rm unit}
 \right)=\mu_{\rm tar}
 \label{eq:zeta-defining-equation-method}
\end{equation}
has a unique solution, denote it by $\zeta_k(\alpha,\rho)$.
The scalar $\zeta_k$ is not an independent candidate parameter; it is an auxiliary scalar for adjusting the average complementarity.
Both $D_k^{\rm cor}$ and $D_k^{\rm unit}$ are HSD residual-invariant directions.
Lemma~\ref{lem:ofas-exact-update-identities}(ii) directly shows the orthogonality of residual-zero directions.
Hence, both the quadratic complementarity term of each residual-invariant direction and the cross complementarity term between two such directions vanish.
Therefore, the average complementarity $\mu$ is exactly linear in $\zeta$.
With $V_0:=W_k(\alpha)+D_k^{\rm cor}(\alpha,\rho)$, we can then compute
\begin{equation}
 \zeta_k(\alpha,\rho)
 =\frac{\mu_{\rm tar}-\mu(V_0)}
 {\mu(V_0+D_k^{\rm unit})-\mu(V_0)}
 \label{eq:zeta-direct-evaluation-method}
\end{equation}
directly when the denominator is nonzero.
Section~\ref{sec:complexity} verifies this linearity and uniqueness.
We further define
\(D_k^{\rm full}(\alpha,\rho)
 :=D_k^{\rm cor}(\alpha,\rho)+\zeta_k(\alpha,\rho)D_k^{\rm unit}\)
and use
\begin{equation}
 W^+(\alpha,\rho,\beta)
 :=W_k(\alpha)+\beta D_k^{\rm full}(\alpha,\rho),
 \qquad 0<\beta\le1
 \label{eq:corrected-point-paper}
\end{equation}
as the corrected candidate.
The choice $\beta=1$ gives the full correction, while $0<\beta<1$ gives a damped correction.
Thus, $\beta$ controls only the amount of correction.

Candidate selection uses only quantities that can be evaluated directly from an actual candidate point.
First, set $\omega_0:=1$ and $\omega_{k+1}:=(1-\sin\alpha_k)\omega_k$.
For an HSD interior point $W$ and $\omega>0$, define
$\theta(W,\omega):=\mu(W)/\omega$ and
\(\mathcal J(W,\omega):=\max\{\omega,\omega/\theta(W,\omega),
\omega/\theta(W,\omega)^2\}\).
The quantity $\mathcal J$ is computed directly from the candidate point and the known value of $\omega$.
Section~\ref{sec:complexity} also uses it to control the scale of the projective normalization.
For a candidate angle with $\sigma=\sin\alpha$, set $\omega^+:=(1-\sigma)\omega_k$.

Fix $\sigma_{\max}\in(0,1)$, and let
$\sigma^{\rm pre}\in(0,\sigma_{\max}]$ be the minimum value of $\sigma=\sin\alpha$ allowed for preliminary candidates.
The threshold $\sigma^{\rm pre}$ is a fixed positive constant, independent of $\nu$, used to separate preliminary candidates from the successive-halving fallback search. The analysis requires only that this threshold be positive, not any particular numerical value.
We also set the centering-recovery parameter used in the successive-halving fallback search to $\rho_0:=1/8$, and define
$\mathcal P_\rho:=\{1,1/2,1/4,\rho_0\}$ and
$\mathcal W_{\rm HSD}:=\Ssym_+^n\times\R^m\times\Ssym_+^n\times\R_+\times\R_+$.
A candidate triple $(\alpha,\rho,\beta)$ is called \emph{OFAS-admissible} if all of the following four conditions hold.
\begin{enumerate}[label=(C\arabic*)]
 \item \textbf{Well-definedness:}
 $\mu_{\rm tar}>0$, and \eqref{eq:zeta-defining-equation-method} has a unique solution $\zeta_k(\alpha,\rho)$.
 \item \textbf{Interiority:}
 $W^+(\alpha,\rho,\beta)\in\operatorname{int}\mathcal W_{\rm HSD}$ and $\mu(W^+)>0$.
 \item \textbf{Neighborhood condition:}
 \(\lambda_{\min}(\Phi(W^+;W_k))\ge\gamma\mu(W^+)\).
 \item \textbf{Local progress:}
 \begin{equation*}
  \mathcal Q(W^+)\le\left(1-\frac{\sigma}{4}\right)\mathcal Q(W_k),
  \qquad
  \mathcal J(W^+,\omega^+)\le
  \left(1-\frac{\sigma}{4}\right)\mathcal J(W_k,\omega_k).
 \end{equation*}
\end{enumerate}
Condition (C3), together with Lemma~\ref{lem:current-frame-certificate}, guarantees that $W^+\in\Nlong(\gamma)$.
Condition (C4) is a local progress condition introduced to obtain a uniform contraction factor and an iteration bound.
The three-term decomposition, the current-iterate correction basis, and the candidate construction in \eqref{eq:corrected-point-paper} do not depend on (C4).
Moreover, the one-factorization reconstruction itself does not require evaluation of $\mathcal J$.
The numerical implementation in Section~\ref{sec:numerical} uses the same candidate formulas and conditions (C1)--(C4), together with ordered preliminary candidates and a finite safeguarded version of the fallback search.
Section~\ref{sec:numerical} lists the numerical search safeguards and fixed parameter values used to make the candidate search more efficient.
Since (C4) uses only quantities computed from the candidate point, checking it requires neither an additional KKT factorization nor an additional algorithm-specific linear-system solve.

Candidate selection has two stages.
First, we evaluate at most $N^{\rm pre}=O(1)$ ordered preliminary candidates satisfying
$\sigma^{\rm pre}\le\sin\alpha\le\sigma_{\max}$, $\rho\in\mathcal P_\rho$, and $0<\beta\le1$, and accept the first admissible candidate.
The preliminary rule is implementation-oriented; the worst-case analysis below uses only that at most $N^{\rm pre}=O(1)$ preliminary candidates are tested before the fallback search.
If no such candidate is accepted, we fix $\rho=\rho_0$ and use the dyadic fallback grid
{
\[
 \sigma_j:=\sigma_{\max}2^{-j},\qquad
 \alpha_j:=\arcsin\sigma_j,\qquad
 \beta_\ell:=2^{-\ell},\qquad j,\ell=0,1,\ldots .
\]
The fallback candidates are examined in stages as specified in Algorithm~\ref{alg:ofas}.}
In Section~\ref{sec:complexity}, we show that this search reaches an admissible candidate with
$\sigma=\Theta(1/\nu)$, $\beta=\Theta(1/\nu)$, and $\rho=\rho_0$.

\subsection{The OFAS Algorithm}

Algorithm~\ref{alg:ofas} summarizes \OFAS{} with the candidate-selection rule above.
In Section~\ref{sec:numerical}, we implement and evaluate this algorithm on the HSD interior-point framework of Clarabel.jl.
For the original SDP, termination is based on the accuracy test
$\mathcal Q_{\rm SDP}(W_k)\le\delta$.
Under Assumption~\ref{ass:primal-dual-solvability}, Subsection~\ref{subsec:normalization-solvable-case} shows that this condition is satisfied after finitely many iterations.
In contrast, neither the HSD candidate selection nor the contraction analysis of the homogeneous progress measure $\mathcal Q$ uses this assumption.
\begin{algorithm}[H]
\caption{One-Factorization Arc-Search Method (OFAS)}
\label{alg:ofas}
\begin{algorithmic}[1]
\Require SDP data $(\A,b,C)$, tolerance $\delta>0$, $\gamma\in(0,1)$,
         $0<\sigma^{\rm pre}\le\sigma_{\max}<1$, and a preliminary candidate rule
\State $W_0\gets(I,0,I,1,1)$, $\omega_0\gets1$, $k\gets0$
\While{$\mathcal Q_{\rm SDP}(W_k)>\delta$}
  \State Compute the current-iterate NT scaling and factorize the corresponding KKT matrix once
  \State Reuse the same factorization for the five right-hand sides defining
  $\dot W_k,\ddot W_k,D_k^{\rm unit},D_k^{\rm mix},D_k^{\rm rem}$; see \eqref{eq:first-derivative-full-kkt-paper}--\eqref{eq:second-order-full-kkt-paper}, \eqref{eq:unit-complementarity-direction}, and \eqref{eq:current-corrector-solution-map}.
  \State Reconstruct $D_k^{\rm cen}$ from \eqref{eq:center-identity-paper}
  {
  \State Test the ordered preliminary candidates from Subsection~\ref{subsec:ofas-mathematical-search}
  \State If none is admissible, set $\rho=\rho_0$ and, for $m=0,1,\ldots$, test the fallback pairs $(j,\ell)$ satisfying $\max\{j,\ell\}=m$, in increasing order of $j$ and then $\ell$, until an admissible candidate is found
  }
  \State Set the accepted triple to $(\alpha_k,\rho_k,\beta_k)$ and
  $W_{k+1}\gets W^+(\alpha_k,\rho_k,\beta_k)$
  \State $\omega_{k+1}\gets(1-\sin\alpha_k)\omega_k$, $k\gets k+1$
\EndWhile
\State \Return $W_k$
\end{algorithmic}
\end{algorithm}

\FloatBarrier
\subsection{CA-OFAS: Curvature-Amplified Candidates}
\label{subsec:caofas-method}

Fix an iteration $k$.
For quantities specific to the OFAS and amplified candidates, we omit the iteration index $k$ when there is no risk of confusion.
We retain $k$ for directions computed at the current iterate, complementarity targets, and quantities used to decide whether to evaluate the amplification branch, so that their relation to the current iterate remains explicit.

At each iteration, \CAOFAS{} first obtains an OFAS candidate
$(\alpha_{\rm OFAS},\rho_{\rm OFAS},\beta_{\rm OFAS})$ accepted by the \OFAS{} candidate-selection rule, together with the corresponding point $W_{\rm OFAS}^+$.
It then evaluates a different family of curvature-amplified arcs using the same current-iterate KKT factorization and correction basis.
Here, $\alpha_{\rm OFAS}$ is the accepted angle obtained from the \OFAS{} candidate-selection rule; it is not the maximum admissible angle on the regular arc $W_k(\alpha)$.
Therefore, the progress guarantee of \OFAS{} is preserved even if no amplified candidate is accepted.

Let $\sigma_{\rm OFAS}:=\sin\alpha_{\rm OFAS}$.
For a fixed amplification factor $\eta>1$ and increment $\Delta_\sigma>0$, define
\begin{equation*}
 \sigma_{\rm CA}:=\min\{\sigma_{\max},\sigma_{\rm OFAS}+\Delta_\sigma\},
 \qquad
 \alpha_{\rm CA}:=\arcsin\sigma_{\rm CA},
 \qquad
 \psi_{\rm CA}:=1-\cos\alpha_{\rm CA}.
\end{equation*}
The value $\sigma_{\rm CA}$ is a trial value for testing an amplified candidate with greater progress than the OFAS candidate.
The amplified candidate does not simply extend $\alpha_{\rm CA}$ farther along the regular arc; instead, it uses a different arc family in which the second-order direction is amplified by $\eta>1$.
We represent the decision to evaluate the amplification branch by $e_k^{\rm amp}\in\{0,1\}$.
This quantity is determined only from values already computed at the current iterate. The activation condition is an implementation choice and does not enter the complexity analysis.

If $e_k^{\rm amp}=1$ and $\sigma_{\rm CA}>\sigma_{\rm OFAS}$, define the amplified arc by
\begin{equation}
 W_{{\rm arc},{\rm CA}}
 :=W_k-\sigma_{\rm CA}\dot W_k+\eta\psi_{\rm CA}\ddot W_k.
 \label{eq:ca-amplified-arc}
\end{equation}
This curve is not obtained by simply extending the regular arc $W_k(\alpha)$ beyond $\alpha_{\rm OFAS}$.
Since $\ddot W_k$ is an HSD residual-invariant direction, this change requires no additional KKT linear-system solve.
The residual reduction factor of the candidate remains $1-\sigma_{\rm CA}$.
For the centering-recovery parameter, we use the OFAS value $\rho_{\rm OFAS}$.

When $\eta\ne1$, simply scaling the \OFAS{} three-term decomposition by $\eta$ is not sufficient because an additional complementarity mismatch proportional to $\Phi(-\dot W_k;W_k)$ appears.
Define
\begin{equation}
 T_{k,{\rm CA}}
 :=(1-\sigma_{\rm CA})\{(1-\rho_{\rm OFAS})\Phi(W_k;W_k)
                  +\rho_{\rm OFAS}\mu(W_k)\bar I\}
 \label{eq:ca-target-complementarity}
\end{equation}
and
\(\Delta_{\rm CA}^{\rm comp}:=\Phi(W_{{\rm arc},{\rm CA}};W_k)-T_{k,{\rm CA}}\).
Also define
\(c_{\rm CA}^{\rm amp}:=(\eta-1)\psi_{\rm CA}\{(\eta+1)\psi_{\rm CA}-2\}\).
The following proposition gives the additional term caused by the amplification.

\begin{proposition}[Exact decomposition of the amplified complementarity mismatch]
\label{prop:ca-amplified-mismatch}
Using $\Delta_k^{\rm cen}$, $\Delta_k^{\rm mix}$, and $\Delta_k^{\rm rem}$ defined in
\eqref{eq:corrector-basis-rhs-definitions}, we have
\begin{equation}
\begin{aligned}
 \Delta_{\rm CA}^{\rm comp}
 &=\rho_{\rm OFAS}(1-\sigma_{\rm CA})\Delta_k^{\rm cen}
   +\eta\sigma_{\rm CA}\psi_{\rm CA}\Delta_k^{\rm mix}
   +\eta^2\psi_{\rm CA}^2\Delta_k^{\rm rem}
   +\Delta_{\rm CA}^{\rm amp},\\
 \Delta_{\rm CA}^{\rm amp}
 &:=c_{\rm CA}^{\rm amp}\Phi(-\dot W_k;W_k).
\end{aligned}
\label{eq:ca-amplified-mismatch-decomposition}
\end{equation}
Therefore, the only additional term beyond the \OFAS{} three-term decomposition is $\Delta_{\rm CA}^{\rm amp}$.
In particular, when $\eta=1$, we have $c_{\rm CA}^{\rm amp}=0$, and the expression reduces to the \OFAS{} three-term decomposition for the same $(\alpha_{\rm CA},\rho_{\rm OFAS})$.
\end{proposition}

\begin{proof}
Expanding \eqref{eq:ca-amplified-arc} quadratically and using
\eqref{eq:pred-second-current-identities} and \eqref{eq:corrector-basis-rhs-definitions} gives
\[
\begin{aligned}
 \Phi(W_{{\rm arc},{\rm CA}};W_k)
 &=(1-\sigma_{\rm CA})\Phi(W_k;W_k)
 +\eta\sigma_{\rm CA}\psi_{\rm CA}\Delta_k^{\rm mix}
 +\eta^2\psi_{\rm CA}^2\Delta_k^{\rm rem}\\
 &\quad+\{\sigma_{\rm CA}^2-2\eta\psi_{\rm CA}+\eta^2\psi_{\rm CA}^2\}\Phi(-\dot W_k;W_k).
\end{aligned}
\]
Since $\sigma_{\rm CA}^2=2\psi_{\rm CA}-\psi_{\rm CA}^2$, the last coefficient is
$(\eta-1)\psi_{\rm CA}\{(\eta+1)\psi_{\rm CA}-2\}=c_{\rm CA}^{\rm amp}$.
On the other hand, \eqref{eq:ca-target-complementarity} can be written as
$T_{k,{\rm CA}}=(1-\sigma_{\rm CA})\Phi(W_k;W_k)-\rho_{\rm OFAS}(1-\sigma_{\rm CA})\Delta_k^{\rm cen}$.
Taking the difference gives \eqref{eq:ca-amplified-mismatch-decomposition}.
\end{proof}

By Proposition~\ref{prop:ca-amplified-mismatch}, the same directions
$D_k^{\rm cen}$, $D_k^{\rm mix}$, and $D_k^{\rm rem}$ can be reused for the amplified correction.
Thus, the only new algebraic quantity in \CAOFAS{} is the response to the amplification-specific term $\Delta_{\rm CA}^{\rm amp}$.
This response can also be reconstructed from the already computed direction $\ddot W_k$.

\begin{lemma}[Algebraic reconstruction of the amplification-specific correction]
\label{lem:ca-amp-direction-solve-free}
\begin{equation}
 D_{k,{\rm CA}}^{\rm amp}
 :=\mathcal G(\Delta_{\rm CA}^{\rm amp};W_k)
 =\frac{c_{\rm CA}^{\rm amp}}{2}\,\ddot W_k.
 \label{eq:ca-amplification-specific-direction}
\end{equation}
Therefore, no additional linear-system solve is required for $D_{k,{\rm CA}}^{\rm amp}$.
\end{lemma}

\begin{proof}
Equation~\eqref{eq:second-order-full-kkt-paper} gives
$\mathcal L(W_k)[\ddot W_k]=(O_{\mathcal F},-2\Phi(\dot W_k;W_k))$.
Using $\Phi(-\dot W_k;W_k)=\Phi(\dot W_k;W_k)$ and the linearity of $\mathcal L(W_k)$, \eqref{eq:second-order-full-kkt-paper} gives
\[
 \mathcal L(W_k)\left[\frac{c_{\rm CA}^{\rm amp}}{2}\ddot W_k\right]
 =\bigl(O_{\mathcal F},-c_{\rm CA}^{\rm amp}\Phi(-\dot W_k;W_k)\bigr)
 =\bigl(O_{\mathcal F},-\Delta_{\rm CA}^{\rm amp}\bigr).
\]
Therefore, the uniqueness of the correction-response map gives \eqref{eq:ca-amplification-specific-direction}.
\end{proof}

Combining the correction directions corresponding to the components of the amplified complementarity mismatch with the same coefficients gives the amplified corrector direction
\begin{equation}
\begin{aligned}
 D_{k,{\rm CA}}^{\rm cor}
 &=\rho_{\rm OFAS}(1-\sigma_{\rm CA})D_k^{\rm cen}
  +\eta\sigma_{\rm CA}\psi_{\rm CA} D_k^{\rm mix}
  +\eta^2\psi_{\rm CA}^2D_k^{\rm rem}
  +\frac{c_{\rm CA}^{\rm amp}}{2}\ddot W_k.
\end{aligned}
 \label{eq:ca-amplified-corrector-direction}
\end{equation}
By linearity and Lemma~\ref{lem:ca-amp-direction-solve-free},
$\mathcal L(W_k)[D_{k,{\rm CA}}^{\rm cor}]
 =\bigl(O_{\mathcal F},-\Delta_{\rm CA}^{\rm comp}\bigr)$.
Therefore, the amplification branch reuses the current-iterate KKT factorization and the five linear-system solves already performed for \OFAS{}.

The average complementarity is adjusted in the same way as in \OFAS{}.
First, define $\mu_{{\rm tar},{\rm CA}}:=(1-\sigma_{\rm CA})\mu(W_k)$.
If the equation in the scalar $\zeta$
\begin{equation}
 \mu\!\left(W_{{\rm arc},{\rm CA}}+D_{k,{\rm CA}}^{\rm cor}+\zeta D_k^{\rm unit}\right)
 =\mu_{{\rm tar},{\rm CA}}
 \label{eq:ca-zeta-definition}
\end{equation}
has a unique solution, denote it by $\zeta_{\rm CA}$.
The scalar $\zeta_{\rm CA}$ can also be computed directly by the same two-point linear formula as in \eqref{eq:zeta-direct-evaluation-method}, with no additional KKT linear-system solve.
If the solution is not unique, we do not construct the amplified candidate.
Next, for a damping factor $\beta_{\rm CA}\in(0,1]$, define the corrected amplified candidate by
\begin{equation*}
 W_{\rm CA}^+(\beta_{\rm CA})
 :=W_{{\rm arc},{\rm CA}}
   +\beta_{\rm CA}\left(D_{k,{\rm CA}}^{\rm cor}+\zeta_{\rm CA} D_k^{\rm unit}\right).
\end{equation*}
For the amplification branch, conditions (C1)--(C4) are evaluated by replacing the corresponding \OFAS{} quantities with
$\mu_{{\rm tar},{\rm CA}}$, $\zeta_{\rm CA}$, $W_{\rm CA}^+(\beta_{\rm CA})$, and
$\omega_{\rm CA}^+:=(1-\sigma_{\rm CA})\omega_k$.
We evaluate, in a prescribed order, at most $N^{\rm amp}=O(1)$ damping values fixed in advance, and accept the first $\beta_{\rm CA}$ for which the amplified candidate satisfies conditions (C1)--(C4).
If no such $\beta_{\rm CA}$ exists, we use the already accepted OFAS point $W_{\rm OFAS}^+$.

In summary, \CAOFAS{} first obtains the guaranteed OFAS candidate of Algorithm~\ref{alg:ofas} and then tests the amplification branch before updating $W_{\rm OFAS}^+$.
An amplified candidate is accepted only if it also satisfies conditions (C1)--(C4); otherwise, the method uses the OFAS candidate.
Therefore, Section~\ref{sec:complexity} can establish the same invariant conditions and iteration bound as for \OFAS{}.
The numerical implementation in Section~\ref{sec:numerical} combines the same amplified-candidate construction with implementation-oriented search rules and safeguards.
Rather than repeating the complete procedure, Algorithm~\ref{alg:caofas} lists only the changes from \OFAS{}.

\begin{algorithm}[H]
\footnotesize
\caption{\CAOFAS{}: Changes from \OFAS{}}
\label{alg:caofas}
\begin{algorithmic}[1]
\Require The inputs of Algorithm~\ref{alg:ofas}, an amplification factor $\eta>1$, an increment $\Delta_\sigma>0$,
         an activation condition, and at most $N^{\rm amp}=O(1)$ prescribed damping values for amplification
\State Run the current \OFAS{} iteration until an admissible OFAS candidate $(\alpha_{\rm OFAS},\rho_{\rm OFAS},\beta_{\rm OFAS})$ and $W_{\rm OFAS}^+$ are obtained, and postpone the update
\State Set $\sigma_{\rm OFAS}\gets\sin\alpha_{\rm OFAS}$ and $\sigma_{\rm CA}\gets\min\{\sigma_{\max},\sigma_{\rm OFAS}+\Delta_\sigma\}$, and evaluate the amplification activation condition
\If{amplification is activated and $\sigma_{\rm CA}>\sigma_{\rm OFAS}$}
  \State Construct $W_{{\rm arc},{\rm CA}}$ and $D_{k,{\rm CA}}^{\rm cor}$ from the directions already computed in the \OFAS{} iteration using \eqref{eq:ca-amplified-arc}, \eqref{eq:ca-amplification-specific-direction}, and \eqref{eq:ca-amplified-corrector-direction}
  \State If \eqref{eq:ca-zeta-definition} has a unique solution $\zeta_{\rm CA}$, evaluate the prescribed amplification damping values $\beta_{\rm CA}$ in order
  \State If any $W_{\rm CA}^+(\beta_{\rm CA})$ satisfies (C1)--(C4), accept the first such candidate as $W_{k+1}$
\EndIf
\State If no amplified candidate is accepted, set $W_{k+1}\gets W_{\rm OFAS}^+$
\State If an amplified candidate was accepted, set $\omega_{k+1}\gets(1-\sigma_{\rm CA})\omega_k$; otherwise set $\omega_{k+1}\gets(1-\sigma_{\rm OFAS})\omega_k$, and proceed to the next iteration
\end{algorithmic}
\end{algorithm}
\FloatBarrier

\section{Long-Step Analysis of OFAS and CA-OFAS}
\label{sec:complexity}

In this section, we give a long-step analysis of \OFAS{} and \CAOFAS{} using the candidate-selection rule defined in Algorithms~\ref{alg:ofas} and~\ref{alg:caofas} of Section~\ref{sec:methods}.
Throughout this section, $\varepsilon>0$ denotes the target accuracy for the homogeneous progress measure, and $\delta>0$ denotes the normalized KKT accuracy for the original SDP.
Our goal is to preserve strict interiority, membership in the wide long-step neighborhood, and uniform progress without solving a new corrector system at the accepted point. The candidate-dependent correction is instead reconstructed in the NT coordinates of the current iterate.
The homogeneous analysis has four parts. We first derive exact identities for the OFAS update and a certificate in the current-iterate NT coordinates. We then establish NT direction estimates in the wide long-step neighborhood. Next, we construct a uniform admissible window with $\sin\alpha=\Theta(1/\nu)$ and $\beta=\Theta(1/\nu)$. Finally, we prove finite termination of the candidate search and an iteration bound of $O\!\left(\nu\log(1/\varepsilon)\right)$.
We then use projective normalization to evaluate the accuracy for the original SDP.

We use the notation of Sections~\ref{sec:framework}--\ref{sec:methods}.
Assumption~\ref{ass:surjectivity} is used for unique solvability of the KKT operator at the current iterate.
Assumption~\ref{ass:primal-dual-solvability}, however, is not used until Subsection~\ref{subsec:normalization-solvable-case}; it is imposed only when recovering a solution of the original SDP by projective normalization.
For the standard initial point $W_0$, let $\mathcal Q_0:=\mathcal Q(W_0)$.
The condition $\mathcal Q(W)\le\varepsilon$ measures the accuracy of the homogeneous HSD quantities and is distinguished from the stopping condition $\mathcal Q_{\rm SDP}(W)\le\delta$ for the original SDP.

To construct the uniformly guaranteed window, we use $\rho_0=1/8$ defined in Subsection~\ref{subsec:ofas-mathematical-search}.
If all preliminary candidates are rejected, the method fixes $\rho=\rho_0$ and moves to the successive-halving fallback search in the angle and damping. The particular outcomes for the preliminary values $\rho\in\mathcal P_\rho=\{1,1/2,1/4,\rho_0\}$ therefore do not affect the worst-case argument.
Therefore, for the worst-case analysis, it is enough to show that this successive-halving fallback search contains an admissible candidate with $\sin\alpha=\Theta(1/\nu)$ and $\beta=\Theta(1/\nu)$.

Below, we analyze the number of outer iterations and use the results of Section~\ref{sec:methods} for the numbers of matrix factorizations and linear-system solves per iteration.
As in \citet{YangYamashita2018}, the basic proof structure consists of bounds on products of directions, the existence of an $O(1/\nu)$ step, and a uniform cumulative decrease rate.

\subsection{Exact Identities for the OFAS Update and a Current-Iterate Certificate}
\label{subsec:exact-identities}

In the following local calculations, we focus on iteration $k$ and write $\dot W:=\dot W_k$ and $\ddot W:=\ddot W_k$ for brevity.
We also set $\sigma:=\sin\alpha$ and $\psi:=1-\cos\alpha$.
The HSD residual map $\Hres$ is linear. The first- and second-order directions satisfy $\Hres(\dot W)=\Hres(W)$ and $\Hres(\ddot W)=0$, respectively.
Define the total complementarity, combining the PSD block and the scalar HSD block, by
$\mathcal C(W):=X\bullet S+\tau\kappa=\nu\mu(W)$.
For HSD variations $U=(U_X,u_y,U_S,u_\tau,u_\kappa)$ and $V=(V_X,v_y,V_S,v_\tau,v_\kappa)$, define the polarization of $\mathcal C$ by
\begin{equation*}
 \Gamma(U,V):=U_X\bullet V_S+V_X\bullet U_S
   +u_\tau v_\kappa+v_\tau u_\kappa.
\end{equation*}
Then $\mathcal C(U+V)=\mathcal C(U)+\Gamma(U,V)+\mathcal C(V)$.
Here, $\Gamma$ is the cross term of the scaling-independent scalar complementarity $\mathcal C$. This is different from the Jordan-product cross term $\mathcal B(\cdot,\cdot;W_k)$ in the current-iterate NT coordinates; see Section~\ref{sec:framework}.

\begin{lemma}[Exact identities for the OFAS update]
\label{lem:ofas-exact-update-identities}
The following statements hold.
\begin{enumerate}[label=(\roman*)]
 \item The predictor arc $W_k(\alpha)=W-\sigma\dot W+\psi\ddot W$ satisfies
 $\Hres(W_k(\alpha))=(1-\sigma)\Hres(W)$.
 This identity is unchanged if the second-order direction is multiplied by any scalar.

 \item For any residual-invariant direction $D=(D_X,d_y,D_S,d_\tau,d_\kappa)$ satisfying $\Hres(D)=0$,
 $D_X\bullet D_S+d_\tau d_\kappa=0$.
 Moreover, for any two residual-invariant directions $D_1,D_2$, we have $\Gamma(D_1,D_2)=0$.
 {Here $D$ denotes an arbitrary residual-invariant direction.}

 \item Suppose that $D_k^{\rm cor}(\alpha,\rho)$ and $D_k^{\rm unit}$ are residual-invariant directions, and define
 \begin{align*}
   \mu_{\rm tar}&:=(1-\sigma)\mu(W), \\
   D_\zeta&:=\Gamma(W_k(\alpha),D_k^{\rm unit}),\\
   N_\zeta&:=\nu\mu_{\rm tar}-\mathcal C(W_k(\alpha))
       -\Gamma(W_k(\alpha),D_k^{\rm cor}(\alpha,\rho)).
 \end{align*}
 If $D_\zeta\neq0$, then
 $
   \zeta=\frac{N_\zeta}{D_\zeta}
 $
 is the unique scalar for which the fully corrected point corresponding to $\beta=1$ in \eqref{eq:corrected-point-paper},
 $
  W_k(\alpha)+D_k^{\rm cor}(\alpha,\rho)+\zeta D_k^{\rm unit},
 $
 satisfies
 $
  \mu\!\left(W_k(\alpha)+D_k^{\rm cor}(\alpha,\rho)+\zeta D_k^{\rm unit}\right)=\mu_{\rm tar}.
 $

 \item Let
 {
 \[
   D_k^{\rm full}(\alpha,\rho)
   :=D_k^{\rm cor}(\alpha,\rho)+\zeta D_k^{\rm unit},\qquad
   W^+(\beta):=W_k(\alpha)+\beta D_k^{\rm full}(\alpha,\rho),\qquad 0\le\beta\le1.
 \]
 }
 Then
 {
 \begin{align}
   \Hres(W^+(\beta))&=(1-\sigma)\Hres(W),
   \label{eq:stage2-residual-exact}\\
   \mu(W^+(\beta))&=(1-\beta)\mu(W_k(\alpha))
      +\beta(1-\sigma)\mu(W).
   \label{eq:stage2-mu-exact}
 \end{align}
 }
\end{enumerate}
\end{lemma}

\begin{proof}
{
For (i), linearity of $\Hres$ gives
\[
 \Hres(W_k(\alpha))
 =\Hres(W)-\sigma\Hres(\dot W)+\psi\Hres(\ddot W)
 =(1-\sigma)\Hres(W).
\]
Multiplying $\ddot W$ by any scalar does not change this identity because
$\Hres(\ddot W)=0$.
}

For (ii), the orthogonality relation follows from the skew-adjoint structure of the HSD constraint system.
The same structure appears in the HSD formulation for LP \citep{YeToddMizuno1994}, but we derive it directly here from the SDP-HSD residual equations used in this paper.
The condition $\Hres(D)=0$ means
$\A(D_X)=bd_\tau$, $\Ast(d_y)+D_S=Cd_\tau$, and
$C\bullet D_X-b^Td_y+d_\kappa=0$.
Taking the inner product of the second equation with $D_X$ and using the first equation gives
\[
 D_X\bullet D_S
 =d_\tau(C\bullet D_X-b^Td_y)
 =-d_\tau d_\kappa.
\]
Hence, $\mathcal C(D)=0$.
Since the residual-invariance equations are linear, $D_1+D_2$ is also residual-invariant.
Therefore,
$0=\mathcal C(D_1+D_2)=\mathcal C(D_1)+\Gamma(D_1,D_2)+\mathcal C(D_2)$,
which gives $\Gamma(D_1,D_2)=0$.
This proof uses only the matrix inner product and adjoint relations; it is not specific to LP.

{
For (iii), both $D_k^{\rm cor}(\alpha,\rho)$ and $D_k^{\rm unit}$ are residual-invariant.
Part (ii) therefore gives
\[
 \mathcal C(D_k^{\rm cor}(\alpha,\rho))=\mathcal C(D_k^{\rm unit})=0,
 \qquad
 \Gamma(D_k^{\rm cor}(\alpha,\rho),D_k^{\rm unit})=0.
\]
Thus, for any scalar $\zeta$,
\[
 \mathcal C\!\left(D_k^{\rm cor}(\alpha,\rho)+\zeta D_k^{\rm unit}\right)=0.
\]
Moreover, the same quadratic expansion gives
\[
 \begin{aligned}
 &\mathcal C\!\left(W_k(\alpha)+D_k^{\rm cor}(\alpha,\rho)+\zeta D_k^{\rm unit}\right)\\
 &\quad=\mathcal C(W_k(\alpha))
 +\Gamma\!\left(W_k(\alpha),D_k^{\rm cor}(\alpha,\rho)\right)+\zeta D_\zeta.
 \end{aligned}
\]
Hence, the average complementarity after the full correction is exactly affine in $\zeta$.
Equating it to $\nu\mu_{\rm tar}$ gives $\zeta=N_\zeta/D_\zeta$.
}

{
For (iv), write $D^{\rm full}:=D_k^{\rm full}(\alpha,\rho)$.
By part~(ii), $\mathcal C(D^{\rm full})=0$. Hence, for every $0\le\beta\le1$,
\[
 \mathcal C(W_k(\alpha)+\beta D^{\rm full})
 =\mathcal C(W_k(\alpha))
   +\beta\Gamma(W_k(\alpha),D^{\rm full}).
\]
At $\beta=1$, part~(iii) gives
\[
 \mathcal C(W_k(\alpha))
 +\Gamma(W_k(\alpha),D^{\rm full})
 =\nu\mu_{\rm tar}.
\]
Eliminating the cross term between these two identities yields
\[
 \mathcal C(W_k(\alpha)+\beta D^{\rm full})
 =(1-\beta)\mathcal C(W_k(\alpha))+\beta\nu\mu_{\rm tar}.
\]
Dividing by $\nu$ gives \eqref{eq:stage2-mu-exact}.
Likewise, linearity of $\Hres$ and $\Hres(D^{\rm full})=0$ give
$\Hres(W_k(\alpha)+\beta D^{\rm full})=\Hres(W_k(\alpha))$, which is
\eqref{eq:stage2-residual-exact} by part~(i).
}
\end{proof}

This lemma does not require strict interiority of the predictor arc, so the same algebra also applies to the amplified arc of \CAOFAS{}.
However, positivity and membership in the wide long-step neighborhood for an amplified candidate are checked separately by the acceptance conditions.
Also, when $\beta<1$, in general $\mu(W^+)\neq(1-\sigma)\mu(W)$.
The iteration analysis below therefore does not rely on this equality.

The next lemma connects the current-iterate NT certificate in acceptance condition (C3) of Section~\ref{sec:methods} with the scaling-independent wide long-step neighborhood.

\begin{lemma}[Wide Long-Step Neighborhood Certificate in the Current-Iterate NT Coordinates]
\label{lem:current-frame-certificate}
For a strictly interior candidate $W_{\rm cand}$, if
 $\lambda_{\min}\!\bigl(\Phi(W_{\rm cand};W_k)\bigr)
 \ge\gamma\mu(W_{\rm cand})$,
then
 $W_{\rm cand}\in\Nlong(\gamma)$.
\end{lemma}

\begin{proof}
Write $W_{\rm cand}=(X_{\rm cand},y_{\rm cand},S_{\rm cand},\tau_{\rm cand},\kappa_{\rm cand})$.
Apply the current-iterate NT coordinates in \eqref{eq:fixed-nt-coordinate-map} to $W_{\rm cand}$ and define the PSD blocks by
$P_X=M_kX_{\rm cand}M_k^T$ and $R_X=M_k^{-T}S_{\rm cand}M_k^{-1}$.
Then $P_XR_X=M_kX_{\rm cand}S_{\rm cand}M_k^{-1}$, so $P_XR_X$ is similar to
$X_{\rm cand}^{1/2}S_{\rm cand}X_{\rm cand}^{1/2}$ and therefore has the same positive real eigenvalues.
Symmetrization of the noncommutative complementarity product is standard in the analysis of Monteiro--Zhang-type search directions for SDP, and the NT direction is also treated in this framework \citep{MonteiroZhang1998,ToddTohTutuncu1998}.
Let $\lambda_*$ be the smallest eigenvalue of $P_XR_X$, and let $v$ be a corresponding real unit eigenvector.
Since $P_X$ and $R_X$ are symmetric,
$v^T(P_X\circ R_X)v=v^TP_XR_Xv=\lambda_*$.
The Rayleigh quotient therefore gives
\[
 \lambda_{\min}(P_X\circ R_X)
 \le \lambda_*
 =\lambda_{\min}(X_{\rm cand}^{1/2}S_{\rm cand}X_{\rm cand}^{1/2}).
\]
For the scalar HSD block,
$(\xi_k\tau_{\rm cand})(\xi_k^{-1}\kappa_{\rm cand})=\tau_{\rm cand}\kappa_{\rm cand}$.
Therefore, the lower bound in (C3) also holds for the true complementarity spectrum of both the PSD block and the scalar block of $W_{\rm cand}$.
Hence, $W_{\rm cand}\in\Nlong(\gamma)$.
\end{proof}

This certificate is sufficient but not necessary.
For a general candidate, $P_X$ and $R_X$ need not commute, and the smallest eigenvalue of $P_X\circ R_X$ can be smaller than the smallest value in the true complementarity spectrum.
Thus, the test can be performed entirely in the current-iterate NT coordinates, although it can be conservative.

\subsection{NT Direction Estimates in the Wide Long-Step Neighborhood}
\label{subsec:wide-direction-estimates}

This subsection collects the direction estimates needed later for the arc remainder and the guaranteed candidate.
We use the product Euclidean Jordan algebra $\mathcal E=\Ssym^n\times\R$ defined in Section~\ref{sec:framework}, whose rank is $\nu=n+1$.
Fix an iteration and write $W:=W_k\in\Nlong(\gamma)$. We use the following local NT notation:
\begin{equation}
 {
 \begin{aligned}
 Q&:=Q^{\rm NT}(W),\qquad H:=Q\circ Q=\Phi(W;W),\qquad \mu:=\mu(W),\\
 L_Q(U)&:=Q\circ U,\\
 (\dot P,\dot R)&:=(P(\dot W;W),R(\dot W;W)),\\
 (\ddot P,\ddot R)&:=(P(\ddot W;W),R(\ddot W;W)),\\
 (P_k^{\rm term},R_k^{\rm term})
 &:= (P(D_k^{\rm term};W_k),R(D_k^{\rm term};W_k)),\\
 &\hspace{2em}{\rm term}\in\{{\rm cen},{\rm mix},{\rm rem},{\rm unit}\}.
 \end{aligned}}
 \label{eq:wide-nt-local-notation}
\end{equation}
By \eqref{eq:fixed-nt-coordinate-map}, the spectrum of the PSD block of $H$ is $\lambda(X,S)$, while the scalar HSD block is $\tau\kappa$.
Since $W\in\Nlong(\gamma)$, every spectral value of $H$ is at least $\gamma\mu$, and their sum is
$X\bullet S+\tau\kappa=\nu\mu$.

We use the standard theory of NT scaling on self-scaled cones \citep{NesterovTodd1997,NesterovTodd1998} and derive the estimates needed here in the present notation.
Similar estimates for scaled directions are also used in arc-search methods for SDP \citep{KheirfamMoslemi2018,ZhangYuanZhouLuoHuang2019}.
We also need their relation to the orthogonality of HSD residual-invariant directions.
The NT congruence transformation preserves the primal-dual cross inner product.
Thus, for a direction $D$ with NT components $(P^D,R^D)$,
$\langle P^D,R^D\rangle=D_X\bullet D_S+d_\tau d_\kappa$.
Hence, by Lemma~\ref{lem:ofas-exact-update-identities}(ii),
$\langle P^D,R^D\rangle=0$ for every HSD residual-invariant direction.

\begin{proposition}[NT-scaled direction estimates for OFAS]
\label{prop:ofas-nt-scaled-bounds}
Let $W\in\Nlong(\gamma)$.
Then the NT components of the first- and second-order directions satisfy
\begin{align}
 \|\dot P\|^2+\|\dot R\|^2&=\nu\mu,
 &\|\dot P\circ\dot R\|&\le\frac{\nu\mu}{2},
 \label{eq:wide-first-product}\\
 \|\ddot P\|^2+\|\ddot R\|^2&\le\frac{\nu^2}{\gamma}\mu,
 &\|\dot P\circ\ddot R+\ddot P\circ\dot R\|
 &\le\frac{\nu^{3/2}}{\sqrt\gamma}\mu,\notag\\
 &&\|\ddot P\circ\ddot R\|&\le\frac{\nu^2}{2\gamma}\mu.\notag
\end{align}
For the centering right-hand side,
\[
 \|L_Q^{-1}(H-\mu\bar I)\|^2
 \le \nu(\gamma^{-1}-1)\mu.
\]
Let $(P_\Delta,R_\Delta)$ be the NT components of the residual-invariant solution $D=\mathcal G(\Delta;W)$ whose only complementarity right-hand side is $-\Delta$.
Then
\[
 L_Q(P_\Delta+R_\Delta)=-\Delta,\qquad
 \langle P_\Delta,R_\Delta\rangle=0,
\]
and therefore
\[
 \sqrt{\|P_\Delta\|^2+\|R_\Delta\|^2}
 =\|L_Q^{-1}(\Delta)\|.
\]
Moreover, for any two pairs $(P,R)$ and $(\widehat P,\widehat R)$,
\[
 \|P\circ\widehat R+\widehat P\circ R\|
 \le
 \sqrt{\|P\|^2+\|R\|^2}\,
 \sqrt{\|\widehat P\|^2+\|\widehat R\|^2},
 \qquad
 \|P\circ R\|\le\frac12(\|P\|^2+\|R\|^2).
\]
In particular, the four correction-basis directions defined in Section~\ref{sec:methods} satisfy
\begin{align*}
 \sqrt{\|P_k^{\rm cen}\|^2+\|R_k^{\rm cen}\|^2}
 &\le \sqrt{\nu(\gamma^{-1}-1)\mu},\\
 \sqrt{\|P_k^{\rm mix}\|^2+\|R_k^{\rm mix}\|^2}
 &\le \frac{\nu^{3/2}}{\gamma}\sqrt\mu,\\
 \sqrt{\|P_k^{\rm rem}\|^2+\|R_k^{\rm rem}\|^2}
 &\le \frac{1+\gamma}{2\gamma^{3/2}}\nu^2\sqrt\mu,\\
 \sqrt{\|P_k^{\rm unit}\|^2+\|R_k^{\rm unit}\|^2}
 &\le \sqrt{\frac{\nu}{\gamma\mu}}.
\end{align*}
\end{proposition}

\begin{proof}
For NT scaling on self-scaled cones, we use the standard Jordan-product bound
$\|U\circ V\|\le\|U\|\|V\|$ and the inverse-operator estimate obtained from the spectral decomposition \citep{NesterovTodd1997,NesterovTodd1998}.
Since every spectral value of $H=Q\circ Q$ is at least $\gamma\mu$, every spectral value of $Q$ is at least $\sqrt{\gamma\mu}$.
If the spectral values of $Q$ on the PSD block are denoted by $q_i$, the eigenvalues of $L_Q$ are $(q_i+q_j)/2$; on the scalar HSD block, $L_Q$ is multiplication by the scalar component of $Q$.
These are eigenvalues of the linear operator $L_Q:\Ssym^n\to\Ssym^n$, not the ordinary matrix eigenvalues of $Q_X$. Indeed, if $Q_X=U\operatorname{Diag}(q_1,\ldots,q_n)U^T$, then in the induced symmetric-matrix basis the mode associated with indices $(i,j)$ is multiplied by $(q_i+q_j)/2$ (with value $q_i$ when $i=j$).
Therefore, for every $U\in\mathcal E$,
\begin{equation}
 \|L_Q^{-1}(U)\|\le\frac{1}{\sqrt{\gamma\mu}}\|U\|.
 \label{eq:lq-inverse-wide-bound}
\end{equation}

The standard Jordan-product bound and Cauchy--Schwarz also give, for any two pairs $(P,R)$ and $(\widehat P,\widehat R)$,
\[
 \|P\circ\widehat R+\widehat P\circ R\|
 \le
 \sqrt{\|P\|^2+\|R\|^2}\,
 \sqrt{\|\widehat P\|^2+\|\widehat R\|^2},
 \qquad
 \|P\circ R\|\le\frac12(\|P\|^2+\|R\|^2).
\]

We first establish orthogonality of the first-order direction, which does not follow directly from residual invariance.
For
$U=(U_X,u_y,U_S,u_\tau,u_\kappa)$ and
$V=(V_X,v_y,V_S,v_\tau,v_\kappa)$,
direct substitution into the HSD residual equations gives
\begin{align*}
 \Gamma(U,V)
 &=U_X\bullet\Rd(V)-u_y^T\Rp(V)+u_\tau\Rg(V)\\
 &\quad+V_X\bullet\Rd(U)-v_y^T\Rp(U)+v_\tau\Rg(U).
\end{align*}
Since $\Hres(\dot W)=\Hres(W)$, we obtain
$\Gamma(W,\dot W)=\mathcal C(W)+\mathcal C(\dot W)$.
On the other hand, taking the inner product of the first-order complementarity equation
$L_Q(\dot P+\dot R)=H$ with $\bar I$ gives
$\Gamma(W,\dot W)=\langle Q,\dot P+\dot R\rangle=\langle\bar I,H\rangle=\mathcal C(W)$.
Thus, $\mathcal C(\dot W)=0$, and preservation of the cross inner product under the NT congruence transformation gives
$\langle\dot P,\dot R\rangle=0$.
Also, $L_Q(\dot P+\dot R)=L_Q(Q)$ implies
$\dot P+\dot R=Q$.
Therefore,
\[
 \|\dot P\|^2+\|\dot R\|^2=\|Q\|^2=\nu\mu.
\]
Using $\|U\circ V\|\le\|U\|\|V\|$ and $2ab\le a^2+b^2$, we obtain
\[
 \|\dot P\circ\dot R\|
 \le\|\dot P\|\|\dot R\|
 \le\frac12(\|\dot P\|^2+\|\dot R\|^2)
 =\frac{\nu\mu}{2},
\]
which is the product bound in \eqref{eq:wide-first-product}.

The second-order direction satisfies $\Hres(\ddot W)=O_{\mathcal F}$.
Hence, Lemma~\ref{lem:ofas-exact-update-identities}(ii) gives
$\langle\ddot P,\ddot R\rangle=0$.
Moreover,
$L_Q(\ddot P+\ddot R)=-2(\dot P\circ\dot R)$.
Using \eqref{eq:lq-inverse-wide-bound} and \eqref{eq:wide-first-product},
\[
 \|\ddot P+\ddot R\|
 \le\frac{2}{\sqrt{\gamma\mu}}\|\dot P\circ\dot R\|
 \le\frac{\nu}{\sqrt\gamma}\sqrt\mu.
\]
By orthogonality, the square of the left-hand side equals
$\|\ddot P\|^2+\|\ddot R\|^2$, which gives the stated second-order bound.
The standard Jordan-product bound and Cauchy--Schwarz then give
\begin{align*}
 \|\dot P\circ\ddot R+\ddot P\circ\dot R\|
 &\le
 \sqrt{\|\dot P\|^2+\|\dot R\|^2}\,
 \sqrt{\|\ddot P\|^2+\|\ddot R\|^2}
 \le\frac{\nu^{3/2}}{\sqrt\gamma}\mu,\\
 \|\ddot P\circ\ddot R\|
 &\le\frac12(\|\ddot P\|^2+\|\ddot R\|^2)
 \le\frac{\nu^2}{2\gamma}\mu.
\end{align*}

We next estimate the centering right-hand side.
Let the spectral decomposition of $H$ be
$H=\mu\sum_{i=1}^{\nu}t_i c_i$, where $\{c_i\}_{i=1}^{\nu}$ is a Jordan frame of mutually orthogonal primitive idempotents.
Then $\bar I=\sum_i c_i$ and
$Q=\sqrt\mu\sum_i\sqrt{t_i}c_i$, with
$t_i\ge\gamma$ and $\sum_i t_i=\nu$.
Hence,
\begin{align*}
 \|L_Q^{-1}(H-\mu\bar I)\|^2
 &=\mu\sum_{i=1}^{\nu}\frac{(t_i-1)^2}{t_i}
 =\mu\left(\sum_{i=1}^{\nu}t_i^{-1}-\nu\right)
  \le\nu(\gamma^{-1}-1)\mu.
\end{align*}

Finally, for the residual-invariant solution $D=\mathcal G(\Delta;W)$,
$L_Q(P_\Delta+R_\Delta)=-\Delta$ and
$\langle P_\Delta,R_\Delta\rangle=0$.
Therefore,
\[
 \|P_\Delta\|^2+\|R_\Delta\|^2
 =\|P_\Delta+R_\Delta\|^2
 =\|L_Q^{-1}(\Delta)\|^2.
\]
By the definitions in Section~\ref{sec:methods} and the product bounds above,
\begin{align*}
 \|\Delta_k^{\rm mix}\|&\le\frac{\nu^{3/2}}{\sqrt\gamma}\mu,\\
 \|\Delta_k^{\rm rem}\|
 &\le \|\ddot P\circ\ddot R\|+\|\dot P\circ\dot R\|\\
 &\le \frac{\nu^2}{2\gamma}\mu+\frac{\nu}{2}\mu
 \le \left(\frac{1}{2\gamma}+\frac12\right)\nu^2\mu
 = \frac{1+\gamma}{2\gamma}\nu^2\mu.
\end{align*}
The last inequality uses $\nu\ge1$. Applying \eqref{eq:lq-inverse-wide-bound} to this estimate gives the factor $(1+\gamma)/(2\gamma^{3/2})$ in Proposition~\ref{prop:ofas-nt-scaled-bounds}.
Combining these estimates with \eqref{eq:lq-inverse-wide-bound}, and applying the centering estimate above to
$\Delta_k^{\rm cen}=H-\mu\bar I$, gives the stated bounds for
$D_k^{\rm cen}$, $D_k^{\rm mix}$, and $D_k^{\rm rem}$.
The bound for $D_k^{\rm unit}$ follows from
$D_k^{\rm unit}=\mathcal G(-\bar I;W)$ and $\|\bar I\|=\sqrt\nu$.
\end{proof}

These are the scaled estimates needed later for the arc remainder and the OFAS correction.
The detailed constants are collected in Appendix~\ref{app:guaranteed-ofas-candidate} and are used only where needed.

\subsection{Uniformly Guaranteed Candidates and Candidate Selection}
\label{subsec:ofas-search-termination}

Throughout this subsection, let $\sigma:=\sin\alpha$ and $\psi:=1-\cos\alpha$.
At iteration $k$, the current iterate $W$ has primal and dual factors $P=R=Q$ in NT coordinates.
Using the same current-iterate scaling, write the primal and dual factors at a point on the predictor arc as
\begin{equation*}
 P(\alpha)=Q-\sigma\dot P+\psi\ddot P,
 \qquad
 R(\alpha)=Q-\sigma\dot R+\psi\ddot R.
\end{equation*}
Thus, $P(\alpha)-Q$ and $R(\alpha)-Q$ are the NT-coordinate displacements from $W$ to the predictor-arc point.
Combining $\psi\le\sigma^2$ with Proposition~\ref{prop:ofas-nt-scaled-bounds} yields the next arc estimate.

\begin{lemma}[Wide-neighborhood arc displacement and remainder]
\label{lem:wide-arc-remainder}
Let $\Delta P(\alpha):=P(\alpha)-Q$ and $\Delta R(\alpha):=R(\alpha)-Q$.
If {$0\le\sigma\le\min\{1,c/\nu\}$}, then
\begin{equation*}
 \norm{\Delta P(\alpha)}+\norm{\Delta R(\alpha)}
 \le d_\gamma(c)\sqrt{\frac{{\mu(W)}}{\nu}},
 \qquad
 d_\gamma(c):=\sqrt2c+\sqrt{\frac2\gamma}c^2.
\end{equation*}
Moreover,
\begin{equation}
 P(\alpha)\circ R(\alpha)
 =(1-\sigma)H+\mathcal E_{\rm arc}(\alpha),
 \label{eq:wide-arc-product-expansion}
\end{equation}
where
\begin{equation*}
 \mathcal E_{\rm arc}(\alpha)
 =-\psi^2(\dot P\circ\dot R)
  -\sigma\psi(\dot P\circ\ddot R+\ddot P\circ\dot R)
  +\psi^2(\ddot P\circ\ddot R),
\end{equation*}
and
\begin{equation}
 \norm{\mathcal E_{\rm arc}(\alpha)}
 \le r_\gamma(c)\frac{{\mu(W)}}{\nu^{3/2}},
 \quad
 r_\gamma(c):=
 \frac{c^3}{\sqrt\gamma}
 +\frac{c^4}{2}
 +\frac{c^4}{2\gamma}.
 \label{eq:wide-arc-remainder-bound}
\end{equation}
For the average complementarity, we also have
\begin{equation}
 \left|\mu(W_k(\alpha))-(1-\sigma)\mu(W)\right|
 \le r_\gamma(c)\frac{{\mu(W)}}{\nu^2}.
 \label{eq:wide-arc-mu-remainder-bound}
\end{equation}
\end{lemma}

\begin{proof}
{Throughout this proof, $\mu$ denotes $\mu(W)$, consistently with the local notation in \eqref{eq:wide-nt-local-notation}.}
From $\psi\le\sigma^2$ and Proposition~\ref{prop:ofas-nt-scaled-bounds},
\[
 \|\Delta P(\alpha)\|+\|\Delta R(\alpha)\|
 \le \sigma\sqrt{2\nu\mu}
 +\sigma^2\sqrt{2/\gamma}\,\nu\sqrt\mu
 \le d_\gamma(c)\sqrt{\mu/\nu}.
\]
Expanding $P(\alpha)=Q-\sigma\dot P+\psi\ddot P$ and
$R(\alpha)=Q-\sigma\dot R+\psi\ddot R$, and using
$Q\circ(\dot P+\dot R)=H$,
$Q\circ(\ddot P+\ddot R)=-2\dot P\circ\dot R$, and
$\sigma^2-2\psi=-\psi^2$, gives \eqref{eq:wide-arc-product-expansion}.
Therefore, using the three product estimates in Proposition~\ref{prop:ofas-nt-scaled-bounds}, together with
$\psi\le\sigma^2$ and $\sigma\le c/\nu$, we obtain
\begin{align*}
 \|\mathcal E_{\rm arc}(\alpha)\|
 &\le \psi^2\frac{\nu\mu}{2}
 +\sigma\psi\frac{\nu^{3/2}}{\sqrt\gamma}\mu
 +\psi^2\frac{\nu^2}{2\gamma}\mu
 \le r_\gamma(c)\frac{\mu}{\nu^{3/2}},
\end{align*}
which proves \eqref{eq:wide-arc-remainder-bound}.
Finally, from \eqref{eq:wide-arc-product-expansion},
$\mathcal C(W_k(\alpha))=\langle\bar I,P(\alpha)\circ R(\alpha)\rangle_{\mathcal E}$, and
$\mathcal C(W)=\langle\bar I,H\rangle_{\mathcal E}=\nu\mu$, we have
\begin{align*}
 \mu(W_k(\alpha))-(1-\sigma)\mu(W)
 &=\frac1\nu\langle\bar I,\mathcal E_{\rm arc}(\alpha)\rangle_{\mathcal E}.
\end{align*}
Since $\|\bar I\|=\sqrt\nu$, the Cauchy--Schwarz inequality and \eqref{eq:wide-arc-remainder-bound} give
\[
 \left|\mu(W_k(\alpha))-(1-\sigma)\mu(W)\right|
 \le\frac1{\sqrt\nu}\|\mathcal E_{\rm arc}(\alpha)\|
 \le r_\gamma(c)\frac{\mu}{\nu^2},
\]
which proves \eqref{eq:wide-arc-mu-remainder-bound}.
\end{proof}

To preserve strict interiority and the wide long-step neighborhood, the correction damping parameter $\beta$ must be sufficiently small.
At the same time, this damping must not destroy HSD progress if we are to obtain a long-step iteration bound.
We first show that the arc itself gives $O(1/\nu)$ progress and that a damped correction preserves this uniform decrease.

We first establish uniform progress for a small angle.
Let $W\in\Nlong(\gamma)$, and let $0<c\le1/2$ satisfy $r_\gamma(c)\le c/4$ and
$c/(2\nu)<\sigma\le c/\nu$.
Choose $\zeta$ so that the fully corrected point has average complementarity $(1-\sigma)\mu(W)$. Combining Lemma~\ref{lem:wide-arc-remainder} with the damping interpolation in Lemma~\ref{lem:ofas-exact-update-identities}, we obtain, for every $0\le\beta\le1$,
{
\begin{equation}
 \mu(W^+(\beta))
 \le\left(1-\sigma+\frac{r_\gamma(c)}{\nu^2}\right)\mu(W)
 \le\left(1-\frac{c}{4\nu}\right)\mu(W).
 \label{eq:small-angle-mu-progress}
\end{equation}
}
The last inequality uses $\sigma>c/(2\nu)$, $r_\gamma(c)\le c/4$, and $\nu\ge1$.
{Combining the exact residual identity \eqref{eq:stage2-residual-exact} with the average complementarity bound \eqref{eq:small-angle-mu-progress} gives
$\mathcal Q(W^+(\beta))\le(1-c/(4\nu))\mathcal Q(W)$.}
This estimate does not use a positive lower bound on $\beta$.
Later, we impose $\beta=\Theta(1/\nu)$ for the neighborhood argument, not for progress. The upper bound on $\beta$ preserves strict interiority, while the positive lower bound provides enough correction to recover the wide long-step neighborhood.

We next collect the constants required for a uniformly guaranteed candidate.
For the fixed branch $\rho=\rho_0=1/8$, Appendix~\ref{app:guaranteed-ofas-candidate} explicitly constructs positive constants
$b_\gamma,c_\gamma>0$ that depend only on $\gamma$, $\rho_0$, and $\sigma_{\max}$, and not on $\nu$ or the problem data.
They are chosen so that, for $\sigma=\Theta(1/\nu)$ and $\beta=\Theta(1/\nu)$, the well-definedness of $\zeta$, strict interiority, recovery of the wide long-step neighborhood, and the local progress condition below all hold.
The appendix gives the explicit constants and verifies the required estimates.
Here, we use only the resulting uniform admissible window.

Combining these results, every candidate triple with $\sigma$ and $\beta$ in suitable intervals of width proportional to $1/\nu$ satisfies (C1)--(C3). It also gives a uniform decrease in both $\mathcal Q$ and the projective scale measure $\mathcal J$.
The next proposition gives an admissible window of positive width that the successive-halving fallback search is guaranteed to reach.

\begin{proposition}[Uniform admissible window for OFAS candidate selection]
\label{prop:ofas-guaranteed-search-window}
Let $W\in\Nlong(\gamma)$ and $\omega>0$, and use the notation of Section~\ref{sec:methods}.
{Let $b_\gamma$ and $c_\gamma$ be the positive constants constructed in Appendix~\ref{app:guaranteed-ofas-candidate}, specifically in \eqref{eq:b-gamma-app}--\eqref{eq:c-gamma-app}; in particular,
$0<c_\gamma\le\min\{\sigma_{\max},1/2\}$.}
For any candidate satisfying
\[
 \frac{c_\gamma}{2\nu}<\sigma\le\frac{c_\gamma}{\nu},\qquad
 \rho=\rho_0=\frac18,\qquad
 \frac{b_\gamma}{4\nu}\le\beta\le\frac{b_\gamma}{\nu},
\]
we have $\mu_{\rm tar}>0$ and $D_\zeta\ge\nu/2>0$, so $\zeta$ is well-defined.
The corrected point $W^+$ is strictly interior and satisfies
$\lambda_{\min}(\Phi(W^+;W))\ge\gamma\mu(W^+)$, and hence $W^+\in\Nlong(\gamma)$.
Moreover, if $\omega^+:=(1-\sigma)\omega$, then the local progress condition (C4) of Section~\ref{sec:methods} holds.
Therefore, every candidate triple $(\alpha,\rho,\beta)$ in the three ranges above satisfies conditions (C1)--(C4).
\end{proposition}

\begin{proof}
Since $c_\gamma\le1/2$, we have $\mu_{\rm tar}=(1-\sigma)\mu(W)>0$.
The bound $D_\zeta\ge\nu/2$, strict interiority, and the current-iterate NT certificate follow from
Lemma~\ref{lem:fixed-rho-zeta-cancellation}, Lemma~\ref{lem:fixed-rho-damped-positivity},
Lemma~\ref{lem:fixed-rho-damped-restoration}, and the choices of $c_\gamma$ and $b_\gamma$ used there.
Thus, (C1)--(C3) hold.

For (C4), the small-angle progress estimate above gives
$\mathcal Q(W^+)\le(1-c_\gamma/(4\nu))\mathcal Q(W)$.
Since $\sigma\le c_\gamma/\nu$, we have
$1-c_\gamma/(4\nu)\le1-\sigma/4$, which gives (C4) for $\mathcal Q$.

Next, let $a_\gamma:=r_\gamma(c_\gamma)/(1-\sigma_{\max})$.
{
By the definition of $c_\gamma$ in Appendix~\ref{app:guaranteed-ofas-candidate}, specifically \eqref{eq:c-gamma-app},
\[
 r_\gamma(c_\gamma)\le\bar r_\gamma c_\gamma^3,
 \qquad
 c_\gamma^2\le\frac{1-\sigma_{\max}}{64\bar r_\gamma}.
\]
Hence
\[
 a_\gamma
 =\frac{r_\gamma(c_\gamma)}{1-\sigma_{\max}}
 \le\frac{\bar r_\gamma c_\gamma^3}{1-\sigma_{\max}}
 \le\frac{c_\gamma}{64}.
\]
}
Using the {average complementarity} estimate in Lemma~\ref{lem:fixed-rho-damped-restoration} and $\omega^+=(1-\sigma)\omega$, set
$\theta:=\mu(W)/\omega$ and $\theta^+:=\mu(W^+)/\omega^+$.
Then
\[
 \left|\frac{\theta^+}{\theta}-1\right|
 \le \frac{a_\gamma}{\nu^2}
 \le \frac{c_\gamma}{64\nu^2}.
\]
Hence $\theta^+/\theta\ge1-a_\gamma/\nu^2$.
Using also $\sigma>c_\gamma/(2\nu)$, the three ratios that define $\mathcal J$ satisfy
\[
 \max\left\{
 \frac{\omega^+}{\omega},
 \frac{\omega^+/\theta^+}{\omega/\theta},
 \frac{\omega^+/(\theta^+)^2}{\omega/\theta^2}
 \right\}
 \le
 \frac{1-c_\gamma/(2\nu)}{(1-a_\gamma/\nu^2)^2}
 \le 1-\frac{c_\gamma}{4\nu}
 \le 1-\frac{\sigma}{4}.
\]
The middle inequality follows from
\[
 \left(1-\frac{c_\gamma}{4\nu}\right)
 \left(1-\frac{a_\gamma}{\nu^2}\right)^2
 -\left(1-\frac{c_\gamma}{2\nu}\right)
 \ge \frac{c_\gamma}{4\nu}-\frac{2a_\gamma}{\nu^2}>0.
\]
Thus, (C4) also holds for $\mathcal J$.
\end{proof}

From now on, let $\alpha_k$ be the angle of the candidate accepted by \OFAS{}, and write $\sigma_k:=\sin\alpha_k$.
Define the positive constant
\[
 c_\gamma^{\rm fb}:=\frac12\min\{c_\gamma,\sigma_{\max}b_\gamma\}.
\]

\begin{theorem}[Well-definedness of the OFAS search and a lower bound on the accepted angle]
\label{thm:ofas-candidate-selection-termination}
Suppose that $W_k\in\Nlong(\gamma)$.
Even if none of the preliminary candidates is admissible, the search of Section~\ref{sec:methods} finds a candidate satisfying (C1)--(C4) after finitely many candidate tests.
Moreover, if a candidate is accepted by this search, then $\sigma_k>c_\gamma^{\rm fb}/\nu$.
\end{theorem}
\begin{proof}
\begingroup
{Let $J_\nu$ be the smallest nonnegative integer satisfying
$\sigma_{J_\nu}\le c_\gamma/\nu$, and let $L_\nu$ be the smallest nonnegative integer satisfying
$\beta_{L_\nu}\le b_\gamma/\nu$.
Since $\sigma_j=\sigma_{\max}2^{-j}\to0$ and $\beta_\ell=2^{-\ell}\to0$, both indices are finite.
Moreover, $c_\gamma\le\sigma_{\max}$ by Proposition~\ref{prop:ofas-guaranteed-search-window}.}
The defining inequality gives
$\sigma_{J_\nu}\le c_\gamma/\nu$. If $J_\nu>0$, minimality and the halving rule give
$\sigma_{J_\nu}>c_\gamma/(2\nu)$. If $J_\nu=0$, the same lower bound follows from
$\sigma_0\ge c_\gamma$. Hence
\[
 \frac{c_\gamma}{2\nu}<\sigma_{J_\nu}\le\frac{c_\gamma}{\nu}.
\]
Similarly, $b_\gamma\le1=\beta_0$ implies that $L_\nu$ is finite and
\[
 \frac{b_\gamma}{2\nu}<\beta_{L_\nu}\le\frac{b_\gamma}{\nu}.
\]
In particular, $\beta_{L_\nu}\ge b_\gamma/(4\nu)$, and therefore
Proposition~\ref{prop:ofas-guaranteed-search-window} shows that
$(\sigma_{J_\nu},\rho_0,\beta_{L_\nu})$ satisfies (C1)--(C4).

Set $M_\nu:=\max\{J_\nu,L_\nu\}$.
By construction, the staged fallback search has tested every pair $(j,\ell)$ with
$0\le j,\ell\le M_\nu$ by the end of stage $M_\nu$.
Hence it must accept a candidate no later than the test of $(J_\nu,L_\nu)$.
If $(j,\ell)$ is the accepted fallback pair, then $j\le M_\nu$, and therefore
$\sigma_k=\sigma_j\ge\sigma_{M_\nu}$.
If $M_\nu=J_\nu$, then
$\sigma_{M_\nu}=\sigma_{J_\nu}>c_\gamma/(2\nu)$.
If $M_\nu=L_\nu$, then
$\sigma_{M_\nu}=\sigma_{\max}2^{-L_\nu}=\sigma_{\max}\beta_{L_\nu}
>\sigma_{\max}b_\gamma/(2\nu)$.
Thus, in either case,
\[
 \sigma_k\ge\sigma_{M_\nu}>
 \frac{1}{2\nu}\min\{c_\gamma,\sigma_{\max}b_\gamma\}
 =\frac{c_\gamma^{\rm fb}}{\nu}.
\]
Finally,
{
\[
 J_\nu\le\left\lceil\log_2\frac{\sigma_{\max}\nu}{c_\gamma}\right\rceil,
 \qquad
 L_\nu\le\left\lceil\log_2\frac{\nu}{b_\gamma}\right\rceil.
\]
}
Hence $M_\nu=O(\log\nu)$, and at most $(M_\nu+1)^2=O((\log\nu)^2)$ fallback pairs are tested before termination.
\endgroup
\end{proof}

\subsection{Iteration Complexity of OFAS and CA-OFAS}
\label{subsec:ofas-complexity}

For the homogeneous progress analysis, we temporarily omit the stopping test based on $\mathcal Q_{\rm SDP}$ in Algorithms~\ref{alg:ofas}--\ref{alg:caofas}. We consider the theoretical sequence obtained by continuing the same update rule.
If the actual algorithm terminates earlier, the theoretical sequence need not be continued beyond that point.
In Subsection~\ref{subsec:normalization-solvable-case}, we use this sequence to derive an iteration bound for the actual algorithm by showing that the condition $\mathcal Q_{\rm SDP}\le\delta$ is satisfied after finitely many iterations.

Set
$c_\gamma^{\rm acc}:=\min\{\sigma^{\rm pre},c_\gamma^{\rm fb}\}$,
$\widehat c_\gamma:=c_\gamma^{\rm acc}/4$, and
$q_\gamma:=1-\widehat c_\gamma/\nu$, where $\sigma^{\rm pre}>0$ is the lower bound for the preliminary candidates in Section~\ref{sec:methods}.
By construction, $0<\widehat c_\gamma<\nu$, so $q_\gamma\in(0,1)$.

If a preliminary candidate is accepted, then $\sigma_k\ge\sigma^{\rm pre}$.
Otherwise, Theorem~\ref{thm:ofas-candidate-selection-termination} gives $\sigma_k>c_\gamma^{\rm fb}/\nu$.
Thus, in both cases,
\begin{equation}
 \sigma_k\ge\frac{c_\gamma^{\rm acc}}{\nu}.
 \label{eq:accepted-angle-uniform-lower-bound}
\end{equation}
Together with condition (C4), this bound gives a uniform contraction estimate.

\begin{theorem}[Well-definedness and iteration complexity of OFAS]
\label{thm:ofas-longstep}
Consider the theoretical sequence obtained by continuing the update rule of Algorithm~\ref{alg:ofas} from the standard initial point.
Every iteration is well defined, and, for all $k\ge0$,
\begin{equation*}
 W_k\in\Nlong(\gamma),\qquad
 \mathcal Q(W_k)\le q_\gamma^k\mathcal Q_0,\qquad
 \mathcal J(W_k,\omega_k)\le q_\gamma^k.
\end{equation*}
For $0<\varepsilon<\mathcal Q_0$, define
\begin{equation*}
 K_\varepsilon
 :=\left\lceil
 \frac{\log(\mathcal Q_0/\varepsilon)}{-\log q_\gamma}
 \right\rceil
 \le
 \left\lceil
 \frac{\nu}{\widehat c_\gamma}\log\frac{\mathcal Q_0}{\varepsilon}
 \right\rceil
\end{equation*}
Then $\mathcal Q(W_{K_\varepsilon})\le\varepsilon$.
Since $\gamma$, $\sigma^{\rm pre}$, $\sigma_{\max}$, and the initial scale $\mathcal Q_0$ are independent of $\varepsilon$, we obtain
$K_\varepsilon=O(\nu\log(1/\varepsilon))$ as $\varepsilon\downarrow0$.
\end{theorem}

\begin{proof}
At $k=0$, all complementarity eigenvalues of the standard initial point $W_0=(I,0,I,1,1)$ are equal to $1$.
Thus $W_0\in\Nlong(\gamma)$ and $\mu(W_0)=1$.
Since $\omega_0=1$, we also have $\theta(W_0,\omega_0)=1$ and $\mathcal J(W_0,\omega_0)=1$.

Suppose that $W_k\in\Nlong(\gamma)$ at iteration $k$.
If an admissible preliminary candidate exists, the algorithm accepts one.
Otherwise, Theorem~\ref{thm:ofas-candidate-selection-termination} guarantees that the successive-halving fallback search produces an admissible candidate.
Hence every iteration is well defined.
The accepted candidate satisfies (C1)--(C4). Condition (C3) and Lemma~\ref{lem:current-frame-certificate} imply $W_{k+1}\in\Nlong(\gamma)$, while condition (C4) and \eqref{eq:accepted-angle-uniform-lower-bound} give
\begin{align*}
 \mathcal Q(W_{k+1})
 &\le\left(1-\frac{\sigma_k}{4}\right)\mathcal Q(W_k)
 \le q_\gamma\mathcal Q(W_k),\\
 \mathcal J(W_{k+1},\omega_{k+1})
 &\le\left(1-\frac{\sigma_k}{4}\right)\mathcal J(W_k,\omega_k)
 \le q_\gamma\mathcal J(W_k,\omega_k).
\end{align*}
Induction then gives
$\mathcal Q(W_k)\le q_\gamma^k\mathcal Q_0$ and
$\mathcal J(W_k,\omega_k)\le q_\gamma^k$ for all $k$.

Finally, $q_\gamma=1-\widehat c_\gamma/\nu$ and $-\log(1-x)\ge x$ give
\[
 -\log q_\gamma\ge\frac{\widehat c_\gamma}{\nu},
 \qquad
 \frac1{-\log q_\gamma}\le\frac{\nu}{\widehat c_\gamma}.
\]
This proves the bound on $K_\varepsilon$.
\end{proof}

\begin{corollary}[Iteration-bound inheritance for CA-OFAS]
\label{cor:caofas-complexity-inheritance}
Consider the theoretical sequence obtained by continuing the update rule of Algorithm~\ref{alg:caofas} from the standard initial point.
Every iteration is well defined, and, for all $k\ge0$,
\begin{equation*}
 W_k\in\Nlong(\gamma),\qquad
 \mathcal Q(W_k)\le q_\gamma^k\mathcal Q_0,
 \qquad
 \mathcal J(W_k,\omega_k)\le q_\gamma^k.
\end{equation*}
Hence the iteration bound in Theorem~\ref{thm:ofas-longstep} also holds for \CAOFAS{}, and
$K_\varepsilon=O(\nu\log(1/\varepsilon))$ as $\varepsilon\downarrow0$.
\end{corollary}

\begin{proof}
At each iteration, \CAOFAS{} first applies the \OFAS{} candidate-selection rule and obtains an OFAS candidate $W_{\rm OFAS}^+$ satisfying (C1)--(C4).
If a preliminary candidate is accepted, then $\sigma_{\rm OFAS}\ge\sigma^{\rm pre}$; otherwise, Theorem~\ref{thm:ofas-candidate-selection-termination} gives $\sigma_{\rm OFAS}>c_\gamma^{\rm fb}/\nu$.
Thus the OFAS candidate satisfies the same uniform lower bound as in \eqref{eq:accepted-angle-uniform-lower-bound} and preserves the three inductive invariants used in Theorem~\ref{thm:ofas-longstep}.

An amplified candidate is accepted by Algorithm~\ref{alg:caofas} only if it also satisfies (C1)--(C4).
The amplification branch evaluates such a candidate only when $\sigma_{\rm CA}>\sigma_{\rm OFAS}$. Hence the uniform lower bound \eqref{eq:accepted-angle-uniform-lower-bound} for the OFAS candidate also gives $\sigma_{\rm CA}\ge c_\gamma^{\rm acc}/\nu$.
Hence Lemma~\ref{lem:current-frame-certificate} and condition (C4) show that the amplified candidate preserves the same three invariants and the same contraction factor $q_\gamma$.
If no amplified candidate is accepted, \CAOFAS{} uses the OFAS candidate $W_{\rm OFAS}^+$.
Therefore the induction in Theorem~\ref{thm:ofas-longstep} applies in either branch and gives the same iteration bound.
No additional guarantee is required concerning the admissibility of an amplified candidate, its maximum admissible angle, or the existence of an amplified angle larger than the OFAS angle.
\end{proof}

\subsection{Projective Normalization and Recovery of a Solution to the Original SDP}
\label{subsec:normalization-solvable-case}

For an HSD point with $\tau>0$, the standard normalization for homogeneous embeddings of SDP and general conic programs recovers the primal-dual variables of the original problem as $(X/\tau,y/\tau,S/\tau)$
\citep{deKlerkRoosTerlaky1997,GoulartChen2026}.
Here we derive the control of $\tau$ needed to translate the homogeneous progress measure $\mathcal Q(W)$ into a bound on the normalized KKT accuracy.

\begin{proposition}[Finite-accuracy HSD normalization]
\label{prop:quantitative-normalization-longstep}
Suppose that $W=(X,y,S,\tau,\kappa)$ is an interior HSD point with $\tau>0$ and $\mathcal Q(W)\le\varepsilon$.
Let $\bar X=X/\tau$, $\bar y=y/\tau$, and $\bar S=S/\tau$.
Then
\begin{align}
 \norm{\A(\bar X)-b}&\le\frac{\varepsilon}{\tau},&
 \norm{\Ast(\bar y)+\bar S-C}&\le\frac{\varepsilon}{\tau},
 \label{eq:normalized-feasibility-longstep}\\
 \bar X\bullet\bar S&\le\frac{\nu\varepsilon}{\tau^2},&
 |C\bullet\bar X-b^T\bar y|&\le\frac{\varepsilon}{\tau}+\frac{\nu\varepsilon}{\tau^2}.
 \label{eq:normalized-gap-longstep}
\end{align}
Consequently,
$\mathcal Q_{\rm SDP}(W)\le\varepsilon/\tau+\nu\varepsilon/\tau^2$.
\end{proposition}

\begin{proof}
Since $\mathcal Q(W)\le\varepsilon$, each of the three HSD residual norms is at most $\varepsilon$.
Dividing the primal and dual residual equations by $\tau$ gives \eqref{eq:normalized-feasibility-longstep}.
Also,
$X\bullet S+\tau\kappa=\nu\mu(W)\le\nu\varepsilon$, so
$\bar X\bullet\bar S\le\nu\varepsilon/\tau^2$.
{
Finally, the gap residual satisfies
\[
 C\bullet\bar X-b^T\bar y
 =\frac{C\bullet X-b^Ty}{\tau}
 =\frac{\Rg(W)-\kappa}{\tau}.
\]
Therefore,
\[
 |C\bullet\bar X-b^T\bar y|
 \le \frac{|\Rg(W)|}{\tau}+\frac{\kappa}{\tau}
 \le \frac{\varepsilon}{\tau}
      +\frac{\tau\kappa}{\tau^2}
 \le \frac{\varepsilon}{\tau}
      +\frac{\nu\varepsilon}{\tau^2},
\]
which proves the second inequality in \eqref{eq:normalized-gap-longstep}.
}
\end{proof}

For $\omega_k$ defined in Section~\ref{sec:methods}, the predictor arc multiplies all three HSD residuals by $1-\sigma_k$, whereas the correction directions leave them unchanged.
Hence
\begin{equation}
 \Hres(W_k)=\omega_k\Hres(W_0),
 \qquad
 \omega_k=\prod_{i=0}^{k-1}(1-\sigma_i).
 \label{eq:balanced-residual-scaling}
\end{equation}
For the remainder of this subsection, write
\[
 \mathcal Q_k:=\mathcal Q(W_k),\qquad
 \theta_k:=\frac{\mu(W_k)}{\omega_k},\qquad
 \mathcal R_0^{\max}:=
 \max\{\norm{\Rp(W_0)},\norm{\Rd(W_0)},|\Rg(W_0)|\}.
\]
At the standard initial point, $\mu(W_0)=1$. Thus the initial progress scale is
$\mathcal Q_0=\mathcal Q(W_0)=\max\{1,\mathcal R_0^{\max}\}$.
Therefore,
\begin{equation*}
 \mathcal Q_k=\omega_k\max(\theta_k,\mathcal R_0^{\max}).
\end{equation*}
By the definition of $\mathcal J$ and $\mathcal R_0^{\max}\le \mathcal Q_0$,
\begin{equation*}
 \max\left(\frac{\mathcal Q_k}{\theta_k},
           \frac{\mathcal Q_k}{\theta_k^2}\right)
 \le \mathcal Q_0\mathcal J(W_k,\omega_k)
 \le \mathcal Q_0q_\gamma^k.
\end{equation*}

Fix $Z^*=(X^*,y^*,S^*)\in{\mathcal Z^*}$ and define
\begin{equation}
 \delta_{\rm rec}(Z^*)
 :=\inf\left\{
 \begin{aligned}
 &\norm{\A(D_X)}+\norm{\Ast(d_y)+D_S}
 +S^*\bullet D_X+X^*\bullet D_S:\\
 &D_X,D_S\succeq O,\quad d_y\in\R^m,\quad \operatorname{tr}D_X+\operatorname{tr}D_S=1
 \end{aligned}\right\}.
 \label{eq:recession-margin-main}
\end{equation}
Appendix~\ref{subsec:appendix-projective-normalization} shows that
$\delta_{\rm rec}(Z^*)>0$.
Define also
\begin{align*}
 \mathfrak D_*:=1+\frac{\norm{b}+\norm{C}}{\delta_{\rm rec}(Z^*)}, \quad
 \mathfrak H_*:=\frac{3+\norm{y^*}+\norm{X^*}}{\delta_{\rm rec}(Z^*)}
       +1+\norm{y^*}+\norm{X^*}.
\end{align*}
Both are positive constants depending only on the fixed problem and $Z^*$.
The next lemma gives two scale estimates used to establish boundedness and accuracy after normalization.

\begin{lemma}[Projective scale estimates]
\label{lem:projective-scale-estimates-main}
Under Assumptions~\ref{ass:surjectivity}--\ref{ass:primal-dual-solvability}, for every $k$,
\begin{align}
 \operatorname{tr}X_k+\operatorname{tr}S_k
 &\le
 \frac{(\norm{b}+\norm{C})\tau_k
 +(3+\norm{y^*}+\norm{X^*})\mathcal Q_k}
 {\delta_{\rm rec}(Z^*)},
 \label{eq:projective-trace-bound-main}\\
 \tau_k
 &\ge\frac{\nu\theta_k-\mathfrak H_*\mathcal Q_k}{\mathfrak D_*},
 \qquad
 \frac{\mathcal Q_k}{\theta_k}\le\frac{\nu}{2\mathfrak H_*}
 \ \Longrightarrow\
 \tau_k\ge\frac{\nu\theta_k}{2\mathfrak D_*}.
 \label{eq:projective-scale-bound-main}
\end{align}
\end{lemma}
\begin{proof}
See Appendix~\ref{subsec:appendix-projective-normalization}.
\end{proof}

For a primal-dual triple $Z=(X,y,S)$, define the product-space norm
$\|Z\|_{\rm pd}:=(\|X\|_F^2+\|y\|^2+\|S\|_F^2)^{1/2}$.
The distance from $Z$ to the optimal solution set ${\mathcal Z^*}$ is
\begin{equation*}
 \operatorname{dist}_{\rm pd}(Z,{\mathcal Z^*})
 :=\inf_{Z^*\in{\mathcal Z^*}}\|Z-Z^*\|_{\rm pd}.
\end{equation*}

\begin{theorem}[Recovery of an original SDP solution and iteration bound after normalization]
\label{thm:normalized-distance-to-optimal-set}
Under Assumptions~\ref{ass:surjectivity}--\ref{ass:primal-dual-solvability}, consider the theoretical sequence of \OFAS{} or \CAOFAS{} defined in the preceding subsection.
Then
\begin{equation}
 \max\left(\frac{\mathcal Q_k}{\theta_k},\frac{\mathcal Q_k}{\theta_k^2}\right)
 \le \mathcal Q_0q_\gamma^k.
 \label{eq:projective-ratio-geometric-theorem}
\end{equation}
Let
$\bar X_k:=X_k/\tau_k$,
$\bar y_k:=y_k/\tau_k$,
$\bar S_k:=S_k/\tau_k$, and
$Z_k:=(\bar X_k,\bar y_k,\bar S_k)$.
Then $\{Z_k\}$ is bounded and
\begin{equation}
 \operatorname{dist}_{\rm pd}(Z_k,{\mathcal Z^*})\to0.
 \label{eq:distance-to-optimal-set}
\end{equation}
Moreover, for $\delta>0$, define
\begin{equation*}
 \Delta_\delta
 :=\min\left\{
 \mathcal Q_0,\frac{\nu}{2\mathfrak H_*},
 \frac{\delta\nu}{4\mathfrak D_*},
 \frac{\delta\nu}{8\mathfrak D_*^2}
 \right\}.
\end{equation*}
{
The four terms in $\Delta_\delta$ have separate roles. The first keeps the threshold no larger than the initial envelope $\mathcal Q_0$. For the remaining three, if $\mathcal Q_0q_\gamma^k\le\Delta_\delta$, then \eqref{eq:projective-ratio-geometric-theorem} gives
$\mathcal Q_k/\theta_k\le\Delta_\delta$ and $\mathcal Q_k/\theta_k^2\le\Delta_\delta$.
The bound $\Delta_\delta\le\nu/(2\mathfrak H_*)$ permits the second estimate in \eqref{eq:projective-scale-bound-main}, namely
$\tau_k\ge\nu\theta_k/(2\mathfrak D_*)$.
Consequently,
\[
 \frac{\mathcal Q_k}{\tau_k}
 \le\frac{2\mathfrak D_*}{\nu}\Delta_\delta,
 \qquad
 \frac{\nu\mathcal Q_k}{\tau_k^2}
 \le\frac{4\mathfrak D_*^2}{\nu}\Delta_\delta.
\]
The bounds $\Delta_\delta\le\delta\nu/(4\mathfrak D_*)$ and
$\Delta_\delta\le\delta\nu/(8\mathfrak D_*^2)$ make these two quantities at most $\delta/2$, respectively; Proposition~\ref{prop:quantitative-normalization-longstep} then gives $\mathcal Q_{\rm SDP}(W_k)\le\delta$.
}
Then, with
\begin{equation}
 K_\delta
 :=\left\lceil\frac{\log(\mathcal Q_0/\Delta_\delta)}{-\log q_\gamma}\right\rceil
 \le
 \left\lceil\frac{\nu}{\widehat c_\gamma}\log\frac{\mathcal Q_0}{\Delta_\delta}\right\rceil,
 \label{eq:conditioned-normalized-iteration-bound}
\end{equation}
the theoretical sequence satisfies $\mathcal Q_{\rm SDP}(W_k)\le\delta$ for every $k\ge K_\delta$.
Therefore, unless they terminate earlier, Algorithms~\ref{alg:ofas} and \ref{alg:caofas} satisfy their stopping criterion no later than iteration $K_\delta$.
Since $\mathcal Q_0$, $\mathfrak D_*$, and $\mathfrak H_*$ are independent of $\delta$, as $\delta\downarrow0$,
\begin{equation}
 K_\delta=O\!\left(\nu\log\frac1\delta\right).
 \label{eq:normalized-accuracy-complexity}
\end{equation}
In addition,
$C\bullet\bar X_k\to p^*$ and $b^T\bar y_k\to d^*=p^*$.
If ${\mathcal Z^*}$ is a singleton, the normalized iterates converge to that point.
\end{theorem}

\begin{proof}
Equation~\eqref{eq:projective-ratio-geometric-theorem} implies
$\mathcal Q_k/\theta_k\to0$ and $\mathcal Q_k/\theta_k^2\to0$.
Thus \eqref{eq:projective-scale-bound-main} applies for all sufficiently large $k$, and
\begin{align*}
 \frac{\mathcal Q_k}{\tau_k}
 &\le\frac{2\mathfrak D_*}{\nu}\frac{\mathcal Q_k}{\theta_k}\to0,
\\
 \frac{\nu \mathcal Q_k}{(\tau_k)^2}
 &\le\frac{4\mathfrak D_*^2}{\nu}\frac{\mathcal Q_k}{\theta_k^2}\to0.
\end{align*}
Proposition~\ref{prop:quantitative-normalization-longstep} then shows that the primal residual, dual residual, complementarity, and duality gap of the normalized point all converge to zero.

Next, dividing \eqref{eq:projective-trace-bound-main} by $\tau_k$ gives
\[
 \operatorname{tr}\bar X_k+\operatorname{tr}\bar S_k
 \le
 \frac{\norm{b}+\norm{C}}{\delta_{\rm rec}(Z^*)}
 +\frac{3+\norm{y^*}+\norm{X^*}}{\delta_{\rm rec}(Z^*)}
  \frac{\mathcal Q_k}{\tau_k}.
\]
Hence $\{\bar X_k\}$ and $\{\bar S_k\}$ are bounded.
Moreover,
$\Ast(\bar y_k)=C-\bar S_k+\Rd(W_k)/\tau_k$ has a bounded right-hand side.
Assumption~\ref{ass:surjectivity} makes $\Ast$ injective, so $\{\bar y_k\}$ is also bounded.
Thus $\{Z_k\}$ is bounded.

By closedness of the PSD cone and convergence of the KKT residuals, every accumulation point is primal-dual feasible and has zero duality gap.
Hence every accumulation point belongs to ${\mathcal Z^*}$.
A standard subsequence argument for bounded sequences then gives \eqref{eq:distance-to-optimal-set}.

Now suppose that $\mathcal Q_0q_\gamma^k\le\Delta_\delta$.
Equation~\eqref{eq:projective-ratio-geometric-theorem} and
$\Delta_\delta\le\nu/(2\mathfrak H_*)$ allow us to apply \eqref{eq:projective-scale-bound-main}, giving
\begin{align*}
 \frac{\mathcal Q_k}{\tau_k}
 &\le\frac{2\mathfrak D_*}{\nu}\Delta_\delta\le\frac\delta2,
 &
 \frac{\nu \mathcal Q_k}{(\tau_k)^2}
 &\le\frac{4\mathfrak D_*^2}{\nu}\Delta_\delta\le\frac\delta2.
\end{align*}
Proposition~\ref{prop:quantitative-normalization-longstep} then gives
$\mathcal Q_{\rm SDP}(W_k)\le\delta$.
The definition of $K_\delta$ in \eqref{eq:conditioned-normalized-iteration-bound} guarantees
$\mathcal Q_0q_\gamma^k\le\Delta_\delta$ for every $k\ge K_\delta$.
The upper bound in the same equation follows from
$q_\gamma=1-\widehat c_\gamma/\nu$ and $-\log(1-x)\ge x$.
For a fixed problem, $\Delta_\delta=\Theta(\delta)$ as $\delta\downarrow0$, which proves \eqref{eq:normalized-accuracy-complexity}.

Finally, by \eqref{eq:distance-to-optimal-set} and compactness of ${\mathcal Z^*}$, choose a nearest point $Z_k^*\in{\mathcal Z^*}$ to $Z_k$.
Then $\|Z_k-Z_k^*\|_{\rm pd}\to0$, which gives convergence of the objective values.
If ${\mathcal Z^*}$ is a singleton, the normalized iterates converge to its unique point.
\end{proof}

\begin{remark}[Multiple PSD blocks]
\label{rem:multiple-psd-block-extension}
For simplicity, this paper uses a single PSD block, but the same analysis extends to multiple PSD blocks through the standard product-cone treatment.
NT scaling and the Jordan-algebra operations are applied blockwise \citep{NesterovTodd1997,NesterovTodd1998}.
If the $J$ PSD blocks have sizes $n_1,\ldots,n_J$, the rank of the product cone is $\nu=1+\sum_{j=1}^J n_j$.
Hence the analysis in this section and the iteration bound $O(\nu\log(1/\varepsilon))$ extend directly to the multiple-block case.
\end{remark}

\section{Numerical Experiments}
\label{sec:numerical}

The numerical experiments address two questions. First, we examine the computational effects of the post-arc corrector, the one-factorization structure, and curvature amplification. Second, we compare the resulting \CAOFAS{} method with the standard Mehrotra-type predictor--corrector method in Clarabel.jl \citep{GoulartChen2026,Mehrotra1992} and a Liu-type method, namely, our HSD/NT implementation based on the wide-neighborhood second-order corrector method of Liu--Liu (2012) \citep{LiuLiu2012}. We use 54 test problems in total: 18 SDPLIB problems and 36 NNV-SDP problems. We report the number of successful problems including Low Accuracy, the number reaching Full Accuracy, computation times, outer iteration counts, and performance profiles.

\subsection{Experimental Setting and Compared Methods}

Our implementation is built on the HSD interior-point framework of Clarabel.jl. We compare six methods: \OFAS{}, \CAOFAS{}, Two-factorization, Corrector-free, Mehrotra, and the Liu-type method. Two-factorization follows essentially the same arc-search rule as \OFAS{}, but after selecting the OFAS arc point, it updates the NT scaling and refactorizes the KKT coefficient matrix at that point before computing the post-arc corrector. Corrector-free is a one-factorization reference method without an independent post-arc corrector, and its centering parameter is fixed at 0.10. For the overall comparison, we use Mehrotra, the standard method in Clarabel.jl, and an HSD/NT adaptation of the wide-neighborhood second-order corrector method of Liu--Liu (2012). We refer to the latter as the Liu-type method throughout this section.

The experiments were run with Julia 1.12.6 on a Linux node equipped with two 2.90 GHz Intel Xeon Gold 6226R processors and 192 GB RAM. Julia itself ran with one thread, whereas BLAS and OpenMP each used eight threads. We used CHOLMOD for direct linear solves and disabled chordal decomposition. The six methods shared the same high-accuracy gap and feasibility tolerances of $10^{-8}$ and a 3600-second limit per instance.

The computational experiments are designed to isolate the effect of the proposed step-selection strategy within a common implementation framework. All Clarabel-based methods therefore use the same linear-algebra backend and solver settings, and the reported computation times are intended for relative comparison among these methods.

We use MathOptInterface (MOI) termination statuses. Following the terminology used in the numerical evaluation of Clarabel \citep{GoulartChen2026}, we classify \texttt{OPTIMAL} as Full Accuracy and \texttt{ALMOST\_OPTIMAL} as Low Accuracy. Low Accuracy or better is regarded as success, while pairwise comparisons of computation time and outer iterations use only problems for which both methods reached Full Accuracy. Detailed termination criteria are given in Appendix~\ref{app:detailed-numerical-results}.

To reduce timing variability, all six methods were run three times on each eligible problem, and we report the median computation time and outer iteration count. Repetition was used only when all six methods succeeded in the first run, all computation times exceeded one second, and the maximum computation time was below 120 seconds; otherwise, only one run was used. We also prespecified that a problem would be excluded from the computation-time comparison if any successful method finished in one second or less. None of the final 54 problems met this exclusion condition.

The computation times and outer iteration counts in Tables~\ref{tab:sdplib-summary} and \ref{tab:nnvsdp-summary} are geometric means over the problems successfully solved by each method, including Low Accuracy results. These values summarize the overall behavior on each problem set. For pairwise speed comparisons without mixing different accuracy levels, we instead use geometric means of paired ratios over problems for which both methods reached Full Accuracy. In the following, $t_A/t_B$ and $k_A/k_B$ denote the geometric means of the computation-time ratio and outer-iteration ratio, respectively, over these matched Full Accuracy problems. In mathematical subscripts, $\mathrm{CA}$ denotes \CAOFAS{}. Problem-by-problem computation times, outer iteration counts, and termination states are given in Appendix~\ref{app:detailed-numerical-results}.

For \OFAS{} and \CAOFAS{}, we set the wide long-step neighborhood parameter to $\gamma=0.01$. For \CAOFAS{}, we additionally used $\eta=1.75$ and $\Delta_\sigma=0.10$ for all benchmark problems. The \CAOFAS{} amplification parameters were selected using a 12-problem NNV-SDP core set and then kept unchanged for the other reported problems.

The \OFAS{} and \CAOFAS{} implementations use the same candidate formulas and acceptance conditions (C1)--(C4) as in Section~\ref{sec:methods}. The theoretical algorithms use the successive-halving fallback search without an a priori limit in order to obtain the worst-case guarantee, whereas the numerical code uses a finite safeguarded version. The implementation details of this search are given in Appendix~\ref{app:detailed-numerical-results}.

\subsection{Test Problem Sets}
\label{subsec:test-problem-sets}

We use two types of semidefinite programming test problems. The first consists of 18 standard benchmark problems from SDPLIB \citep{Borchers1999}. These 18 problems were fixed before the reported benchmark runs by selecting problems whose preliminary computation times were between 1 and 3600 seconds.

The second consists of 36 NNV-SDP problems arising from neural network verification. The SDP construction and the ellipsoid and rectangle input sets are taken from Azuma--Kim--Yamashita \citep{AzumaKimYamashita2026}. Let $r=n_0=n_1$ denote the network width used to generate each problem. The order of the main PSD block is $n_{\rm PSD}=1+2r$. In SDPA standard form, the number of equality constraints is $m=3r+2$ for the ellipsoid family and $m=4r+1$ for the rectangle family. We use both dense and identity settings for the weight matrices. The ranges of $(n_{\rm PSD},m)$ are $(81,122)$--$(113,170)$ for ellipsoid/dense, $(65,129)$--$(97,193)$ for rectangle/dense, and $(81,161)$--$(129,257)$ for rectangle/identity. Problem-by-problem specifications and results are reported in Appendix~\ref{app:detailed-numerical-results}.

\subsection{Numerical Comparison of Algorithmic Components}
\label{subsec:component-comparison}

Table~\ref{tab:component-paired-comparisons} compares computation time and outer iterations on the instances where both methods attained Full Accuracy in each component experiment. On these identical instance sets, we additionally record KKT factorization counts, linear-solve counts, and the time used by those two linear-algebra operations.

We first compare \OFAS{} with Corrector-free to examine the role of the post-arc corrector. The corrector requires an additional linear-system solve in each outer iteration. On both problem sets, \OFAS{} consequently used more linear-system solves than Corrector-free, but required fewer outer iterations and KKT factorizations.

For the matched Full Accuracy problems, these reductions were accompanied by lower total computation times. The geometric-mean computation time of \OFAS{} was about 37\% lower on SDPLIB and about 42\% lower on NNV-SDP; \OFAS{} was faster on 10 of the 11 matched SDPLIB problems and on all 25 matched NNV-SDP problems. Thus, on the tested problems, the additional solve per outer iteration did not increase the overall computational cost, while the number of outer iterations was reduced.

We next compare \OFAS{} with Two-factorization to evaluate the one-factorization structure. Although \OFAS{} required slightly more outer iterations on SDPLIB and more outer iterations on NNV-SDP, it was faster on every matched problem in both problem sets. The geometric-mean computation time was reduced by about 44\% on SDPLIB and about 38\% on NNV-SDP.

This difference is also reflected in the linear-algebra operations. Relative to Two-factorization, \OFAS{} reduced the number of KKT factorizations and the factorization time by about 48\% on SDPLIB and about 40\% on NNV-SDP. The number of linear-system solves was about 13\% smaller on SDPLIB and nearly unchanged on NNV-SDP. These results support the computational benefit of reusing a single KKT factorization for the post-arc corrector.

Finally, we compare \CAOFAS{} with \OFAS{} to assess the additional effect of curvature amplification. On SDPLIB, the change was modest: the geometric-mean ratios for outer iterations and computation time were 0.964 and 0.963, respectively. On NNV-SDP, the corresponding ratios were 0.808 and 0.811, showing reductions of about 19\% in both measures. \CAOFAS{} was faster on 7 of the 11 matched SDPLIB problems and on 28 of the 31 matched NNV-SDP problems. The similar iteration and time ratios on NNV-SDP suggest that most of the observed reduction in computation time came from fewer outer iterations, while the average cost per outer iteration remained similar.

\begin{table}[t]
\centering
\caption{Paired comparisons on problems solved to Full Accuracy by both methods}
\label{tab:component-paired-comparisons}
\small
\begin{tabular}{llrrr}
\toprule
Problem set & Comparison $A/B$ & Matched & $t_A/t_B$ & $k_A/k_B$ \\
\midrule
SDPLIB 18 & \OFAS/Corrector-free & 11 & 0.629 & 0.597 \\
 & \OFAS/Two-factorization & 11 & 0.557 & 1.048 \\
 & \CAOFAS/\OFAS & 11 & 0.963 & 0.964 \\
\addlinespace
NNV-SDP 36 & \OFAS/Corrector-free & 25 & 0.575 & 0.548 \\
 & \OFAS/Two-factorization & 26 & 0.623 & 1.202 \\
 & \CAOFAS/\OFAS & 31 & 0.811 & 0.808 \\
\bottomrule
\end{tabular}
\begin{minipage}{0.96\linewidth}
\footnotesize Matched gives the number of problems for which both methods reached Full Accuracy. $t_A/t_B$ and $k_A/k_B$ are the geometric means of the computation-time ratio and iteration-count ratio, respectively, over these problems. A ratio below one favors the method in the numerator.
\end{minipage}
\end{table}
\FloatBarrier

\subsection{Comparison with Existing Methods}
\label{subsec:existing-method-comparison}

We next compare \CAOFAS{} with Mehrotra and the Liu-type method to assess its overall performance relative to these methods. Table~\ref{tab:existing-paired-comparisons} reports the paired computation-time and outer-iteration ratios on problems for which both methods reached Full Accuracy.

\begin{table}[htbp]
\centering
\caption{Paired comparisons of \CAOFAS{} with Mehrotra and the Liu-type method on problems for which both methods reached Full Accuracy}
\label{tab:existing-paired-comparisons}
\small
\begin{tabular}{llrrr}
\toprule
Problem set & Comparison $A/B$ & Matched & $t_A/t_B$ & $k_A/k_B$ \\
\midrule
SDPLIB 18 & \CAOFAS/Mehrotra & 14 & 0.830 & 0.804 \\
 & \CAOFAS/Liu-type & 14 & 0.797 & 0.766 \\
\addlinespace
NNV-SDP 36 & \CAOFAS/Mehrotra & 31 & 0.869 & 0.848 \\
 & \CAOFAS/Liu-type & 33 & 0.882 & 0.864 \\
\bottomrule
\end{tabular}
\begin{minipage}{0.96\linewidth}
\footnotesize Matched gives the number of problems for which both methods reached Full Accuracy. $t_A/t_B$ and $k_A/k_B$ are the geometric means of the computation-time ratio and iteration-count ratio, respectively, over these problems. A ratio below one favors \CAOFAS{}.
\end{minipage}
\end{table}
\FloatBarrier

\subsubsection{SDPLIB}

Table~\ref{tab:sdplib-summary} summarizes the results for \CAOFAS{}, Mehrotra, and the Liu-type method. All three methods succeeded on all 18 SDPLIB problems, so their geometric means of computation time and outer iteration count are based on the same problem set. \CAOFAS{} reached Full Accuracy on 16 problems, while Mehrotra and the Liu-type method each reached Full Accuracy on 15. The geometric-mean computation time was 151.6 seconds for \CAOFAS{}, compared with 178.7 seconds for Mehrotra and 249.9 seconds for the Liu-type method. On the matched Full Accuracy problems, \CAOFAS{} reduced the geometric-mean outer iteration count by about 20\% relative to Mehrotra and about 23\% relative to the Liu-type method, with corresponding computation-time reductions of about 17\% and 20\%, respectively. It was also faster than the Liu-type method on all 14 matched problems. Thus, for these SDPLIB problems, the lower computation times of \CAOFAS{} are accompanied by fewer outer iterations.

\begin{table}[t]
\centering
\caption{Aggregated results for 18 SDPLIB problems}
\label{tab:sdplib-summary}
\small
\begin{tabular}{lrrrr}
\toprule
Method & Success & Full Accuracy & Time (s) & Iterations \\
\midrule
\CAOFAS & 18/18 & 16/18 & 151.6 & 16.21 \\
Mehrotra & 18/18 & 15/18 & 178.7 & 20.06 \\
Liu-type & 18/18 & 15/18 & 249.9 & 28.44 \\
\bottomrule
\end{tabular}
\begin{minipage}{0.92\linewidth}
\footnotesize Success includes Low Accuracy, and Full Accuracy gives the number of problems with status \texttt{OPTIMAL}. Time and Iterations are geometric means over the problems successfully solved by each method, including Low Accuracy results. For SDPLIB, the three methods have the same success set, so these geometric means are based on the same 18 problems.
\end{minipage}
\end{table}

\subsubsection{NNV-SDP}

Table~\ref{tab:nnvsdp-summary} summarizes the results for the 36 NNV-SDP problems. The Liu-type method succeeded on 35 of the 36 problems, compared with 34 for both \CAOFAS{} and Mehrotra. The geometric-mean computation time over each method's successful problems was 605.7 seconds for \CAOFAS{}, 734.7 seconds for Mehrotra, and 709.9 seconds for the Liu-type method; because the success sets are not identical, these aggregate times are descriptive and are not used for direct pairwise comparison. On the matched Full Accuracy problems, \CAOFAS{} reduced the geometric-mean outer iteration count by about 15\% relative to Mehrotra and about 14\% relative to the Liu-type method. The corresponding computation-time reductions were about 13\% and 12\%, and \CAOFAS{} was faster on every matched problem in both comparisons. The similar reductions in iteration count and computation time indicate that the time advantage on the matched NNV-SDP problems is closely associated with the smaller number of outer iterations.

\begin{table}[t]
\centering
\caption{Aggregated results for 36 NNV-SDP problems}
\label{tab:nnvsdp-summary}
\small
\begin{tabular}{lrrrr}
\toprule
Method & Success & Full Accuracy & Time (s) & Iterations \\
\midrule
\CAOFAS & 34/36 & 33/36 & 605.7 & 27.37 \\
Mehrotra & 34/36 & 33/36 & 734.7 & 32.50 \\
Liu-type & 35/36 & 35/36 & 709.9 & 32.15 \\
\bottomrule
\end{tabular}
\begin{minipage}{0.92\linewidth}
\footnotesize Success includes Low Accuracy, and Full Accuracy gives the number of problems with status \texttt{OPTIMAL}. Time and Iterations are geometric means over the problems successfully solved by each method, including Low Accuracy results. Pairwise comparisons with Mehrotra and the Liu-type method use the matched Full Accuracy ratios in Table~\ref{tab:existing-paired-comparisons}.
\end{minipage}
\end{table}

\subsection{Performance Profiles for Iteration Count and Computation Time}
\label{subsec:performance-profiles}

{We complement the paired tables with performance profiles in the sense of Dolan--Mor\'e \citep{DolanMore2002}. An earlier ratio-to-best display of Tits and Yang \citep[pp.~1442--1443]{TitsYang1996} plotted cumulative counts against a logarithmic performance-ratio axis. Dividing those counts by the number of test instances yields the same empirical cumulative form as a performance profile.} The paired tables compare identical Full Accuracy instance sets directly; the profiles instead summarize iteration and time ratios across the entire test set and also incorporate Low Accuracy successes. All six implementations are included. Let $\mathcal I$ be the problem set, and let $\mathcal S_p$ be the set of methods that reached Low Accuracy or better on problem $p\in\mathcal I$. Let $k_{p,s}$ and $t_{p,s}$ denote the outer iteration count and computation time, respectively, for method $s$. When $\mathcal S_p\neq\varnothing$, define
\[
 r_{p,s}^{\rm iter}=\begin{cases}
 \displaystyle\frac{k_{p,s}}{\min_{s'\in\mathcal S_p}k_{p,s'}}, & s\in\mathcal S_p,\\[2mm]
 +\infty, & s\notin\mathcal S_p,
 \end{cases}
 \qquad
 r_{p,s}^{\rm time}=\begin{cases}
 \displaystyle\frac{t_{p,s}}{\min_{s'\in\mathcal S_p}t_{p,s'}}, & s\in\mathcal S_p,\\[2mm]
 +\infty, & s\notin\mathcal S_p.
 \end{cases}
\]
If all methods fail on a problem, we set all performance ratios to $+\infty$ and still include that problem in $|\mathcal I|$. For $q\in\{\mathrm{iter},\mathrm{time}\}$, the performance profile is
\[
 \varrho_s^q(\vartheta)=\frac{1}{|\mathcal I|}\bigl|\{p\in\mathcal I:r_{p,s}^q\le\vartheta\}\bigr|,\qquad \vartheta\ge1.
\]
Hence $\varrho_s^q(1)$ gives the share of instances where method $s$ attains the best observed value of measure $q$. As $\vartheta$ becomes large, the profile approaches the success fraction of that method.

Figure~\ref{fig:performance-profile-sdplib} shows the profiles for the 18 SDPLIB problems. For \CAOFAS{}, the time profile is more concentrated near a performance ratio of one than the iteration profile, indicating that differences in outer iteration count do not translate directly into the same differences in total computation time. \CAOFAS{} was fastest on 9 of the 18 problems and required no more than 1.11 times the best time on every problem. Thus, the profile complements the paired comparisons by showing that its computation time remains close to the best observed performance throughout the SDPLIB set.

\begin{figure}[H]
\centering
\includegraphics[width=0.98\linewidth]{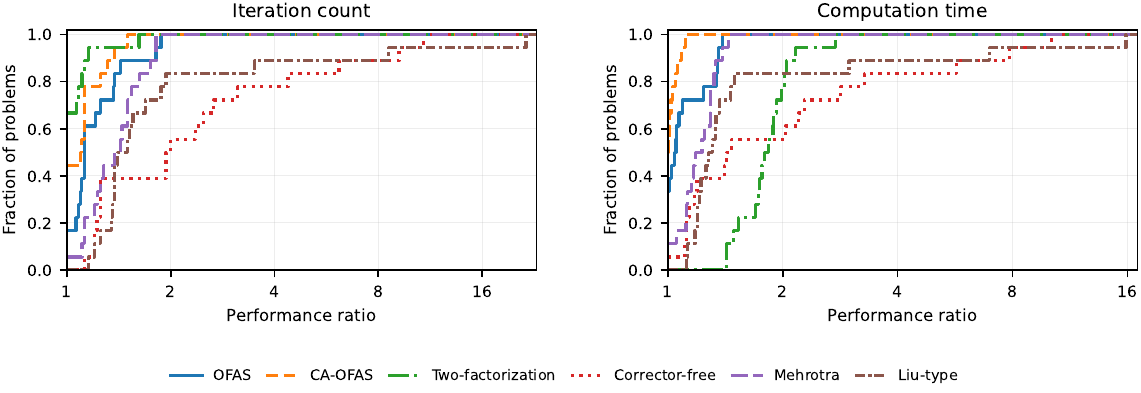}
\caption{Performance profiles for the six implementations on 18 SDPLIB problems (success includes Low Accuracy; left: outer iteration count; right: computation time; the performance ratio is $+\infty$ for unsuccessful problems; logarithmic horizontal axis)}
\label{fig:performance-profile-sdplib}
\end{figure}

Figure~\ref{fig:performance-profile-nnvsdp} shows the profiles for the 36 NNV-SDP problems. \CAOFAS{} was fastest on 31 problems and required no more than 1.03 times the best time on all 34 problems that it successfully solved; its iteration profile is likewise concentrated near one. The Liu-type method succeeded on 35 problems, compared with 34 for \CAOFAS{}, and therefore has a slightly higher limiting profile value. These profiles complement the aggregate results by showing that \CAOFAS{} remains close to the best observed performance on its successful problems, while the Liu-type method succeeds on one additional problem.

\begin{figure}[H]
\centering
\includegraphics[width=0.98\linewidth]{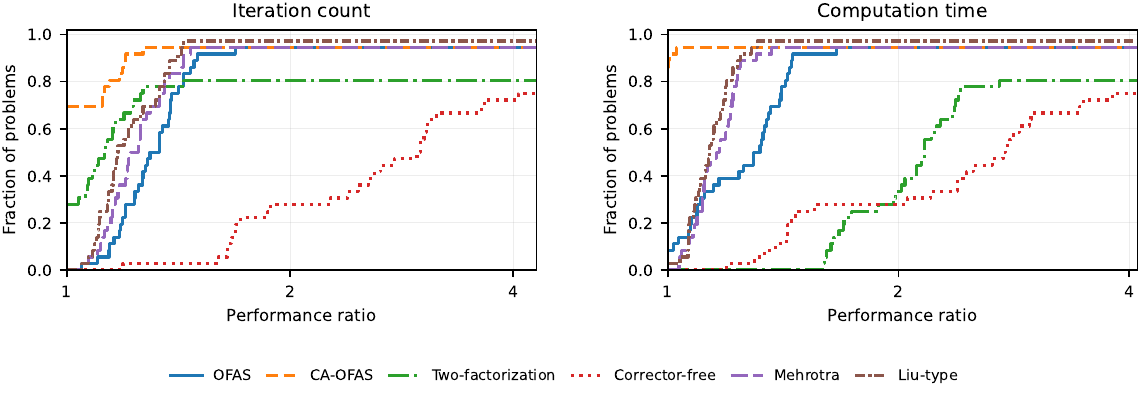}
\caption{Performance profiles for the six implementations on 36 NNV-SDP problems (success includes Low Accuracy; left: outer iteration count; right: computation time; the performance ratio is $+\infty$ for unsuccessful problems; logarithmic horizontal axis)}
\label{fig:performance-profile-nnvsdp}
\end{figure}

The paired comparisons clarify the contributions of the proposed design. \OFAS{} already provides substantial reductions in computation time over the structural reference methods: 44\% and 38\% relative to Two-factorization, and 37\% and 42\% relative to Corrector-free, on SDPLIB and NNV-SDP, respectively. Furthermore, \CAOFAS{} reduces computation time relative to \OFAS{} by an additional 4\% on SDPLIB and 19\% on NNV-SDP through curvature amplification. \CAOFAS{} also improves on the Mehrotra-type and Liu-type methods: its computation time is reduced by 17\% and 13\% relative to Mehrotra, and by 20\% and 12\% relative to Liu, on SDPLIB and NNV-SDP, respectively. These results show substantial gains from the one-factorization structure and post-arc correction, while curvature amplification provides a further improvement, particularly on NNV-SDP.
\FloatBarrier

\section{Conclusion}
\label{sec:conclusion}

This work introduced \OFAS{}, a one-factorization arc-search method that keeps both the post-arc correction and flexible candidate selection without refactorizing the KKT matrix at the accepted point.
The method and its analysis are formulated within an HSD embedding.
We also introduced \CAOFAS{}, which selectively amplifies the arc term associated with the second-order direction while keeping the same one-factorization structure.
If an amplified candidate is not accepted, \CAOFAS{} falls back to an \OFAS{} candidate and therefore inherits the same theoretical guarantees as \OFAS{}.

For the theoretical analysis, we derived direction and centering estimates in a wide long-step neighborhood.
Based on these estimates, we showed that the guaranteed candidate-selection rule finds an admissible candidate after finitely many halvings of the angle and damping parameter.
We also established an iteration bound of $O(\nu\log(1/\varepsilon))$ for both \OFAS{} and \CAOFAS{}.
Under Assumption~\ref{ass:primal-dual-solvability}, the distance from the projectively normalized HSD sequence to the primal--dual optimal set was also shown to converge to zero.
This analysis also gave a relative lower bound on the HSD scaling variable $\tau_k$ with respect to the homogeneous scale $\theta_k$.

In the numerical experiments, we compared six methods on SDPLIB and NNV-SDP.
Relative to Two-factorization, \OFAS{} reduced both the KKT factorization count and factorization time by about 48\% on SDPLIB and 40\% on NNV-SDP. \CAOFAS{} further improved upon \OFAS{} for many of the tested problems. These results indicate practical benefits from factorization reuse, post-arc correction, and curvature amplification.

Future work includes adaptive curvature amplification based on local curvature information, sparse SDP processing based on chordal decomposition and positive definite matrix completion, and evaluation of these ideas in high-performance implementations
\citep{FukudaKojimaMurotaNakata2001,NakataFujisawaFukudaKojimaMurota2003,YamashitaNakata2015}.
Another direction is to study the relationship between the conditioning of the HSD KKT systems and practical computational performance.

\appendix
\renewcommand{\thetheorem}{\Alph{section}.\arabic{theorem}}
\section{Proofs of Technical Lemmas}
\label{app:technical-proofs}

\subsection{Nonsingularity of the Current-Iterate KKT Operator}
\label{app:kkt-nonsingularity}

We prove Lemma~\ref{lem:current-kkt-nonsingular-longstep}.

\begin{proof}
Fix an arbitrary interior HSD point and denote it by $W_k$.
We show that the homogeneous system
\(\mathcal L(W_k)[D]=(O_{\mathcal F},\bar O), \qquad D=(D_X,d_y,D_S,d_\tau,d_\kappa)\)
has only the zero solution $D=0$.
For brevity, set
\(P:=P(D;W_k),\qquad R:=R(D;W_k).\)
Since $P(W_k;W_k)=R(W_k;W_k)=Q_k^{\rm NT}$, the complementarity row is
\(\mathcal B(W_k,D;W_k)=Q_k^{\rm NT}\circ(P+R)=\bar O.\)
Since $Q_k^{\rm NT}\in\Ssym_{++}^n\times\R_{++}$, the Jordan multiplication operator
$L_{Q_k^{\rm NT}}:U\mapsto Q_k^{\rm NT}\circ U$ is invertible.
Indeed, let $q_1,\ldots,q_n>0$ be the eigenvalues of the PSD block $Q_{X,k}$. After an orthogonal diagonalization of $Q_{X,k}$, $L_{Q_k^{\rm NT}}$ multiplies each matrix component by the positive scalar $(q_i+q_j)/2$. On the scalar block, it also acts by multiplication with a positive scalar.
Therefore,
\begin{equation}
 P+R=\bar O.
 \label{eq:kkt-kernel-p-plus-r-zero}
\end{equation}

The three HSD residual rows give
\begin{align*}
 \A(D_X)-bd_\tau&=0,\\
 \Ast(d_y)+D_S-Cd_\tau&=0,\\
 C\bullet D_X-b^Td_y+d_\kappa&=0.
\end{align*}
The second equation gives $D_S=Cd_\tau-\Ast(d_y)$. Hence, using the first equation and the adjoint relation,
\begin{align*}
 D_X\bullet D_S
 &=d_\tau C\bullet D_X-D_X\bullet\Ast(d_y)
 =d_\tau C\bullet D_X-\A(D_X)^Td_y\\
 &=d_\tau\{C\bullet D_X-b^Td_y\}
 =-d_\tau d_\kappa.
\end{align*}
Thus, $D_X\bullet D_S+d_\tau d_\kappa=0$.
Moreover, by the definition of the NT coordinates in \eqref{eq:direction-nt-components} and cyclicity of the trace,
\begin{align*}
 \langle P,R\rangle_{\mathcal E}
 &=\operatorname{tr}\!\left(
   M_kD_XM_k^T M_k^{-T}D_SM_k^{-1}\right)
   +(\xi_kd_\tau)(\xi_k^{-1}d_\kappa)\\
 &=D_X\bullet D_S+d_\tau d_\kappa
 =0.
\end{align*}
Equation~\eqref{eq:kkt-kernel-p-plus-r-zero} gives $R=-P$, so
\(0=\langle P,R\rangle_{\mathcal E}=-\norm{P}_{\mathcal E}^2.\)
Hence, $P=R=\bar O$.
Since $M_k$ and $\xi_k$ are invertible, we obtain
$D_X=D_S=O$ and $d_\tau=d_\kappa=0$.
The dual residual row then gives $\Ast(d_y)=O$.
By Assumption~\ref{ass:surjectivity}, $\Ast$ is injective, and therefore $d_y=0$.
Thus, $\ker\mathcal L(W_k)=\{0\}$.
The current-iterate KKT system is finite-dimensional and square. Hence, injectivity implies bijectivity, and $\mathcal L(W_k)$ is nonsingular.
\end{proof}

\subsection{Technical Estimates for Uniformly Guaranteed OFAS Candidates}
\label{app:guaranteed-ofas-candidate}

This subsection gives the technical estimates for the fixed choice $\rho_0=1/8$ used in Proposition~\ref{prop:ofas-guaranteed-search-window}.
We separate the constant estimates from the main argument and verify three facts: (i) $D_\zeta$ is uniformly bounded away from zero, (ii) strict interiority is preserved for $\beta=\Theta(1/\nu)$, and (iii) the wide long-step neighborhood is recovered on the same damping scale.
The constants below depend only on $\gamma$, the fixed value $\rho_0$, and the prescribed upper bound on the angle. They are independent of $\nu$ and the problem data.

The candidate set in Algorithm~\ref{alg:ofas} contains
$\rho\in\{1,1/2,1/4,1/8\}$.
The choice $\rho_0=1/8$ is used only to construct a theoretical guarantee and does not exclude the other values of $\rho$.
For the finite-termination proof, we fix the smallest value $\rho_0$ and show that the successive-halving fallback search reaches an admissible candidate when $\beta$ is on the $O(1/\nu)$ scale, while retaining an $O(1/\nu)$ angle.
This search requires no additional KKT solve. It uses only linear combinations of the correction directions already constructed and the unit-complementarity direction.

Throughout this subsection, fix iteration $k$ and use the abbreviations
\[
 \begin{aligned}
 W&:=W_k,& \mu&:=\mu(W_k),& Q&:=Q_k^{\rm NT},& H&:=Q\circ Q,\\
 \sigma&:=\sin\alpha,& \psi&:=1-\cos\alpha,
 &\mu_{\rm tar}&:=(1-\sigma)\mu.
 \end{aligned}
\]
In the NT coordinates defined by the current iterate $W_k$, write the predictor-arc factors and their displacements as
\[
 \begin{aligned}
 P(\alpha)&:=P(W_k(\alpha);W_k),&
 R(\alpha)&:=R(W_k(\alpha);W_k),\\
 \Delta P(\alpha)&:=P(\alpha)-Q,&
 \Delta R(\alpha)&:=R(\alpha)-Q.
 \end{aligned}
\]
We also write
$\Delta_{\rm arc}(\alpha):=W_k(\alpha)-W_k$ and
$H_{\rm arc}(\alpha):=P(\alpha)\circ R(\alpha)$.
The quantities $D_\zeta$, $N_\zeta$, and $\zeta$ are defined as in Lemma~\ref{lem:ofas-exact-update-identities}(iii) with $\rho=\rho_0$. We use the notation $d_\gamma(c)$, $r_\gamma(c)$, and $\mathcal E_{\rm arc}(\alpha)$ from Lemma~\ref{lem:wide-arc-remainder}.

Let $0<c\le1/2$ and assume $\sigma\le c/\nu$.
{For the correction direction $D_k^{\rm cor}(\alpha,\rho_0)$, define its NT components by}
\[
 (P^{\rm cor},R^{\rm cor})
 :=\bigl(P(D_k^{\rm cor}(\alpha,\rho_0);W_k),
 R(D_k^{\rm cor}(\alpha,\rho_0);W_k)\bigr).
\]
{By the definition in Section~\ref{sec:methods}, this correction direction is}
\[
 D_k^{\rm cor}(\alpha,\rho_0)
 =\rho_0(1-\sigma)D_k^{\rm cen}
  +\sigma\psi D_k^{\rm mix}
  +\psi^2D_k^{\rm rem}.
\]
Therefore, using the three squared-norm bounds in Proposition~\ref{prop:ofas-nt-scaled-bounds},
\[
 \sqrt{\norm{P_k^{\rm cen}}^2+\norm{R_k^{\rm cen}}^2}
 \le \sqrt{\nu(\gamma^{-1}-1)\mu},
\]
\[
 \sqrt{\norm{P_k^{\rm mix}}^2+\norm{R_k^{\rm mix}}^2}
 \le \frac{\nu^{3/2}}{\gamma}\sqrt\mu,
 \qquad
 \sqrt{\norm{P_k^{\rm rem}}^2+\norm{R_k^{\rm rem}}^2}
 \le \frac{1+\gamma}{2\gamma^{3/2}}\nu^2\sqrt\mu,
\]
and the inequality $\norm P+\norm R\le\sqrt2\sqrt{\norm P^2+\norm R^2}$, we obtain
\begin{align*}
 \norm{P^{\rm cor}}+\norm{R^{\rm cor}}
 &\le \sqrt2\left(
 \rho_0(1-\sigma)\sqrt{\nu(\gamma^{-1}-1)\mu}
 +\sigma\psi\frac{\nu^{3/2}}\gamma\sqrt\mu
 +\psi^2\frac{1+\gamma}{2\gamma^{3/2}}\nu^2\sqrt\mu
 \right).
\end{align*}
Since $0\le\psi=1-\cos\alpha\le\sin^2\alpha=\sigma^2$, $\sigma\le c/\nu$, and $\nu\ge1$, it follows that
\begin{equation*}
 \norm{P^{\rm cor}}+\norm{R^{\rm cor}}
 \le A_\gamma(c)\sqrt{\nu\mu},
\end{equation*}
where
\begin{equation*}
 A_\gamma(c)
 :=\sqrt2\left(
   \rho_0\sqrt{\gamma^{-1}-1}
   +\frac{c^3}{\gamma}
   +\frac{1+\gamma}{2\gamma^{3/2}}c^4
 \right).
\end{equation*}
{This gives the required bound for the correction direction $D_k^{\rm cor}(\alpha,\rho_0)$.}

Let $(P_k^{\rm unit},R_k^{\rm unit})$ be the NT components of the unit-complementarity direction $D_k^{\rm unit}$.
{Define also the complementarity-row target for $D_k^{\rm cor}(\alpha,\rho_0)$ by}
$T_0(\alpha):=(1-\sigma)\bigl((1-\rho_0)H+\rho_0\mu\bar I\bigr)$.
This is exactly $T_k(\alpha,\rho)$ from Section~\ref{sec:methods} with $\rho=\rho_0$.
The following exact cancellations show that $D_\zeta$ remains close to $\nu$ and that $\zeta$ is small.

\begin{lemma}[Cancellation in $\zeta$ for the fixed-$\rho_0$ corrector]
\label{lem:fixed-rho-zeta-cancellation}
For $D_k^{\rm cor}(\alpha,\rho_0)$, the following identities hold exactly:
\begin{align*}
 D_\zeta
 &=\nu+
   \langle\Delta P(\alpha),R_k^{\rm unit}\rangle
   +\langle P_k^{\rm unit},\Delta R(\alpha)\rangle,\\
 N_\zeta
 &=-\langle\Delta P(\alpha),R^{\rm cor}\rangle
   -\langle P^{\rm cor},\Delta R(\alpha)\rangle.
\end{align*}
\end{lemma}

\begin{proof}
{For the scalar denominator $D_\zeta$ defined in Lemma~\ref{lem:ofas-exact-update-identities}(iii), the unit-complementarity direction gives $\bar I$ in the complementarity row in the current-iterate NT coordinates.} Hence, its cross complementarity with $W_k$ is
\(\Gamma(W,D_k^{\rm unit})=\operatorname{Tr}_{\mathcal E}(\bar I)=\nu.\)
Using the arc displacement $\Delta_{\rm arc}(\alpha)$ defined above and bilinearity, we obtain
\begin{align*}
 D_\zeta
 &=\Gamma(W_k(\alpha),D_k^{\rm unit})
 =\Gamma(W,D_k^{\rm unit})+\Gamma(\Delta_{\rm arc}(\alpha),D_k^{\rm unit})\\
 &=\nu+\langle\Delta P(\alpha),R_k^{\rm unit}\rangle
 +\langle P_k^{\rm unit},\Delta R(\alpha)\rangle.
\end{align*}
Thus, only the two arc-displacement cross terms perturb the denominator from $\nu$.

{
By \eqref{eq:current-corrector-solution-map}, the linearity of $\mathcal G(\cdot;W)$, and \eqref{eq:corrector-rhs-basis-paper},
\begin{align*}
 \mathcal B(W,D_k^{\rm cor}(\alpha,\rho_0);W)
 &=-\Delta_k^{\rm comp}(\alpha,\rho_0)\\
 &=T_0(\alpha)-\Phi(W_k(\alpha);W).
\end{align*}
Taking $\operatorname{Tr}_{\mathcal E}$ and using
$\operatorname{Tr}_{\mathcal E}(\Phi(V;W))=\mathcal C(V)$,
$\operatorname{Tr}_{\mathcal E}(\mathcal B(U,V;W))=\Gamma(U,V)$, and
$\operatorname{Tr}_{\mathcal E}(T_0(\alpha))=\nu\mu_{\rm tar}$ gives
\[
 \Gamma(W,D_k^{\rm cor}(\alpha,\rho_0))
 =\nu\mu_{\rm tar}-\mathcal C(W_k(\alpha)).
\]
}On the other hand, the exact numerator of $\zeta$ is
\begin{align*}
 N_\zeta
 &=\nu\mu_{\rm tar}-\mathcal C(W_k(\alpha))
 -\Gamma(W_k(\alpha),D_k^{\rm cor}(\alpha,\rho_0))\\
 &=\nu\mu_{\rm tar}-\mathcal C(W_k(\alpha))
 -\Gamma(W,D_k^{\rm cor}(\alpha,\rho_0))
 -\Gamma(\Delta_{\rm arc}(\alpha),D_k^{\rm cor}(\alpha,\rho_0)).
\end{align*}
The first three terms cancel exactly by the identity above, and hence
\(N_\zeta=-\Gamma(\Delta_{\rm arc}(\alpha),D_k^{\rm cor}(\alpha,\rho_0)).\)
Expanding this cross complementarity in the current-iterate NT coordinates gives
\[
 \Gamma(\Delta_{\rm arc}(\alpha),D_k^{\rm cor}(\alpha,\rho_0))
 =\langle\Delta P(\alpha),R^{\rm cor}\rangle
  +\langle P^{\rm cor},\Delta R(\alpha)\rangle.
\]
Therefore,
\[
 N_\zeta
 =-\langle\Delta P(\alpha),R^{\rm cor}\rangle
  -\langle P^{\rm cor},\Delta R(\alpha)\rangle.
\]
\end{proof}

We now use the following bounds for the arc displacement and the unit-complementarity direction:
\[
 \norm{\Delta P(\alpha)}+\norm{\Delta R(\alpha)}
 \le d_\gamma(c)\sqrt{\frac{\mu}{\nu}},
 \qquad
 \norm{P_k^{\rm unit}}+\norm{R_k^{\rm unit}}
 \le \sqrt{\frac{2\nu}{\gamma\mu}}.
\]
Applying Cauchy--Schwarz to
$D_\zeta-\nu
 =\langle\Delta P(\alpha),R_k^{\rm unit}\rangle
  +\langle P_k^{\rm unit},\Delta R(\alpha)\rangle$
gives
\(|D_\zeta-\nu|\le d_\gamma(c)\sqrt{2/\gamma}.\)
Hence, if $d_\gamma(c)\sqrt{2/\gamma}\le1/2$, then $\nu\ge1$ implies
$D_\zeta\ge\nu-1/2\ge\nu/2$.

Similarly, using the full-corrector bound
\(\norm{P^{\rm cor}}+\norm{R^{\rm cor}}\le A_\gamma(c)\sqrt{\nu\mu},\)
and the arc-displacement bound above to
$N_\zeta
 =-\langle\Delta P(\alpha),R^{\rm cor}\rangle
  -\langle P^{\rm cor},\Delta R(\alpha)\rangle$
gives
\(|N_\zeta|\le d_\gamma(c)A_\gamma(c)\mu.\)
Therefore,
\begin{equation*}
 |\zeta|
 \le 2d_\gamma(c)A_\gamma(c)\frac{\mu}{\nu}.
\end{equation*}
Let
${D^{\rm full}}:=D_k^{\rm cor}(\alpha,\rho_0)+\zeta D_k^{\rm unit}$
be the full composite correction, and define its NT components by
$({P^{\rm full}},{R^{\rm full}}):=(P({D^{\rm full}};W_k),R({D^{\rm full}};W_k))$.
Then,
\begin{equation}
 \norm{{P^{\rm full}}}+\norm{{R^{\rm full}}}
 \le B_\gamma(c)\sqrt{\nu\mu},
 \label{eq:fixed-rho-full-combined-bound}
\end{equation}
where
\begin{equation}
 B_\gamma(c)
 :=A_\gamma(c)\left(1+2d_\gamma(c)\sqrt{\frac2\gamma}\right).
 \label{eq:B-gamma-fixed-rho}
\end{equation}

For use in Lemmas~\ref{lem:fixed-rho-damped-positivity} and \ref{lem:fixed-rho-damped-restoration}, define
\begin{align*}
 \bar d_\gamma&:=\sqrt2+\sqrt{2/\gamma},
 &\bar r_\gamma&:=\gamma^{-1/2}+\frac{1+\gamma^{-1}}2,\\
 \overline A_\gamma
 &:=\sqrt2\left(
 \rho_0\sqrt{\gamma^{-1}-1}
 +\frac1{8\gamma}
 +\frac{1+\gamma}{32\gamma^{3/2}}
 \right),
 &\overline B_\gamma
 &:=\overline A_\gamma\left(1+2\bar d_\gamma\sqrt{2/\gamma}\right).
\end{align*}
Moreover, set
\begin{equation}
 b_\gamma
 :=\min\left\{
 1,
 \frac{\sqrt\gamma}{8\overline B_\gamma},
 \frac{\rho_0(1-\gamma)}{64\overline B_\gamma^2}
 \right\}>0,
 \label{eq:b-gamma-app}
\end{equation}
and
\begin{equation}
 c_\gamma
 :=\min\left\{
 \sigma_{\max},\frac12,
 \frac{\sqrt\gamma}{4\bar d_\gamma},
 \frac{\rho_0(1-\gamma)}{384\bar d_\gamma\overline B_\gamma},
 \left(\frac{b_\gamma\rho_0(1-\gamma)}{128\bar r_\gamma}\right)^{1/3},
 \sqrt{\frac{1-\sigma_{\max}}{64\bar r_\gamma}},
 \frac1{2\sqrt{\bar r_\gamma}}
 \right\}>0.
 \label{eq:c-gamma-app}
\end{equation}
These constants depend only on $\gamma$, the fixed value $\rho_0$, and $\sigma_{\max}$, and are independent of $\nu$ and the problem data.
For
$d_\gamma(c)=\sqrt2c+\sqrt{2/\gamma}\,c^2$ and
$r_\gamma(c)=c^3/\sqrt\gamma+c^4/2+c^4/(2\gamma)$
from Lemma~\ref{lem:wide-arc-remainder}, if $0<c\le1/2$, then
$d_\gamma(c)\le\bar d_\gamma c$ and $r_\gamma(c)\le\bar r_\gamma c^3$.
Also, $c^3\le1/8$ and $c^4\le1/16$ imply
$A_\gamma(c)\le\overline A_\gamma$.
Substituting $d_\gamma(c)\le\bar d_\gamma c\le\bar d_\gamma$ into
\(B_\gamma(c)=A_\gamma(c)(1+2d_\gamma(c)\sqrt{2/\gamma})\)
gives $B_\gamma(c)\le\overline B_\gamma$.
The three bounds in \eqref{eq:b-gamma-app} control, respectively, the damping range, strict interiority, and the quadratic correction term in neighborhood recovery.
The seven bounds in \eqref{eq:c-gamma-app} control, in order, the angle limit, the condition $c\le1/2$, $D_\zeta$ and strict interiority, the mixed-correction error, the cubic arc remainder, the projective quantity $\mathcal J$, and small-angle progress.
In particular, this choice gives $r_\gamma(c_\gamma)\le c_\gamma/4$.

\begin{lemma}[Strict interiority under fixed $\rho_0$ and $O(1/\nu)$ damping]
\label{lem:fixed-rho-damped-positivity}
Let $W\in\Nlong(\gamma)$, let
$0<\sigma\le c_\gamma/\nu$, and set $\rho=\rho_0$.
Assume further that
\begin{equation*}
 \frac{b_\gamma}{4\nu}
 \le\beta\le
 \frac{b_\gamma}{\nu}.
\end{equation*}
Then the uncorrected arc and the damped corrected point
$W^+(\beta)=W_k(\alpha)+\beta {D^{\rm full}}$ are strictly interior.
\end{lemma}

\begin{proof}
The current-iterate NT factors of the uncorrected arc are
\(P(\alpha)=Q+\Delta P(\alpha),\qquad R(\alpha)=Q+\Delta R(\alpha).\)
We use the arc-displacement estimate
\(\norm{\Delta P(\alpha)}+\norm{\Delta R(\alpha)}\le d_\gamma(c_\gamma)\sqrt{\frac{\mu}{\nu}}\)
for $\sigma\le c_\gamma/\nu$.
For the full composite correction
${D^{\rm full}}=D_k^{\rm cor}(\alpha,\rho_0)+\zeta D_k^{\rm unit}$,
\eqref{eq:fixed-rho-full-combined-bound} and \eqref{eq:B-gamma-fixed-rho} give
\[
 \norm{{P^{\rm full}}}+\norm{{R^{\rm full}}}
 \le B_\gamma(c_\gamma)\sqrt{\nu\mu}
 \le \overline B_\gamma\sqrt{\nu\mu}.
\]
Therefore, the two NT factors of the damped corrected point
$W^+(\beta)=W_k(\alpha)+\beta {D^{\rm full}}$,
\(P^+(\beta)=P(\alpha)+\beta {P^{\rm full}},\qquad R^+(\beta)=R(\alpha)+\beta {R^{\rm full}},\)
satisfy, using $\beta\le b_\gamma/\nu$,
\begin{align*}
 \norm{P^+(\beta)-Q}
 &\le \norm{\Delta P(\alpha)}+\beta\norm{{P^{\rm full}}}\\
 &\le
 \bigl(d_\gamma(c_\gamma)+b_\gamma\overline B_\gamma\bigr)
 \sqrt{\frac{\mu}{\nu}},\\
 \norm{R^+(\beta)-Q}
 &\le \norm{\Delta R(\alpha)}+\beta\norm{{R^{\rm full}}}\\
 &\le
 \bigl(d_\gamma(c_\gamma)+b_\gamma\overline B_\gamma\bigr)
 \sqrt{\frac{\mu}{\nu}}.
\end{align*}
Since
$d_\gamma(c)\le\bar d_\gamma c$ and
$c_\gamma\le\sqrt\gamma/(4\bar d_\gamma)$, we have
\(d_\gamma(c_\gamma)\le\frac{\sqrt\gamma}{4}.\)
Also,
$b_\gamma\le\sqrt\gamma/(8\overline B_\gamma)$ gives
\(b_\gamma\overline B_\gamma\le\frac{\sqrt\gamma}{8}.\)
Using $\nu\ge1$, each factor perturbation is therefore bounded by
\[
 \bigl(d_\gamma(c_\gamma)+b_\gamma\overline B_\gamma\bigr)
 \sqrt{\frac{\mu}{\nu}}
 \le \frac{3\sqrt\gamma}{8}\sqrt{\frac{\mu}{\nu}}
 \le \frac{3}{8}\sqrt{\gamma\mu}
 <\frac12\sqrt{\gamma\mu}.
\]

Since $W\in\Nlong(\gamma)$, the current complementarity product
$H=Q\circ Q$ satisfies
$\lambda_{\min}(H)\ge\gamma\mu$. Hence, the current-iterate NT factor satisfies
\(\lambda_{\min}(Q)\ge\sqrt{\gamma\mu}.\)
Applying the spectral perturbation inequality
\(\lambda_{\min}(Q+\Delta)\ge\lambda_{\min}(Q)-\norm\Delta\)
to $P^+(\beta)=Q+(P^+(\beta)-Q)$ and
$R^+(\beta)=Q+(R^+(\beta)-Q)$ gives
\(\lambda_{\min}(P^+(\beta)),\ \lambda_{\min}(R^+(\beta))>\frac12\sqrt{\gamma\mu}>0.\)
Thus, $W^+(\beta)$ is strictly interior.
The same argument with $\beta=0$ shows that $P(\alpha)$ and $R(\alpha)$ are positive definite. Hence, the uncorrected arc is also strictly interior.
\end{proof}

Let $H^+(\beta)$ denote the complementarity product of the damped point in the current-iterate NT coordinates.
To prove neighborhood recovery, we decompose $H^+(\beta)$ into a leading centering term and higher-order error terms.
The full-corrector equation gives
\begin{equation}
 H^+(\beta)
 =(1-\beta)H_{\rm arc}(\alpha)
 +\beta(T_0(\alpha)+\zeta\bar I)
 +\beta\bigl(\Delta P(\alpha)\circ {R^{\rm full}}+{P^{\rm full}}\circ\Delta R(\alpha)\bigr)
 +\beta^2{P^{\rm full}}\circ {R^{\rm full}},
 \label{eq:damped-fixed-rho-product-expansion}
\end{equation}
where
\(T_0(\alpha)=(1-\sigma)\bigl((1-\rho_0)H+\rho_0\mu\bar I\bigr).\)
Define
$\widetilde T(\beta):=(1-\sigma)\bigl((1-\beta\rho_0)H+\beta\rho_0\mu\bar I\bigr)$.
This is the leading target after damping reduces the centering strength from $\rho_0$ to $\beta\rho_0$.
Substituting
$H_{\rm arc}(\alpha)=(1-\sigma)H+\mathcal E_{\rm arc}(\alpha)$
from Lemma~\ref{lem:wide-arc-remainder} into \eqref{eq:damped-fixed-rho-product-expansion}, we define the remaining error by
\[
 \mathcal E_{\rm damp}(\beta):=H^+(\beta)-\widetilde T(\beta).
\]
The next lemma shows that this error is smaller than the positive spectral margin of $\widetilde T(\beta)$.

\begin{lemma}[Damped neighborhood recovery with fixed $\rho_0$]
\label{lem:fixed-rho-damped-restoration}
Let $W\in\Nlong(\gamma)$, and suppose that
\[
 0<\sigma\le\frac{c_\gamma}{\nu},
 \qquad
 \rho=\rho_0=\frac18,
 \qquad
 \frac{b_\gamma}{4\nu}\le\beta\le\frac{b_\gamma}{\nu}.
\]
Then
\begin{equation*}
 \norm{\mathcal E_{\rm damp}(\beta)}
 \le
 \left(
 r_\gamma(c_\gamma)
 +3b_\gamma d_\gamma(c_\gamma)\overline B_\gamma
 +\frac14b_\gamma^2\overline B_\gamma^2
 \right)\frac\mu\nu.
\end{equation*}
Moreover,
\begin{equation}
 \left|
 \mu(W^+(\beta))-(1-\sigma)\mu
 \right|
 \le r_\gamma(c_\gamma)\frac\mu{\nu^2}.
 \label{eq:damped-mu-near-target-fixed-rho}
\end{equation}
Therefore, in the NT coordinates defined by the current iterate $W_k$,
\begin{equation*}
 \lambda_{\min}\bigl(\Phi(W^+(\beta);W_k)\bigr)\ge\gamma\mu(W^+(\beta)).
\end{equation*}
In particular, $W^+(\beta)\in\Nlong(\gamma)$.
\end{lemma}

\begin{proof}
The complementarity product after damping is
\[
 H^+(\beta)
 =(1-\beta)H_{\rm arc}(\alpha)
 +\beta(T_0(\alpha)+\zeta\bar I)
 +\beta\bigl(\Delta P(\alpha)\circ {R^{\rm full}}
          +{P^{\rm full}}\circ\Delta R(\alpha)\bigr)
 +\beta^2{P^{\rm full}}\circ {R^{\rm full}},
\]
where
\[
 T_0(\alpha)=(1-\sigma)\bigl((1-\rho_0)H+\rho_0\mu\bar I\bigr),
 \qquad
 H_{\rm arc}(\alpha)=(1-\sigma)H+\mathcal E_{\rm arc}(\alpha).
\]
Hence, using $\widetilde T(\beta)$ defined above, we have
$H^+(\beta)=\widetilde T(\beta)+\mathcal E_{\rm damp}(\beta)$, where
\begin{equation}
 \mathcal E_{\rm damp}(\beta)
 =(1-\beta)\mathcal E_{\rm arc}(\alpha)
 +\beta\zeta\bar I
 +\beta\bigl(\Delta P(\alpha)\circ {R^{\rm full}}
          +{P^{\rm full}}\circ\Delta R(\alpha)\bigr)
 +\beta^2{P^{\rm full}}\circ {R^{\rm full}}.
 \label{eq:damped-error-decomposition}
\end{equation}

We bound the four terms in \eqref{eq:damped-error-decomposition}. The arc remainder, arc displacement, full composite correction, and $\zeta$ satisfy
\begin{align*}
 \norm{\mathcal E_{\rm arc}(\alpha)}
 &\le r_\gamma(c_\gamma)\frac{\mu}{\nu^{3/2}},\\
 \norm{\Delta P(\alpha)}+\norm{\Delta R(\alpha)}
 &\le d_\gamma(c_\gamma)\sqrt{\frac\mu\nu},\\
 \norm{{P^{\rm full}}}+\norm{{R^{\rm full}}}
 &\le B_\gamma(c_\gamma)\sqrt{\nu\mu}
 \le\overline B_\gamma\sqrt{\nu\mu},\\
 |\zeta|
 &\le2d_\gamma(c_\gamma)A_\gamma(c_\gamma)\frac\mu\nu
 \le2d_\gamma(c_\gamma)\overline B_\gamma\frac\mu\nu.
\end{align*}
Also, $\norm{\bar I}=\sqrt\nu$ and $\beta\le b_\gamma/\nu$.
For the mixed Jordan-product term, we use
\[
 \norm{P\circ\widehat R+\widehat P\circ R}
 \le
 \sqrt{\norm{P}^2+\norm{R}^2}\,
 \sqrt{\norm{\widehat P}^2+\norm{\widehat R}^2}
\]
with
$(P,R)=(\Delta P(\alpha),\Delta R(\alpha))$ and
$(\widehat P,\widehat R)=({P^{\rm full}},{R^{\rm full}})$.
For the quadratic term, we use
\[
 \norm{{P^{\rm full}}\circ {R^{\rm full}}}
 \le\norm{{P^{\rm full}}}\norm{{R^{\rm full}}}
 \le\frac14\bigl(\norm{{P^{\rm full}}}+\norm{{R^{\rm full}}}\bigr)^2,
\]
These estimates give the following bounds for the four terms in \eqref{eq:damped-error-decomposition}:
\begin{align*}
 (1-\beta)\norm{\mathcal E_{\rm arc}(\alpha)}
 &\le r_\gamma(c_\gamma)\frac{\mu}{\nu},\\
 \beta|\zeta|\norm{\bar I}
 &\le2b_\gamma d_\gamma(c_\gamma)\overline B_\gamma\frac{\mu}{\nu},\\
 \beta\norm{\Delta P(\alpha)\circ {R^{\rm full}}
                 +{P^{\rm full}}\circ\Delta R(\alpha)}
 &\le b_\gamma d_\gamma(c_\gamma)\overline B_\gamma\frac{\mu}{\nu},\\
 \beta^2\norm{{P^{\rm full}}\circ {R^{\rm full}}}
 &\le\frac14 b_\gamma^2\overline B_\gamma^2\frac{\mu}{\nu}.
\end{align*}
Here we used $\nu^{-3/2}\le\nu^{-1}$ because $\nu\ge1$.
Adding these four inequalities gives
\[
 \norm{\mathcal E_{\rm damp}(\beta)}
 \le\left(
 r_\gamma(c_\gamma)
 +3b_\gamma d_\gamma(c_\gamma)\overline B_\gamma
 +\frac14b_\gamma^2\overline B_\gamma^2
 \right)\frac\mu\nu.
\]

Next, we estimate the average complementarity. Since the residual-invariant full correction ${D^{\rm full}}$ satisfies
$\mu(W_k(\alpha)+{D^{\rm full}})=(1-\sigma)\mu$, damping gives the exact interpolation
\(\mu(W^+(\beta)) =(1-\beta)\mu(W_k(\alpha))+\beta(1-\sigma)\mu.\)
Also,
\[
 H_{\rm arc}(\alpha)=(1-\sigma)H+\mathcal E_{\rm arc}(\alpha),
 \qquad
 \mu(W_k(\alpha))=\frac1\nu\operatorname{Tr}_{\mathcal E}(H_{\rm arc}(\alpha)),
 \qquad
 \operatorname{Tr}_{\mathcal E}(H)=\nu\mu.
\]
Therefore,
\[
 \mu(W^+(\beta))-(1-\sigma)\mu
 =\frac{1-\beta}{\nu}\operatorname{Tr}_{\mathcal E}(\mathcal E_{\rm arc}(\alpha)).
\]
Using
$\operatorname{Tr}_{\mathcal E}(Z)=\langle\bar I,Z\rangle$,
$\norm{\bar I}=\sqrt\nu$, and
$\norm{\mathcal E_{\rm arc}(\alpha)}\le r_\gamma(c_\gamma)\mu/\nu^{3/2}$, we obtain
\[
 \left|\mu(W^+(\beta))-(1-\sigma)\mu\right|
 \le
 \frac{\sqrt\nu}{\nu}
 r_\gamma(c_\gamma)\frac\mu{\nu^{3/2}}
 =r_\gamma(c_\gamma)\frac\mu{\nu^2},
\]
which proves \eqref{eq:damped-mu-near-target-fixed-rho}.

{
Finally, since $W\in\Nlong(\gamma)$ and $\lambda_{\min}(H)\ge\gamma\mu$,
\begin{align*}
 \lambda_{\min}(\widetilde T(\beta))
 &\ge (1-\sigma)\bigl((1-\beta\rho_0)\gamma\mu+\beta\rho_0\mu\bigr)\\
 &=(1-\sigma)\bigl(\gamma+\beta\rho_0(1-\gamma)\bigr)\mu.
\end{align*}
Moreover, \eqref{eq:damped-mu-near-target-fixed-rho} and $\nu\ge1$ give
\[
 \gamma\mu(W^+(\beta))
 \le \gamma(1-\sigma)\mu+\gamma r_\gamma(c_\gamma)\frac{\mu}{\nu}.
\]
Thus, by the spectral perturbation inequality, it is enough to show
\begin{equation*}
 \lambda_{\min}(\widetilde T(\beta))-\gamma(1-\sigma)\mu
 \ge \norm{\mathcal E_{\rm damp}(\beta)}
    +\gamma r_\gamma(c_\gamma)\frac{\mu}{\nu}.
\end{equation*}
For the left-hand side, $\sigma\le c_\gamma/\nu\le1/2$ and
$\beta\ge b_\gamma/(4\nu)$ imply
\[
 \lambda_{\min}(\widetilde T(\beta))-\gamma(1-\sigma)\mu
 \ge \frac{32}{256}b_\gamma\rho_0(1-\gamma)\frac{\mu}{\nu}.
\]
On the other hand, the definitions of $b_\gamma$ and $c_\gamma$, together with
$d_\gamma(c)\le\bar d_\gamma c$ and $r_\gamma(c)\le\bar r_\gamma c^3$, give
\begin{align*}
 \frac14 b_\gamma^2\overline B_\gamma^2
 &\le \frac{b_\gamma\rho_0(1-\gamma)}{256},\\
 3b_\gamma\bar d_\gamma c_\gamma\overline B_\gamma
 &\le \frac{b_\gamma\rho_0(1-\gamma)}{128},\\
 (1+\gamma)\bar r_\gamma c_\gamma^3
 &\le \frac{b_\gamma\rho_0(1-\gamma)}{64}.
\end{align*}
Hence the right-hand side above is at most
\[
 \frac7{256}b_\gamma\rho_0(1-\gamma)\frac{\mu}{\nu}.
\]
Since $7/256<32/256$, the required inequality holds. Therefore,
\[
 \lambda_{\min}(H^+(\beta))
 \ge\lambda_{\min}(\widetilde T(\beta))-\norm{\mathcal E_{\rm damp}(\beta)}
 \ge\gamma\mu(W^+(\beta)),
\]
where $H^+(\beta)=\widetilde T(\beta)+\mathcal E_{\rm damp}(\beta)=\Phi(W^+(\beta);W_k)$.
}

The preceding positivity result also gives strict interiority of $W^+(\beta)$. The estimate above is the current-iterate NT certificate
\[
 \lambda_{\min}\bigl(\Phi(W^+(\beta);W_k)\bigr)
 \ge\gamma\mu(W^+(\beta)).
\]
By Lemma~\ref{lem:current-frame-certificate}, the true complementarity spectrum of the candidate therefore satisfies the wide long-step condition. Hence
$W^+(\beta)\in\Nlong(\gamma)$.
\end{proof}

\subsection{Quantitative Scale Estimates for Projective Normalization}
\label{subsec:appendix-projective-normalization}

This subsection proves the projective scale estimates used in Subsection~\ref{subsec:normalization-solvable-case}.

By \eqref{eq:balanced-residual-scaling},
$D_k^{\rm inv}:=W_k-\omega_kW_0$ is a residual-invariant direction.
Applying Lemma~\ref{lem:ofas-exact-update-identities}(ii) to $D_k^{\rm inv}$ and using
$W_0=(I,0,I,1,1)$, $\nu=n+1$, and $\mu(W_k)=\theta_k\omega_k$, we obtain
\begin{equation}
 \operatorname{tr}X_k+\operatorname{tr}S_k+\tau_k+\kappa_k
 =\nu(\theta_k+\omega_k).
 \label{eq:cross-normalization-longstep}
\end{equation}

Let
$Z^*=(X^*,y^*,S^*)\in{\mathcal Z^*}$ be the point fixed in Subsection~\ref{subsec:normalization-solvable-case}, and let
{$\delta_{\rm rec}(Z^*)$ be defined by~\eqref{eq:recession-margin-main}.}

\begin{lemma}[Positive recession-direction margin]
\label{lem:positive-recession-margin}
Under Assumptions~\ref{ass:surjectivity}--\ref{ass:primal-dual-solvability},
$\delta_{\rm rec}(Z^*)>0$.
The boundedness part of Assumption~\ref{ass:primal-dual-solvability} excludes a nonzero direction along which a primal or dual optimal solution can be extended indefinitely while remaining optimal.
\end{lemma}

\begin{proof}
The set of $(D_X,D_S)$ satisfying
$D_X,D_S\succeq O$ and
$\operatorname{tr}D_X+\operatorname{tr}D_S=1$ is compact.
Moreover, Assumption~\ref{ass:surjectivity} implies that $\Ast$ is injective.
A finite-dimensional injective linear map is bounded away from zero on the unit sphere. Since $D_S$ is bounded on the compact set above, it follows that
$\|\Ast(d_y)+D_S\|\to\infty$ as $\|d_y\|\to\infty$.
Hence the infimum in \eqref{eq:recession-margin-main} is attained.

Suppose, to the contrary, that $\delta_{\rm rec}(Z^*)=0$.
Then an attaining minimizer satisfies
$\A(D_X)=0$, $\Ast(d_y)+D_S=O$,
$S^*\bullet D_X=0$, and $X^*\bullet D_S=0$.
{
If $D_X\ne O$, dual feasibility of $(y^*,S^*)$ gives $C=\Ast(y^*)+S^*$, and therefore
\[
 C\bullet D_X
 =y^{*T}\A(D_X)+S^*\bullet D_X=0.
\]
Together with $\A(D_X)=0$ and $D_X\succeq O$, this shows that the ray $X^*+tD_X$, $t\ge0$, remains primal feasible and satisfies $C\bullet(X^*+tD_X)=C\bullet X^*$. Every point on this ray is therefore primal optimal.
If $D_X=O$, the normalization
$\operatorname{tr}D_X+\operatorname{tr}D_S=1$ implies $D_S\ne O$.
Using primal feasibility $\A(X^*)=b$ and $\Ast(d_y)+D_S=O$, we obtain
\[
 b^Td_y
 =X^*\bullet\Ast(d_y)
 =-X^*\bullet D_S=0.
\]
Consequently, $(y^*+td_y,S^*+tD_S)$ stays dual feasible for all $t\ge0$, while $b^T(y^*+td_y)=b^Ty^*$. Hence the entire ray consists of dual optimal points.
}
Either case contradicts the boundedness condition in Assumption~\ref{ass:primal-dual-solvability}.
\end{proof}

We now prove Lemma~\ref{lem:projective-scale-estimates-main}.
Set $\mathcal T_k^{\rm tr}:=\operatorname{tr}X_k+\operatorname{tr}S_k>0$.
Applying the definition of $\delta_{\rm rec}(Z^*)$ to the normalized direction
$(X_k,y_k,S_k)/\mathcal T_k^{\rm tr}$ gives
\[
 \delta_{\rm rec}(Z^*)\mathcal T_k^{\rm tr}
 \le \norm{\A(X_k)}+\norm{\Ast(y_k)+S_k}
    +S^*\bullet X_k+X^*\bullet S_k.
\]
The HSD residual relations give the following bound for the first two terms:
$(\norm{b}+\norm{C})\tau_k+2\mathcal Q_k$.
Primal-dual optimality also gives
\[
S^*\bullet X_k+X^*\bullet S_k
=C\bullet X_k-b^Ty_k-y^{*T}\Rp(W_k)+X^*\bullet\Rd(W_k).
\]
Using
$C\bullet X_k-b^Ty_k=\Rg(W_k)-\kappa_k$ and $\kappa_k\ge0$, we obtain the upper bound
$(1+\norm{y^*}+\norm{X^*})\mathcal Q_k$.
This proves \eqref{eq:projective-trace-bound-main}.

The same identity gives
$0\le\kappa_k\le(1+\norm{y^*}+\norm{X^*})\mathcal Q_k$.
By \eqref{eq:cross-normalization-longstep} and $\omega_k\ge0$,
\[
\nu\theta_k
\le \mathcal T_k^{\rm tr}+\tau_k+\kappa_k
\le\mathfrak D_*\tau_k+\mathfrak H_*\mathcal Q_k.
\]
This proves the first lower bound in \eqref{eq:projective-scale-bound-main}.
If, in addition,
$\mathfrak H_*\mathcal Q_k\le\nu\theta_k/2$, the second lower bound follows immediately.

\section{Per-Instance Numerical Results}
\label{app:detailed-numerical-results}
\setcounter{table}{0}
\renewcommand{\thetable}{\Alph{section}.\arabic{table}}
\renewcommand{\theHtable}{\Alph{section}.\arabic{table}}

This appendix reports the computation time, the number of outer iterations, and the termination status for each of the 54 problems used in Section~\ref{sec:numerical}. Each cell shows ``computation time (seconds)/number of outer iterations.'' Computation times are displayed to one decimal place, while the unrounded values are used for the aggregated results.

We use MathOptInterface (MOI) termination statuses throughout the experiments. The Clarabel.jl statuses \texttt{SOLVED} and \texttt{ALMOST\_SOLVED} correspond to \texttt{OPTIMAL} and \texttt{ALMOST\_OPTIMAL} in MOI, respectively. See \citet[Sec.~2.6]{GoulartChen2026} for the termination criteria of Clarabel. Under our experimental settings, \texttt{OPTIMAL} requires
\[
 \frac{\kappa}{\tau}\le 1,\qquad
 (\mathrm{gap}_{\rm abs}<10^{-8}\ \text{or}\ \mathrm{gap}_{\rm rel}<10^{-8}),\qquad
 r_{\rm p}<10^{-8},\qquad r_{\rm d}<10^{-8}.
\]
Here $\mathrm{gap}_{\rm abs}$ and $\mathrm{gap}_{\rm rel}$ denote Clarabel.jl's absolute and relative duality-gap measures. Its normalized residual measures for primal and dual feasibility are denoted by $r_{\rm p}$ and $r_{\rm d}$. If these full-accuracy conditions are not satisfied before termination due to the iteration limit, time limit, numerical error, or insufficient progress, the solution is checked again using the relaxed tolerances
\[
 \frac{\kappa}{\tau}\le 1,\qquad
 (\mathrm{gap}_{\rm abs}<5\times10^{-5}\ \text{or}\ \mathrm{gap}_{\rm rel}<5\times10^{-5}),\qquad
 r_{\rm p}<10^{-4},\qquad r_{\rm d}<10^{-4}.
\]
If these conditions hold, the status is \texttt{ALMOST\_OPTIMAL}. Following the numerical evaluation of Clarabel \citep{GoulartChen2026}, we classify \texttt{OPTIMAL} as Full Accuracy and \texttt{ALMOST\_OPTIMAL} as Low Accuracy. The statuses \texttt{TIME\_LIMIT} and \texttt{SLOW\_PROGRESS} correspond to \texttt{MAX\_TIME} and \texttt{INSUFFICIENT\_PROGRESS} in Clarabel.jl, respectively, when the relaxed conditions are also not satisfied.

The label \texttt{SUCCESS\_MIXED\_\allowbreak ACCURACY} is not a termination status of Clarabel.jl or MOI. For a problem--method pair with repeated runs, this label is used when the runs include both \texttt{OPTIMAL} and \texttt{ALMOST\_}\allowbreak\texttt{OPTIMAL}. The success count including Low Accuracy includes \texttt{OPTIMAL}, \texttt{ALMOST\_}\allowbreak\texttt{OPTIMAL}, and \texttt{SUCCESS\_MIXED\_}\allowbreak\texttt{ACCURACY}. The Full Accuracy success count includes only \texttt{OPTIMAL}.

For the finite safeguarded candidate search used in the numerical implementation of \OFAS{} and \CAOFAS{}, we set $\sigma_{\max}=0.95$. The angle backtracking factor was $0.5$, with at most 50 backtracking steps. We used the centering-recovery schedule $\rho\in\{1,1/2,1/4,1/8\}$ and damped the corrector by a factor of $0.5$ down to $\beta_{\min}=1/64$. Warm starts, refinement, recovery, and amplification-activation safeguards were used only in the numerical search; the candidate formulas and acceptance conditions (C1)--(C4) are unchanged.

Table~\ref{tab:factorization-ratios-app} reports linear-algebra ratios of \OFAS{} to Two-factorization for problems solved to Full Accuracy by both methods.

\begin{table}[htbp]
\centering
\caption{Linear-algebra ratios of \OFAS{} to Two-factorization for problems solved to Full Accuracy by both methods}
\label{tab:factorization-ratios-app}
\small
\setlength{\tabcolsep}{5pt}
{
\begin{tabular}{lrrrrr}
\toprule
Problem set & Matched & Factorizations & Factorization time & Linear solves & Solve time \\
\midrule
SDPLIB & 11 & 0.524 & 0.523 & 0.873 & 0.894 \\
NNV-SDP & 26 & 0.601 & 0.606 & 1.001 & 1.053 \\
\bottomrule
\end{tabular}
}

\vspace{2pt}
\parbox{0.97\textwidth}{\footnotesize Each entry is the geometric mean of the \OFAS{}/Two-factorization ratio over problems solved to Full Accuracy by both methods. Ratios below one favor \OFAS{}.}
\end{table}

A cell without a superscript denotes \texttt{OPTIMAL}. We use the superscripts $\mathrm{A}$ for \texttt{ALMOST\_\allowbreak OPTIMAL}, $\mathrm{M}$ for \texttt{SUCCESS\_MIXED\_\allowbreak ACCURACY}, $\mathrm{TL}$ for \texttt{TIME\_\allowbreak LIMIT}, and $\mathrm{SP}$ for \texttt{SLOW\_\allowbreak PROGRESS}. For \texttt{TIME\_\allowbreak LIMIT}, the displayed time may slightly exceed 3600 seconds because an ongoing linear-algebra operation is not interrupted.

Table~\ref{tab:sdplib-per-instance} reports the results for the 18 SDPLIB problems, and Table~\ref{tab:nnvsdp-per-instance} reports those for the 36 NNV-SDP problems. In the NNV-SDP problem names, \texttt{ellipsoid} and \texttt{rectangle} denote the input sets considered in Sections~5.2 and 5.3 of Azuma--Kim--Yamashita \citep{AzumaKimYamashita2026}, respectively. The labels \texttt{dense} and \texttt{identity} describe the structure of the first-layer weight matrix $W^0$, where \texttt{identity} means $W^0=I$. This corresponds to Assumption~5.7 of that paper for rectangular inputs. The label \texttt{nXX} denotes the network width $r=n_0=n_1$, and \texttt{s701}, \texttt{s702}, and \texttt{s703} denote random seeds 701, 702, and 703, respectively. Thus, three seeds are used for each problem family and network width. In Table~\ref{tab:nnvsdp-per-instance}, the problems are ordered by increasing $r$. For the same $r$, they are ordered by problem family, weight-matrix structure, and seed. In both tables, the shortest computation time and the smallest number of outer iterations among the successful methods are shown in bold for each problem.

\begin{table}[htbp]
\centering
\caption{Per-instance results for the 18 SDPLIB problems}
\label{tab:sdplib-per-instance}
\scriptsize
\setlength{\tabcolsep}{2.0pt}
\renewcommand{\arraystretch}{1.08}
\resizebox{\textwidth}{!}{%
\begin{tabular}{lrrrrrr}
\toprule
Problem & \OFAS{} & \CAOFAS{} & \shortstack{Two-\\factorization} & \shortstack{Corrector-\\free} & Mehrotra & Liu-type \\
\midrule
\texttt{control4} & $4.7/27$ & $\mathbf{4.5}/\mathbf{25}$ & $8.7/27$ & $9.2/59$ & $5.0/32$ & $6.0/38$ \\
\texttt{control5} & $\mathbf{18.3}/\mathbf{32}$ & $20.3/35$ & $39.6/37$ & $52.0/100$ & $19.3/36$ & $23.0/43$ \\
\texttt{control6} & $59.1/34$ & $\mathbf{55.2}/\mathbf{31}$ & $109.6/33$ & $125.8/77$ & $63.6/38$ & $70.6/42$ \\
\texttt{gpp100} & $854.4/26^{\mathrm{M}}$ & $\mathbf{634.3}/\mathbf{19}$ & $1296.1/21^{\mathrm{M}}$ & $3621.4/117^{\mathrm{M}}$ & $708.7/21^{\mathrm{M}}$ & $708.9/22^{\mathrm{M}}$ \\
\texttt{gpp124-1} & $3496.0/30^{\mathrm{M}}$ & $\mathbf{2602.9}/22$ & $3701.3/\mathbf{16}^{\mathrm{M}}$ & $3658.8/31^{\mathrm{M}}$ & $3435.8/29$ & $3546.7/30$ \\
\texttt{gpp124-2} & $\mathbf{2595.8}/22^{\mathrm{M}}$ & $2823.3/24$ & $3703.0/\mathbf{16}^{\mathrm{M}}$ & $3702.1/32^{\mathrm{M}}$ & $3361.3/29$ & $3134.5/27$ \\
\texttt{gpp124-3} & $3452.4/29$ & $\mathbf{2486.2}/21$ & $3699.1/\mathbf{16}^{\mathrm{M}}$ & $3654.1/31^{\mathrm{M}}$ & $3317.6/28$ & $3634.8/31$ \\
\texttt{mcp100} & $325.9/9$ & $313.8/9$ & $548.0/\mathbf{8}$ & $\mathbf{310.6}/9$ & $387.4/11$ & $371.8/11$ \\
\texttt{mcp124-1} & $1125.0/\mathbf{9}$ & $\mathbf{1115.9}/\mathbf{9}$ & $2102.9/\mathbf{9}$ & $1330.2/11$ & $1552.0/13$ & $1669.2/14$ \\
\texttt{mcp124-2} & $\mathbf{1114.1}/9$ & $1115.3/9$ & $1887.0/\mathbf{8}$ & $1246.5/10$ & $1442.9/12$ & $1330.1/11$ \\
\texttt{mcp124-3} & $\mathbf{1116.4}/9$ & $1137.4/9$ & $1938.4/\mathbf{8}$ & $1236.5/10$ & $1609.0/13$ & $1492.0/12$ \\
\texttt{mcp124-4} & $1134.9/9$ & $\mathbf{1110.9}/9$ & $1917.6/\mathbf{8}$ & $1241.1/10$ & $1437.3/12$ & $1356.4/11$ \\
\texttt{qap6} & $2.6/30^{\mathrm{M}}$ & $\mathbf{1.9}/\mathbf{21}$ & $5.3/34^{\mathrm{M}}$ & $4.3/56^{\mathrm{M}}$ & $2.3/30^{\mathrm{M}}$ & $13.4/180^{\mathrm{M}}$ \\
\texttt{qap7} & $\mathbf{12.2}/29^{\mathrm{M}}$ & $12.9/30^{\mathrm{M}}$ & $18.6/\mathbf{24}^{\mathrm{M}}$ & $39.9/105^{\mathrm{M}}$ & $14.4/37^{\mathrm{M}}$ & $193.6/515^{\mathrm{M}}$ \\
\texttt{theta1} & $\mathbf{4.5}/\mathbf{10}$ & $\mathbf{4.5}/\mathbf{10}$ & $8.2/\mathbf{10}$ & $5.1/12$ & $5.0/12$ & $5.8/14$ \\
\texttt{theta2} & $382.5/11$ & $366.1/11$ & $661.9/\mathbf{10}$ & $410.8/12$ & $\mathbf{349.4}/\mathbf{10}$ & $393.8/12$ \\
\texttt{truss5} & $4.6/17$ & $4.5/\mathbf{16}^{\mathrm{M}}$ & $7.8/\mathbf{16}$ & $34.2/147$ & $\mathbf{4.4}/18$ & $13.0/56$ \\
\texttt{truss8} & $417.4/20^{\mathrm{M}}$ & $\mathbf{336.6}/\mathbf{16}$ & $679.2/18^{\mathrm{M}}$ & $3407.3/173$ & $413.2/20$ & $396.3/20$ \\
\bottomrule
\end{tabular}%
}
\end{table}
\clearpage

\begin{table}[p]
\centering
\caption{Per-instance results for the 36 NNV-SDP problems (in ascending order of network width $r$)}
\label{tab:nnvsdp-per-instance}
\scriptsize
\setlength{\tabcolsep}{2.0pt}
\renewcommand{\arraystretch}{1.03}
\resizebox{\textwidth}{!}{%
\begin{tabular}{lrrrrrr}
\toprule
Problem & \OFAS{} & \CAOFAS{} & \shortstack{Two-\\factorization} & \shortstack{Corrector-\\free} & Mehrotra & Liu-type \\
\midrule
\texttt{rectangle\_dense\_n32\_s701} & $82.0/37$ & $\mathbf{66.3}/\mathbf{30}$ & $131.4/32$ & $136.0/68$ & $70.6/34$ & $70.6/34$ \\
\texttt{rectangle\_dense\_n32\_s702} & $81.2/37$ & $\mathbf{62.8}/\mathbf{29}$ & $126.7/31$ & $150.0/74$ & $69.7/34$ & $68.8/33$ \\
\texttt{rectangle\_dense\_n32\_s703} & $77.9/36$ & $\mathbf{60.2}/\mathbf{28}$ & $118.4/29$ & $132.5/67$ & $67.1/33$ & $64.2/31$ \\
\addlinespace[2pt]
\texttt{ellipsoid\_dense\_n40\_s701} & $403.1/46$ & $\mathbf{283.0}/\mathbf{31}$ & $666.1/39$ & $774.7/92$ & $350.4/39$ & $337.6/39$ \\
\texttt{ellipsoid\_dense\_n40\_s702} & $391.9/44$ & $\mathbf{293.0}/\mathbf{33}$ & $597.2/35$ & $872.5/104$ & $330.1/38$ & $343.0/40$ \\
\texttt{ellipsoid\_dense\_n40\_s703} & $385.0/44$ & $\mathbf{284.4}/\mathbf{32}$ & $628.3/36$ & $991.0/116$ & $341.0/39$ & $318.8/37$ \\
\texttt{rectangle\_dense\_n40\_s701} & $463.8/51$ & $\mathbf{329.5}/\mathbf{37}$ & $712.6/42$ & $825.7/98$ & $158.9/18^{\mathrm{SP}}$ & $378.5/43$ \\
\texttt{rectangle\_dense\_n40\_s702} & $440.7/48$ & $\mathbf{315.4}/\mathbf{35}$ & $754.7/43$ & $773.0/91$ & $380.7/44$ & $356.2/41$ \\
\texttt{rectangle\_dense\_n40\_s703} & $466.8/51$ & $\mathbf{332.4}/\mathbf{37}$ & $753.8/43$ & $957.0/113$ & $378.1/43$ & $361.6/41$ \\
\texttt{rectangle\_identity\_n40\_s701} & $165.6/18$ & $\mathbf{162.9}/18$ & $268.3/\mathbf{15}$ & $217.7/25$ & $176.4/20$ & $173.5/19$ \\
\texttt{rectangle\_identity\_n40\_s702} & $170.3/18$ & $\mathbf{165.0}/18$ & $264.3/\mathbf{15}$ & $214.2/24$ & $197.4/22$ & $181.2/20$ \\
\texttt{rectangle\_identity\_n40\_s703} & $\mathbf{147.9}/16$ & $\mathbf{147.9}/16$ & $247.2/\mathbf{14}$ & $212.0/24$ & $178.6/20$ & $182.1/20$ \\
\addlinespace[2pt]
\texttt{rectangle\_dense\_n44\_s701} & $690.3/43$ & $\mathbf{550.4}/\mathbf{34}$ & $1160.8/37$ & $1459.5/94$ & $618.7/39$ & $586.5/37$ \\
\texttt{rectangle\_dense\_n44\_s702} & $398.5/24^{\mathrm{SP}}$ & $203.6/12^{\mathrm{SP}}$ & $438.0/14^{\mathrm{SP}}$ & $557.2/35^{\mathrm{SP}}$ & $168.9/10^{\mathrm{SP}}$ & $3205.9/207^{\mathrm{SP}}$ \\
\texttt{rectangle\_dense\_n44\_s703} & $760.2/48$ & $\mathbf{573.3}/\mathbf{36}$ & $1209.4/39$ & $1369.1/89$ & $627.2/40$ & $609.9/39$ \\
\addlinespace[2pt]
\texttt{ellipsoid\_dense\_n48\_s701} & $1053.2/41$ & $\mathbf{772.9}/\mathbf{29}$ & $1655.7/32$ & $2660.2/106$ & $901.9/35$ & $864.6/34$ \\
\texttt{ellipsoid\_dense\_n48\_s702} & $1454.7/54$ & $\mathbf{875.9}/\mathbf{32}^{\mathrm{A}}$ & $2371.3/46$ & $3329.3/130$ & $1196.1/46$ & $1145.4/44$ \\
\texttt{ellipsoid\_dense\_n48\_s703} & $1077.2/39^{\mathrm{A}}$ & $\mathbf{923.6}/\mathbf{33}$ & $1998.9/38$ & $2545.6/99$ & $995.8/38$ & $1046.5/40$ \\
\texttt{rectangle\_dense\_n48\_s701} & $1222.6/46$ & $\mathbf{842.6}/\mathbf{32}$ & $2029.6/39$ & $2495.3/97$ & $1098.9/41$ & $996.1/37$ \\
\texttt{rectangle\_dense\_n48\_s702} & $1519.0/56$ & $\mathbf{1061.3}/\mathbf{39}$ & $2377.7/45$ & $3001.0/117$ & $1227.6/47$ & $1149.5/43$ \\
\texttt{rectangle\_dense\_n48\_s703} & $1217.9/44$ & $\mathbf{844.2}/\mathbf{30}$ & $2003.8/38$ & $2359.7/92$ & $988.8/36$ & $962.0/35$ \\
\texttt{rectangle\_identity\_n48\_s701} & $534.7/20$ & $\mathbf{484.7}/18$ & $821.7/\mathbf{16}$ & $702.9/27$ & $602.4/22$ & $602.7/22$ \\
\texttt{rectangle\_identity\_n48\_s702} & $\mathbf{517.3}/19$ & $530.8/19$ & $844.8/\mathbf{16}$ & $714.4/27$ & $614.4/23$ & $627.9/23$ \\
\texttt{rectangle\_identity\_n48\_s703} & $\mathbf{519.2}/19$ & $521.7/19$ & $830.1/\mathbf{15}$ & $740.7/28$ & $537.8/20$ & $561.1/20$ \\
\addlinespace[2pt]
\texttt{ellipsoid\_dense\_n52\_s701} & $2267.5/54$ & $\mathbf{1686.2}/\mathbf{39}$ & $3643.2/44^{\mathrm{A}}$ & $3608.9/86^{\mathrm{TL}}$ & $2043.7/49$ & $2006.5/48$ \\
\texttt{ellipsoid\_dense\_n52\_s702} & $2099.2/49^{\mathrm{A}}$ & $1482.2/34^{\mathrm{SP}}$ & $3621.8/\mathbf{43}^{\mathrm{A}}$ & $3628.9/88^{\mathrm{TL}}$ & $2002.0/48$ & $\mathbf{1922.9}/45$ \\
\texttt{ellipsoid\_dense\_n52\_s703} & $2018.7/48$ & $\mathbf{1531.1}/\mathbf{36}$ & $3621.6/43$ & $3642.2/89^{\mathrm{TL}}$ & $1885.4/45$ & $1755.5/41$ \\
\addlinespace[2pt]
\texttt{ellipsoid\_dense\_n56\_s701} & $2627.9/39$ & $\mathbf{1975.0}/\mathbf{29}$ & $3651.5/28^{\mathrm{TL}}$ & $3638.1/56^{\mathrm{TL}}$ & $2041.8/31$ & $2048.5/31$ \\
\texttt{ellipsoid\_dense\_n56\_s702} & $3630.4/54^{\mathrm{TL}}$ & $\mathbf{3245.4}/\mathbf{48}$ & $3665.4/28^{\mathrm{TL}}$ & $3660.8/55^{\mathrm{TL}}$ & $3609.5/54^{\mathrm{A}}$ & $3446.7/53$ \\
\texttt{ellipsoid\_dense\_n56\_s703} & $3543.4/54$ & $\mathbf{2436.0}/\mathbf{36}$ & $3614.1/27^{\mathrm{TL}}$ & $3650.8/56^{\mathrm{TL}}$ & $2992.8/45$ & $2818.2/43$ \\
\texttt{rectangle\_identity\_n56\_s701} & $1411.9/21$ & $\mathbf{1316.2}/19$ & $2234.2/\mathbf{17}$ & $1885.8/28$ & $1469.3/22$ & $1565.4/23$ \\
\texttt{rectangle\_identity\_n56\_s702} & $1511.0/22$ & $\mathbf{1383.3}/20$ & $2229.3/\mathbf{17}$ & $2146.7/32$ & $1534.3/23$ & $1628.3/24$ \\
\texttt{rectangle\_identity\_n56\_s703} & $1362.3/20$ & $\mathbf{1271.8}/19$ & $2206.5/\mathbf{17}$ & $1864.3/28$ & $1546.6/23$ & $1546.1/23$ \\
\addlinespace[2pt]
\texttt{rectangle\_identity\_n64\_s701} & $3296.7/22$ & $\mathbf{3101.7}/\mathbf{21}$ & $3850.8/13^{\mathrm{TL}}$ & $3702.4/25^{\mathrm{A}}$ & $3279.2/22$ & $3420.5/23$ \\
\texttt{rectangle\_identity\_n64\_s702} & $3248.5/22$ & $\mathbf{2832.7}/\mathbf{19}$ & $3796.9/13^{\mathrm{TL}}$ & $3644.2/25^{\mathrm{TL}}$ & $3444.3/23$ & $3640.3/25$ \\
\texttt{rectangle\_identity\_n64\_s703} & $3312.1/22$ & $\mathbf{2968.3}/\mathbf{20}$ & $3864.0/13^{\mathrm{TL}}$ & $3714.8/25^{\mathrm{TL}}$ & $3239.8/22$ & $3360.6/23$ \\
\bottomrule
\end{tabular}%
}
\end{table}
\clearpage

\section*{Acknowledgements}
ChatGPT 5.6 was used for language editing and editorial assistance.
The authors reviewed the resulting manuscript and remain responsible for its content.

\section*{Statements and Declarations}

\paragraph{Funding.}
The research of Makoto Yamashita was partially supported by Japan Society for the Promotion of Science (JSPS) KAKENHI Grant Numbers 24K14836 and 26K02868.

\paragraph{Competing interests.}
The authors declare no competing interests related to this work.

\paragraph{Data availability.}
The scripts in the public repository \href{https://github.com/makoto-yamashita/ClarabelHSDSearch.jl}{ClarabelHSDSearch.jl} can be used to obtain or generate the data for the numerical experiments.

\paragraph{Code availability.}
Implementations of the proposed methods and the numerical-experiment code are available in the public repository \href{https://github.com/makoto-yamashita/ClarabelHSDSearch.jl}{ClarabelHSDSearch.jl}.

\bibliographystyle{unsrtnat}
\bibliography{references}

\end{document}